\documentclass[10pt,reqno]{amsart}

\usepackage{tikz}
\usetikzlibrary{patterns}
\usepackage{amsmath}
\usepackage{amssymb}
\usepackage{amsthm}
\usepackage{bbm}
\usepackage{mathabx}
\usepackage{url}
\usepackage{float}
\usepackage{graphicx}
\usepackage{verbatim}
\usepackage{bm}
\usepackage{functan}
\usepackage[inner=3cm,outer=3cm]{geometry}
\usepackage{enumitem}
\usepackage{hyperref}
\hypersetup{
    colorlinks,
    citecolor=black,
    filecolor=black,
    linkcolor=black,
    urlcolor=black
}
\allowdisplaybreaks
\makeatletter
\renewcommand{\@seccntformat}[1]{
  \ifcsname prefix@#1\endcsname
    \csname prefix@#1\endcsname
  \else
    \csname the#1\endcsname\quad
  \fi}
\makeatother
\numberwithin{equation}{section}
\newtheorem{theorem}{Theorem}[section]
\newtheorem{corollary}{Corollary}[theorem]
\newtheorem{lemma}[theorem]{Lemma}

\newtheorem{remark}{Remark}
\newtheorem{claim}{Claim}
\newtheorem{definition}{Definition}

\newtheorem{proposition}[theorem]{Proposition}

\newcommand{\R}{\mathbb{R}}
\newcommand{\Z}{\mathbb{Z}}
\newcommand{\eps}{\epsilon}
\newcommand{\half}{\frac{1}{2}}
\newcommand{\tht}{\theta}
\newcommand{\sg}{\sigma}
\newcommand{\dl}{\delta}
\newcommand{\al}{\alpha}
\newcommand{\la}{\lambda}
\newcommand{\om}{\omega}
\newcommand{\Om}{\Omega}

\newcommand{\f}{\frac}
\newcommand{\F}{\mathcal{F}}

\newcommand{\Ss}{\mathcal{S}}
\newcommand{\Tt}{\mathcal{T}}
\newcommand{\Sp}{\mathbb{S}}
\newcommand{\na}{\nabla}
\newcommand{\dd}{\partial}
\newcommand{\dhalf}{(-\Delta)^{1/2}}

\newcommand{\Phip}{\Pi_{\phi^{\perp}}}
\newcommand{\pip}{\Pi_{\tilde{\phi}_{\perp}}}

\newcommand{\D}{\Delta}

\newcommand{\tk}{2^{-k}}

\newcommand{\ko}{k_1}
\newcommand{\kt}{k_2}
\newcommand{\kth}{k_3}
\newcommand{\sko}{S_{k_1}}

\newcommand{\uko}{u_{k_1}}
\newcommand{\ukt}{u_{k_2}}

\newcommand{\lkmt}{\leq k-10}

\newcommand{\x}{\times}
\newcommand{\dphi}{\dhalf \phi}

\newcommand{\kmt}{k-10}

\newcommand{\ti}{\times}

\newcommand{\LOR}{\langle\Omega\rangle}
\newcommand{\sbr}{\sg^\beta_r}

\newcommand{\vpl}{\varphi_{\lambda}}

\newcommand{\rd}{\sqrt{-\Delta}}
\newcommand{\sgb}{\sg^\beta}
\newcommand{\pgt}{\phi_{>-10}}
\newcommand{\plt}{\phi_{\leq-10}}

\newcommand{\pk}{\phi_k}
\newcommand{\phil}{\phi_l}
\newcommand{\psk}{\phi_{\sim k}}
\newcommand{\plk}{\phi_{<k}}

\newcommand{\dau}{\dd^\al}
\newcommand{\dad}{\dd_\al}

\newcommand{\pj}{\phi_j}
\newcommand{\vp}{\varphi}
\newcommand{\po}{\tilde{\psi}}
\newcommand{\ui}{\Uu^{-1}}
\newcommand{\pkmo}{\phi_{K-1}}
\newcommand{\plko}{\phi_{<K-1}}

\newcommand{\pko}{\phi_{\ko}}
\newcommand{\pkt}{\phi_{\kt}}
\newcommand{\pkth}{\phi_{\kth}}
\newcommand{\Lc}{\mathcal{L}}
\newcommand{\gam}{\gamma}
\newcommand{\hx}{\hat{x}}
\newcommand{\hxi}{\hat{\xi}}
\newcommand{\spt}{\text{supp}}
\newcommand{\pso}{\phi_{\sim0}}

\newcommand{\japa}{\langle a \rangle}
\newcommand{\japb}{\langle b \rangle}

\newcommand{\du}{\dhalf \phi}

\newcommand{\puj}{\phi^{(j)}}
\newcommand{\pujmo}{\phi^{(j-1)}}
\newcommand{\puo}{\phi^{(1)}}
\newcommand{\puoj}{\phi^{(1)}_j}
\newcommand{\puojmo}{\phi^{(1)}_{j-1}}
\newcommand{\pji}{\phi^{(J,i)}}

\newcommand{\puJ}{\phi^{(J)}}
\newcommand{\puJmo}{\phi^{(J-1)}}
\newcommand{\pl}{\phi^L}

\newcommand{\bl}{L}
\newcommand{\psil}{\psi^L}

\newcommand{\lnp}{L_n\phi}

\newcommand{\Uu}{\mathbf{U}}

\newcommand{\Aa}{\mathbf{A}}
\newcommand{\plkmt}{\phi_{\leq k-10}}
\newcommand{\pgkmt}{\phi_{>k-10}}

\newcommand{\Psk}{P_{\sim k}}

\newcommand{\vpkt}{\vp_{\kt}}
\newcommand{\fkth}{F_{\kth}}
\newcommand{\xst}{X^{s,\tht}}
\newcommand{\xstmo}{X^{s-1,\tht-1}}

\newcommand{\vptk}{\vp^{(2)}_{\kt}}
\newcommand{\vpthk}{\vp^{(3)}_{\kth}}

\newcommand{\Sko}{S_{\ko}}
\newcommand{\tp}{\tilde{\phi}}
\newcommand{\Pp}{\mathcal{P}}
\newcommand{\omij}{\Om_{ij}}
\newcommand{\pho}{\phi^{(1)}}
\newcommand{\pht}{\phi^{(2)}}
\newcommand{\phth}{\phi^{(3)}}
\newcommand{\tpim}{2+,\infty-}
\newcommand{\imtp}{\infty-,2+}
\newcommand{\llh}{\text{(low)}\na\text{(low)}\na\text{(high)}}
\newcommand{\lllow}{\text{(low)}\na\text{(low)}\na\text{(low)}}
\newcommand{\lhh}{\text{(low)}\na\text{(high)}\na\text{(high)}}
\newcommand{\hhh}{\text{(high)}\na\text{(high)}\na\text{(high)}}
\newcommand{\llowesth}{\text{(low)}\na\text{(lowest)}\na\text{(high)}}
\newcommand{\lowestlh}{\text{(lowest)}\na\text{(low)}\na\text{(high)}}
\newcommand{\hll}{\text{(high)}\na\text{(low)}\na\text{(low)}}
\newcommand{\hhl}{\text{(high)}\na\text{(high)}\na\text{(low)}}
\newcommand{\fpq}{\f{1}{p}+\f{3}{q}}
\newcommand{\mfpq}{-\f{1}{p}-\f{3}{q}}

\newcommand{\tpxst}{\|\tp\|_{\xst_1}}
\newcommand{\dtpt}{\dd_t\tp}
\newcommand{\pkqj}{P_kQ_j}
\newcommand{\wip}{\widetilde{\Pi}_{\tilde{\phi}^{\perp}}}
\newcommand{\Pllk}{P_{\ll k}}
\newcommand{\Plsk}{P_{\lesssim k}}
\newcommand{\Pgsk}{P_{\gtrsim k}}
\newcommand{\Pggk}{P_{\gg k}}
\newcommand{\Qsj}{Q_{\sim j}}
\newcommand{\Qllj}{Q_{\ll j}}

\newcommand{\Qgsj}{Q_{\gtrsim j}}
\newcommand{\Qlsj}{Q_{\lesssim j}}

\title{Global Solutions to the 3D Half-Wave Maps Equation with Angular Regularity}
\author{Katie Marsden}

\begin{document}

\maketitle

\begin{abstract}
    The half-wave maps equation is a nonlocal geometric equation arising in the continuum dynamics of Haldane-Shashtry and Calogero-Moser spin systems. In high dimensions $n\geq4$, global wellposedness for data which is small in the critical Besov space $\dot{B}^{n/2}_{2,1}$ is known since \cite{KS,KK}. There is a major obstruction in extending these results to three dimensions due to the loss of the crucial $L^2_tL^\infty_x$ Strichartz estimate. In this work, we make progress on this case by proving that the equation is ``weakly" globally well-posed (in the sense of \cite{Tao_1}) for initial data which is not only small in $\dot{B}^{3/2}_{2,1}$ but also possesses some angular regularity and weighted decay of derivatives. We use Sterbenz's improved Strichartz estimates \cite{Sterbenz} in conjunction with certain commuting vector fields to develop trilinear estimates in weighted Strichartz spaces which avoid the use of the $L^2_tL^\infty_x$ endpoint.
\end{abstract}
\setcounter{tocdepth}{2}
\tableofcontents
\section{Introduction}
This paper concerns the global existence of solutions to the half-wave maps equation
\begin{equation}\label{eqn1.1}
\begin{cases}
\dd_t\phi=\phi\times\dhalf\phi\\
\phi(0,\cdot)=\phi_0
\end{cases}
\end{equation}
in the critical Besov space $\dot{B}^{3/2}_{2,1}$ in dimension 3, partially generalising known results in higher dimensions (see Theorem \ref{thm1.1} for the precise result). In \eqref{eqn1.1}, $\phi$ is a function on $\R\x\R^3$ taking values in the sphere and the operator $\dhalf$ is defined via its action in Fourier space:\footnote{The reader is directed to Section \ref{notation_section} for standard notation used in this introduction.}
$$
\F(\dphi)(\xi)=|\xi|\F(\phi)(\xi).
$$
The space $\dot{B}^{3/2}_{2,1}$ is critical in the sense that it is invariant with respect to the rescaling  
$$
\phi(t,x)\mapsto \phi_\la(t,x):=\phi(\la t,\la x)
$$
which is a symmetry of the equation.

As the name suggests, the half-wave maps equation has close links to the well-studied wave-maps equation (see \eqref{eqn1.2}), however one notable difference is that the conserved energy
\begin{equation*}
    E(t):=\int_{\R^n}|(-\D)^{1/4}\phi|^2dx
\end{equation*}
(posed for $n$ spatial dimensions) scales like $\dot{H}^{1/2}$, as opposed to $\dot{H}^1$ for the wave maps equation.

In high dimensions ($n>1$), the theory of half-wave maps is currently limited to the questions of wellposedness \cite{KS,KK,yang} and uniqueness \cite{eyeson2022uniqueness}. Krieger and Sire \cite{KS} showed via the reformulation as a semilinear wave equation
\begin{align}
(\dd_t^2-\D)\phi&=-\phi \text{ } \dau\phi\cdot\dad\phi+\Phip[(\dphi)(\phi\cdot\dphi)]\nonumber\\
&\quad\quad+\phi\x[\dhalf(\phi\x\dphi)-(\phi\x(-\D)\phi)]\label{eqn1.2}
\end{align}
that \eqref{eqn1.1} is globally well-posed for small initial data in $\dot{B}^{n/2}_{2,1}\x\dot{B}^{n/2-1}_{2,1}$ for $n\geq5$. This result was extended to Sobolev data in \cite{yang}. In \eqref{eqn1.2}, $\Phip$ denotes the projection onto the orthogonal complement of $\phi$ and we sum over $\al=0,\ldots,3$ with respect to the Minkowski metric. 

The first term in the forcing above corresponds to that of the wave maps equation. Observe that the second and third terms are heuristically of the same form
$$
\phi\,\na\phi\na\phi
$$
if one can account for cancellations of the nonlocal derivatives and the action of the nonlinear projection operator. There is an important distinction to be made, however, in that the wave maps source terms are based on the null form
\begin{equation}\label{ns}
\dad\phi^T\dau\phi=\half(\Box(\phi^T\phi)-2\phi^T\Box\phi)
\end{equation}
which is not present in the additional terms. Krieger and Sire handled the wave maps source terms using dyadic $X^{s,\tht}$-type spaces which have played a central role in the analysis of the wave maps equation (see for example \cite{KlSe,klainerman1996estimates} for subcritical results and \cite{tataru1} in the critical case). The argument relies on the null structure so does not apply to the new half-wave maps terms, however it turns out that these terms have enough geometric structure (see \eqref{geometry}) to be handled independently in the Strichartz spaces $L^p_tL^q_x$, where 
$$
\f{2}{p}+\f{n-1}{q}\leq\f{n-1}{2}, \quad\quad\quad n\geq2, \ (n,p,q)\neq(3,2,\infty)
$$
This range becomes increasingly restrictive in lower dimensions, and in dimension $4$ we lose the $L^2_tL^4_x$ Stichartz estimate which was frequently required in \cite{KS}. This was overcome by Krieger and Kiesenhofer in \cite{KK}, fully generalising the higher dimensional results. In three dimensions, the range of available estimates becomes smaller still, and in particular we lose the endpoint space $L^2_tL^\infty_x$ \cite{tao1998counterexamples}. This space plays an essential role in the arguments of \cite{KS,KK,yang} and hence the methods used there do not obviously generalise to $n=3$. However, if we consider only radial initial data, we may appeal to the wider range of \textit{radially admissible} Strichartz spaces, for which
$$
\f{2}{p}+\f{n-1}{2q}<\f{n-1}{4}
$$
The methods of \cite{KK} can then be straightforwardly applied. The standard and radial Strichartz pairs $(p,q)$ in three dimensions are displayed in Figure \ref{fig1}.
\begin{figure}[h]\label{fig1}
\begin{center}
\begin{tikzpicture}[scale=1.6]
    \draw[->] (-0.2,0) -- (2.2,0) node[right] {$\f{1}{p}$};
    \draw[->] (0,-0.2) -- (0,2.2) node[above] {$\f{1}{q}$};    

    \coordinate (A) at (0,0);
    \coordinate (B) at (0,2);
    \coordinate (C) at (2,0);
    \coordinate (D) at (1,2);
        
    \node[left] at (B) {$\half$};
    \node[below] at (C) {$\half$};
    \node[above right] at (D) {$(\f{1}{4},\half)$};

    \fill[gray] (A) -- (B) -- (C) -- cycle;
    \fill[gray!30] (B) -- (D) -- (C) -- cycle;

    \draw (B) -- (C);
    \draw (B) -- (C);
    \draw (B) -- (D);
    \draw[dashed, dash pattern=on 2.5pt off 2pt] (1,2) -- (2,0);
    \draw (A) -- (B);
    \draw (A) -- (C);

    \fill[gray!30] (B) circle (1.5pt);
\end{tikzpicture}
\end{center}
\caption{The admissible Strichartz pairs $(p,q)$ in $n=3$. The dark gray region depicts the standard admissible pairs, and the light gray region the extended range of radially admissible pairs. The endpoint $L^2_tL^\infty_x$ is radially but not standard admissible.}
\end{figure}
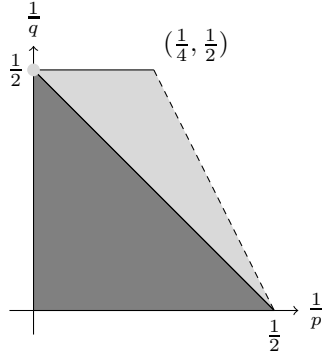

Under the weaker assumption that the data is not radial but merely has some angular regularity, Sterbenz \cite{Sterbenz} also proved Strichartz estimates in the radially admissible spaces \textit{excluding} the $L^2_tL^\infty_x$ endpoint (in fact the whole line $q=2$). The goal of this paper is to exploit these improved estimates to obtain the following ``weak" small data-global wellposedness result for angularly regular initial data. We introduce the notation
\begin{equation}
\|\LOR u\|\equiv\|u\|+\max_{i,j}\|\omij u\|\label{LOR_notn}
\end{equation}
for any norm $\|\cdot\|$ and the angular derivatives $\omij$, see \eqref{com_fields}. Here and throughout, $\|\LOR (x\cdot\na)\phi[0]\|$ is taken to mean $\max_{k,l=1,2,3}\|\LOR (x_k\na_l\phi[0])\|$. 
\begin{theorem}\label{thm1.1}
Let $\phi_0:\R^3\rightarrow \Sp^2$ be a smooth initial datum which is constant outside a compact set. 
 There exists $0<\eps<1$ such that whenever
\begin{equation}\label{small0}
\|\LOR \phi_0\|_{\dot{B}^{3/2}_{2,1}}+\|\LOR (x\cdot\na)\phi_0\|_{\dot{B}^{3/2}_{2,1}}<\eps
\end{equation}
the problem \eqref{eqn1.1} admits a global smooth solution. Moreover for any $s$ sufficiently close to $3/2$ it holds
\begin{equation}\label{global_bounds}
\|\phi(t)\|_{\dot{B}^s_{2,1}}\lesssim_s\|\LOR \phi_0\|_{\dot{B}^s_{2,1}}+\|\LOR (x\cdot\na)\phi_0\|_{\dot{B}^s_{2,1}}
\end{equation}
for all $t\in\R$.
\end{theorem}
The slightly unusual assumptions on the initial data come from our use of commuting vector fields, to be discussed shortly.

We now give some more details on the difficulties encountered in the proof of Theorem \ref{thm1.1}. Due to the loss of the endpoint Strichartz estimate, the low dimensional analysis of the wave maps equation becomes increasingly reliant on the null structure, and it turns out that in three dimensions the geometric structure of the additional half-wave maps terms is insufficient to compensate this in the Strichartz spaces. Indeed the iteration argument of Tataru \cite{tataru} for proving small data-global wellposedness of the wave maps equation in the $n=3$ critical Besov space involves the development of highly tailored function spaces.
We therefore turn to Tao's approach for studying wave maps in the critical \textit{Sobolev} space \cite{Tao_1} which works entirely in the framework of Strichartz spaces and does not rely heavily on the null structure, at the cost that we can only obtain a ``weak" wellposedness result as in Theorem \ref{thm1.1}. We remark that Tao's method was previously used in \cite{yang} to show small data-global wellposedness of the half-wave maps equation in the critical Sobolev space in high dimensions.

The argument of \cite{Tao_1} relies on a carefully chosen coordinate transformation which cancels out the most difficult frequency interactions of the nonlinearity. These are certain terms in which one of the differentiated factors appears at low frequency, but not as low as the non-differentiated factor, which we write as
$$
\lowestlh
$$
Admitting this cancellation, the principal difficulty of the present work is dealing with interactions of the form
\begin{equation}\label{warning}
\llowesth,
\end{equation}
In \cite{Tao_1} such interactions are controlled by placing the terms into $L^2_tL^\infty_x$, $L^2_tL^\infty_x$ and $L^\infty_tL^2_x$ respectively, with no flexibility in the estimate. As already  mentioned, the space $L^2_tL^\infty_x$ is no longer available to us. To overcome this we incorporate into our function spaces a range of commuting vector fields, 
\begin{align}\label{com_fields}
L_n:=x_n\dd_t+t\dd_{x_n}\quad(n=1,2,3)&&\quad\text{and}\quad&&\Om_{ij}:=x_i\dd_{x_j}-x_j\dd_{x_i}\quad(i,j=1,2,3).
\end{align}
These were first introduced in the context of global regularity for nonlinear wave equations in \cite{klainerman}. By incorporating them into the Strichartz norms, we are able to develop spacetime estimates for terms of the form \eqref{warning}, gaining decay in time via the Lorentz boosts and in space via the heuristic
\begin{equation}\label{obs}
    \phi(x)\simeq\f{1}{\omij}\omij\phi(x)\simeq\f{1}{x_i\xi_j-x_j\xi_i}\omij\phi(x)\simeq\f{1}{|x_{ij}||\xi_{ij}|\sin(\angle(x_{ij},\xi_{ij}))}\omij\phi(x)
\end{equation}
Here $x$, $\xi$ denote the physical and Fourier variables respectively, and $x_{ij}$, $\xi_{ij}$ their projections onto the $i-j$ plane. Assuming $\phi$ has angular regularity and can absorb the derivative $\omij$, we therefore gain decay in $x$ when the the Fourier and physical variables have some angular separation (see Lemma \ref{angular_sep_lem}). In practice, we implement this via a simultaneous decomposition of the trilinear term \eqref{warning} on angular caps in physical and Fourier space. See Lemma \ref{SIL} for the detailed argument.

We remark that it is the use of commuting vector fields which limits our result to Besov rather than Sobolev spaces. The issue is that we occasionally need bounds such as
$$
\|\Om_{ij}\phi\|_{L^\infty_tL^\infty_x},\|L_n\phi\|_{L^\infty_tL^\infty_x}\lesssim1,
$$
which in the absence of the commuting vector fields would come for free from the fact that the solution lies on the sphere.

It is appropriate to note here the paper \cite{SterbenzYM} of Sterbenz regarding global regularity of the $(4+1)$-dimensional Yang-Mills equation in Lorentz gauge. This article also uses an angular regularity assumption to exploit the improved estimates of \cite{Sterbenz}, however the argument there is based on measuring angular concentration phenomena and avoids the use of the Lorentz boosts in order to recover estimates in $X^{s,\tht}$-based spaces. See also \cite{hirayama2024sharp} for results on the quadratic nonlinear wave equation with angular regularity and \cite{hong2022improved} in the context of the nonlinear Schr\"odinger equation.

\bigskip

We now give a brief overview of the argument and the structure of this paper. We will prove the following small data-global existence result for the differentiated equation \eqref{eqn1.2}. Denote $\phi[t]\equiv(\phi(t),\dd_t\phi(t))$.
\begin{theorem}\label{thm1.2}
Let $\phi[0]:=(\phi_0,\phi_1):\R^3\rightarrow \Sp^2\x T\Sp^2$ be a smooth initial data pair which is constant outside a compact set. There exists $0<\eps<1$ such that whenever
\begin{equation}\label{small2}
\|\LOR \phi[0]\|_{\dot{B}^{3/2}_{2,1}\x\dot{B}^{1/2}_{2,1}}+\|\LOR (x\cdot\na)\phi[0]\|_{\dot{B}^{3/2}_{2,1}\x\dot{B}^{1/2}_{2,1}}<\eps
\end{equation}
the equation \eqref{eqn1.2} with data $\phi[0]$ admits a global smooth solution $\phi[t]$ with
\begin{equation}\label{global_bounds2}
\|\phi[t]\|_{\dot{B}^s_{2,1}\x \dot{B}^{s-1}_{2,1}}\lesssim_s\|\LOR \phi[0]\|_{\dot{B}^s_{2,1}\x\dot{B}^{s-1}_{2,1}}+\|\LOR (x\cdot\na)\phi[0]\|_{\dot{B}^s_{2,1}\x\dot{B}^{s-1}_{2,1}}
\end{equation}
If moreover $\phi_1=\phi_0\x\dphi_0$, the global solution solves the half-wave maps equation \eqref{eqn1.1}.
\end{theorem}
Note that in the case $\phi_1=\phi_0\x\dphi_0$, the smallness assumption on $\phi_1$ in \eqref{small2} is inherited from that on $\phi_0$.

The starting point for our proof is the following local existence result, which reduces the proof of Theorem \ref{thm1.2} to finding uniform bounds on the solution in subcritical Besov spaces. Following Tao's method of frequency envelopes, we will show in Section \ref{reduction_chap} that this in turn comes down to establishing a priori estimates for the solution in a space $S$, defined in Section \ref{function_spaces_section}.
\begin{theorem}\label{LWP_thm}
There exists $\nu>0$ such that for any $3/2<s<3/2+\nu$ the following holds. Let $\phi[0]\in B^s_{2,1}\x B^{s-1}_{2,1}$ be a smooth initial data taking values in $\Sp^2\x T\Sp^2$, equal to a constant $p$ outside a compact set.\footnote{Since $\phi$ lies on $\Sp^2$, when we say e.g. $\phi\in B^s_{2,1}$ we really mean that $\phi-p\in B^s_{2,1}$ for $p$ the fixed limit of $\phi$ at infinity.}$^{,}$\footnote{This assumption is far stronger than necessary, and not preserved under the flow (since the equation is nonlocal). A better assumption for this paper is that $\phi[0]$, $L_n\phi[0]$ and $\Om_{ij}\phi[0]$ lie in $B^{s'}_{2,1}\x B^{s'-1}_{2,1}$ for every $s'\geq1$. This property is preserved by the flow (as can be seen by a persistence-of-regularity type argument) and thus leads to a blow-up criterion.} Suppose further that
$$
\|\LOR \phi[0]\|_{\dot{B}^{3/2}_{2,1}\x\dot{B}^{1/2}_{2,1}}+\|\LOR (x\cdot\na)\phi[0]\|_{\dot{B}^{3/2}_{2,1}\x\dot{B}^{1/2}_{2,1}}<\eps
$$
for some $\eps$ sufficiently small. Then there exists $T>0$ depending only on $\|\phi[0]\|_{B^s_{2,1}\x B^{s-1}_{2,1}}$ and a smooth solution $\phi\in C([0,T],B^s_{2,1})\cap C^1([0,T],B^{s-1}_{2,1})$ to \eqref{eqn1.2}. Moreover, $\phi(t)\in\Sp^2$ for all $t\in[0,T]$.

If we further have $\phi_1=\phi_0\x\dhalf\phi_0$, this solution solves the half-wave maps equation \eqref{eqn1.1} on its maximal interval of existence.
\end{theorem}
The proof of this local result is postponed to Section \ref{LWP_proof} so as not to distract from the main ideas of the paper. Note the unusual assumption of smallness in a critical norm. This restriction appears somewhat artificial since it is only needed to keep the Picard iterates away from the origin in order to control the projection operator $\Pi_{\phi^{\perp}}$, which is a feature only of the differentiated equation.

Returning to the main argument, we first consider the wave maps contribution to the nonlinearity and in Section \ref{discarding} discard the frequency interactions in which the non-differentiated factor $\phi$ appears at high frequency. These can be dealt with via standard Strichartz estimates. In Section \ref{normal_forms_chapter} we discard the $\lhh$ interactions as well as a certain error term via normal form transformations exploiting the null structure.

This leaves only the $\llh$ interactions in the wave maps source terms.\footnote{We are considering the whole term localised to unit frequency, so $\lllow$ interactions are impossible.} These are handled in Sections \ref{approx_PT_chapter} and \ref{cancellation_chapter} using Tao's gauge transformation and the arguments involving commuting vector fields already discussed, depending on precisely which factor, $\phi$ or $\dad\phi$, appears at lowest frequency. We remark that we are able to slightly simplify the gauge transformation from \cite{Tao_1} due to our working in Besov spaces. In particular, we do not need to antisymmetrise the equation in order to obtain almost-orthogonality of the transformation matrix, which is automatically a perturbation of the identity. 

It remains to discuss how to control the additional nonlocal terms appearing in the half-wave maps equation. This is the content of Section \ref{HWM_chapter}. The main difference from the wave maps terms arises in studying interactions which are (morally speaking) of type
\begin{equation}\label{warning2}
\lhh \quad\quad\text{ or }\quad\quad \lowestlh
\end{equation}
The analogous wave maps source terms were discarded by the normal form and gauge transformations respectively, both of which relied on the structure of the nonlinearity so can no longer be applied. To compensate this we use that the remaining terms of \eqref{eqn1.2} involve interactions which are loosely speaking of the form
\begin{equation}\label{geometry}
\phi\cdot\na\phi,
\end{equation}
which vanishes for functions on the sphere. As in \cite{KS}, we exploit this cancellation via the following identity which allows us to flip the low frequency factors in \eqref{warning2} to high frequency, and thus appeal to Strichartz-based methods:
\begin{equation}\label{geometric_identity_4}
P_k(\phi_{<k-10}\cdot\phi_{\geq k-10})=-\half P_k(\phi_{\geq k-10}\cdot\phi_{\geq k-10})\tag{GeId}
\end{equation}
This is a straightforward consequence of the property $P_k(\phi\cdot\phi)=P_k(1)=0$.

Besides this, the half-wave maps terms present various technical complications due to the nonlocal nature of the operator $\dhalf$. This is a particular issue when working with the commuting vector fields which are non-translation invariant.

\bigskip

We conclude this introduction with a brief discussion of the physical background and low dimensional results pertaining to equation \eqref{eqn1.1}.
The half-wave maps equation was first derived as the classical limit of a continuous Haldane-Shashtry quantum spin chain in \cite{CM_derivation2} (see also \cite{LG}) and was further shown in \cite{lenzmannCM} to arise in the continuum limit of the completely integrable classical Calegero-Moser spin systems. Unsurprisingly therefore, the one-dimensional equation possesses a complete integrability structure by way of a Lax Pair \cite{LG}, and in this setting there has been significant interest in special solutions of the equation. Indeed soliton solutions were first studied numerically and analytically in \cite{CM_derivation2}, with further investigations of multi-solitons in \cite{berntson2020multi,matsuno2022integrability} and a complete classification of the finite-energy travelling solitary waves in \cite{lenzmann}. The equation is energy critical in one dimension, and the problem of wellposedness remains open. Results in this direction include \cite{yang2} which showed the global existence of large data-weak solutions in $\dot{H}^1\cap\dot{H}^{1/2}$, and the more recent work \cite{ohlmann2023analytic} which establishes the global existence of a particular family of rational solutions.

The equation into manifolds other than the sphere has been studied in dimension one in \cite{LG} and in higher dimensions in \cite{yang}, both in the context of hyperbolic space. For further background on the half-wave maps equation, see Lenzmann's primer \cite{primer}.
 
\bigskip
\textbf{Acknowledgements:} The author wishes to thank her doctoral advisor Joachim Krieger for many helpful discussions and suggestions, in particular with regards to the observation and manipulation of \eqref{obs}.

\subsection{Notation}\label{notation_section}
We write $X\lesssim_\al Y$ to mean $X\leq C(\al)\,Y$ for a constant $C$ depending only on the parameters $\al$. $X\sim_\al Y$ means that $X\lesssim_\al Y$ and $Y\lesssim_\al X$. We also use the notation $\F(f)(\xi)$ or sometimes $\hat{f}(\xi)$ to denote the Fourier transform of a function $f$.

To define the Littlewood-Paley multipliers we introduce $\chi:\R^3\rightarrow\R_{+}$ be a smooth function supported in $\{|\xi|\leq 2\}$ equal to $1$ on $\{|\xi|\leq1\}$ and set
$$
\chi_k(\xi):=\chi(2^{-k}\xi)-\chi(2^{-(k-1)}\xi)
$$ 
so that
$$
\sum_{k\in\Z}\chi_k(\xi)=1
$$
for all $\xi\neq0$. We denote
\begin{equation}\label{chi_deff}
P_k\phi\equiv \phi_k:=\F^{-1}(\chi_k(\xi)\hat{\phi}(\xi))
\end{equation}
We will also write $\tilde{\chi}_k(\xi)$ for $\sum_{j=k-C}^{k+C}\chi_j(\xi)$ and $\tilde{P}_k\phi$ or $\phi_{\sim k}$ for $\F^{-1}(\tilde{\chi}_k(\xi)\hat{\phi}(\xi))$ when $C$ is any fixed constant up to $100$. Likewise $\phi_{<k}:=\sum_{j<k}\pj$ and so on.

To reduce notation, we will often abusively write expressions such as $j\ll k$ to mean $j\leq k-C$, of course this really means $2^j\ll 2^k$, $2^j\sim 2^k$, etc..

Our argument is based in the homogeneous Besov spaces with norm
$$
\|\phi\|_{\dot{B}^{s}_{2,1}}:=\sum_{k\in\Z}2^{sk}\|\phi_k\|_{L^2_x}
$$
or in the subcritical case (at the end of the paper) the inhomogeneous spaces
$$
\|\phi\|_{B^{s}_{2,1}}:=\sum_{k>0}2^{sk}\|\phi_k\|_{L^2_x}+\|P_{\leq0}\phi\|_{L^2_x}
$$
In addition to the usual Littlewood-Paley cut-offs we will also need dyadic cut-offs in physical space which we denote $\vp_\la(x)$. Here $\vp_\la(x)\equiv\chi_\la(x)$ (see \eqref{chi_deff}) but we adopt a different notation in order to emphasise that the cut-offs are acting in different spaces. We will also use  notation such as $\vp_{\geq\la}:=\sum_{\la'\geq\la}\vp_{\la'}$, and slightly abusively denote $\vp_{\la}(t)$ for the analogous cut-offs in the time variable.

We will also frequently use $\|\phi\|_{p,q}$ to mean $\|\phi\|_{L^p_tL^q_x}$, with similar variants in the $x$ and $t$ variables, and throughout $M$ should always be interpreted as a very large constant.

\section{Preliminaries}\label{preliminaries_section}
\subsection{Angular Derivatives and Commuting Vector Fields}\label{ang_com_chap}
In our argument the Lorentz boosts, $L_n$, and the angular derivative operators, $\Om_{ij}$, defined in \eqref{com_fields} will play a key role. Observe that these operators obey the Leibniz rule. We will also need the Riesz transforms $R_n$ defined by $\F(R_n\phi)(\xi)=\f{\xi_n}{|\xi|}\F(\phi)(\xi)$ ($n=1,2,3$), which we recall are bounded on $L^p_x$ for $1<p<\infty$.

One may readily verify that the operators $L_n$ and $\Om_{ij}$ commute with the wave operator $\Box$, and satisfy the relations
\begin{align}
[L,\dd]=\dd, && [L,\dhalf]=R \dd_t , &&[L,\Om]=L \label{commr1}
\end{align}
and
\begin{align}
[\Om,R]=R, && [\Om,\dd]=\dd
\end{align} 
Here $L$, $R$, $\Om$, and $\dd$ denote linear combinations of the identity with the operators $(L_n)_{n=1,2,3}$, $(R_n)_{n=1,2,3}$, $(\Om_{ij})_{i,j=1,2,3}$, $(\dd_\al)_{\al=0,1,2,3}$ respectively. Note that $\Om$ commutes with any radial Fourier multiplier such as $\dhalf$, thanks to the property $\F(\Om_{ij}\phi)=\Om_{ij}\F(\phi)$.

Unfortunately, there is a non-trivial commutation relation between the $L_n$ and the Littlewood-Paley operators $P_j$, which will be a source of some irritation throughout the paper. Precisely, let $\mathcal{P}_j$ denote a generic operator corresponding to a (not necessarily radial) smooth multiplier $\chi^{(\mathcal{P})}(2^{-j}\xi)$, with $\text{supp}\chi^{(\mathcal{P})}\subset\text{supp}\chi$. It holds
\begin{align}
 [L_n,\mathcal{P}_j]=2^{-j}\dd_t\mathcal{P}_j &&\text{ and } &&[\Om,\mathcal{P}_j]=\mathcal{P}_j \label{commr3}
\end{align}
for potentially different operators $\Pp_j$ of the same form on the right hand side.
\bigskip

We now introduce the angular Sobolev spaces which will play an important role in this paper, using the construction in \cite{Sterbenz}. For background on the properties of spherical harmonics see \cite{stein1970singular}, Chapter 3.3.

For a function $f$ on $\R^3$, we define fractional angular derivatives $|\Om|^s$ as follows. First decompose $f$ into a sum of spherical harmonics:
\begin{align}\label{sph_decomp}
f(r,\tht)=\sum_{l=0}^\infty\sum_{i=1}^{N_l}c^i_l(r) Y^i_l(\tht), && c^i_l(r):=\f{1}{4\pi}\int_{\Sp^2}f(r,\tht)\overline{Y}^i_l(\tht)d\tht
\end{align}
Here $(Y^i_l)_{i=1,\ldots,N_l}$ is an orthonormal basis for the space of spherical harmonics of degree $l$, of finite dimension $N_l$. We recall that the spherical harmonics are eigenfunctions of the spherical Laplacian $|\Om|^2$:
\begin{align*}
|\Om|^2Y^i_l=l(l+1)Y^i_l, \quad\quad\quad l\geq0,\ i=0,\ldots,N_l.
\end{align*}
It follows that a suitable definition of $|\Om|^sf$ is given by
$$
|\Om|^sf(r,\tht):=\sum_{l=0}^\infty\sum_{i=1}^{N_l}[l(l+1)]^{s/2}c^i_l(r) Y^i_l(\tht)
$$
Note that this vanishes whenever $f$ is a radial function.

Recall the important fact that the spherical harmonics are invariant under the action of the Fourier transform, that is to say for each $l\geq0$ there is a map $T_l:L^2(r^2dr)\mapsto L^2(r^2dr)$ such that
$$
\mathcal{F}(c(r)Y^i_l(\tht))(\rho,\om)=T_l(c)(\rho)Y^i_l(\om)
$$
In particular, since the wave evolution operators $e^{\pm it\sqrt{-\D}}$ are given by radial multipliers we have
$$
e^{\pm it\sqrt{-\D}}(c(r)Y^i_l(\tht))=T_l^{-1}(e^{\pm it|\cdot|}T_l(c))(r)Y^i_l(\tht)
$$
and it follows that the fractional angular derivatives commute with the free evolution operators:
$$
|\Om|^s(e^{\pm it\sqrt{-\D}}f)=e^{\pm it\sqrt{-\D}}(|\Om|^sf)
$$

By incorporating angular regularity into our function spaces, we are able to make use of the following generalised Strichartz estimate which follows from the work of \cite{Sterbenz}.
\begin{theorem}[$n=3$ Strichartz estimates with angular regularity]\label{angular_strichartz}
Let $(p,q)$ be a pair which is radially admissible but not standard wave admissible:
\begin{align*}
\f{1}{p}+\f{2}{q}<1, && \f{1}{p}+\f{1}{q}>\half
\end{align*}
Suppose further that $p\neq2$. Then for all $\eta>0$ sufficiently small it holds
\begin{equation*}
\|e^{\pm it\sqrt{-\D}}P_kf\|_{L^p_tL^q_x}\lesssim_\eta2^{(\f{3}{2}-\f{1}{p}-\f{3}{q})k}(\|P_kf\|_{L^2_x}+\||\Om|^{s(p,q)}P_kf\|_{L^2_x})
\end{equation*}
for any function $f$ such that the right hand side is finite. Here
$$
s(p,q):=\f{2}{p}+\f{2}{q}-1+\eps(p,q;\eta)
$$
where $\eps(p,q;\eta)\to0$ as $\eta\to0$. Note that $s(p,q)\leq\half$ for $(p,q)$ as given and $\eta$ sufficiently small.
\end{theorem}

\begin{proof}
By scaling it suffices to consider $k=0$. We use the notation of \cite{Sterbenz} and direct the reader to that work for further details. In particular, let $\theta:[0,\infty)\rightarrow[0,1]$ be a smooth function equal to $1$ on $[1,2]$ and vanishing outside $[1/2,4]$, and set $\theta_N(l):=\theta(N^{-1}l)$ for $N\in 2^{\mathbb{N}}$. For the decomposition of $f$ as in \eqref{sph_decomp} we then denote 
\begin{equation}\label{ang_lloc}
f_N:=\sum_{l=0}^\infty\sum_{i=1}^{N_l}\tht_N(l)c^i_l(r) Y^i_l(\tht)
\end{equation}

Let $\eta>0$. By Proposition 3.4 in \cite{Sterbenz} we find
\begin{equation*}
\|e^{\pm it\sqrt{-\D}}P_0f_N\|_{L^{2}_tL^{r_\eta}_x}\lesssim_\eta N^{\half+\eta}\|P_0f_N\|_{L^2_x}
\end{equation*}
for some $r_\eta\searrow4$ as $\eta\to0$. A three-way interpolation of this result with the standard Strichartz estimate
\begin{equation*}
\|e^{\pm it\sqrt{-\D}}P_0f_N\|_{L^{\f{2}{1-\eta}}_tL^\infty_x}\lesssim_\eta\|P_0f_N\|_{L^2_x}
\end{equation*}
and the energy estimate
\begin{equation*}
\|e^{\pm it\sqrt{-\D}}P_0f_N\|_{L^\infty_tL^2_x}\lesssim\|P_0f_N\|_{L^2_x}
\end{equation*}
yields
\begin{equation*}
\|e^{\pm it\sqrt{-\D}}P_0f_N\|_{L^p_tL^q_x}\lesssim_\eta N^{s(p,q)}\|P_0f_N\|_{L^2_x}\simeq\||\Om|^{s(p,q)}P_0f_N\|_{L^2_x}
\end{equation*}
provided we choose $\eta$ sufficiently small to ensure that the pair $(p,q)$ is covered by the interpolation.

The radial part of the evolution is covered by the radial Strichartz estimate (Theorem 1.3, \cite{Sterbenz}): denoting $f_0:=c^0_0(r)Y^0_0(\tht)=c^0_0(r)$, the radial part of $f$, we have
\begin{equation*}
\|e^{\pm it\sqrt{-\D}}P_0f_0\|_{L^p_tL^q_x}\lesssim \|P_0f_0\|_{L^2_x}
\end{equation*}

The result then follows from the Littlewood-Paley-Stein theorem for the sphere (Theorem 2 \cite{S_multipliers}, see also \cite{stein_book2}.), upon observing that the angular frequency localisation \eqref{ang_lloc} commutes with the operator $P_0e^{\pm it\sqrt{-\D}}$. 
\end{proof}

In practice, we will only work with integer-order angular derivatives so as to use the Leibniz properties discussed previously. For this we must be able to exchange fractional angular derivatives for true derivatives, which is possible thanks to the following result:
\begin{lemma}[Riesz estimate for angular Sobolev spaces (Theorem 3.5.3, \cite{spheres_book})]\label{angular_riesz}
Let $1<p<\infty$. Then
$$
\max_{i,j}\|\Om_{ij}f\|_{L^p_x}\simeq\||\Om|f\|_{L^p_x}
$$
for any $f$ such that the right hand side is finite.
\end{lemma}

We also use the following monotonicity property for the angular Sobolev spaces, which can be proved for example using the decay of the corresponding multiplier (see Corollary 1, \cite{S_multipliers}).
\begin{lemma}[Monotonicity of Angular Sobolev Spaces]\label{monotoncity}
Let $1<p<\infty$, $s>s'>0$. It holds
$$
\||\Om|^{s'}f\|_{L^p_x}\lesssim_{s-s'} \||\Om|^s f\|_{L^p_x}
$$
\end{lemma}
Combined with Theorem \ref{angular_strichartz} the previous two lemmas yield the following (defining $\LOR$ as in \eqref{LOR_notn}).
\begin{corollary}\label{angular_strichartz_c}
    Let $\mathcal{Q}$ be any finite set of radially admissible pairs $(p,q)$ with $p\neq2$. Then it holds
    \begin{equation*}
    \max_{(p,q)\in\mathcal{Q}}2^{(\f{3}{2}\mfpq)k}\|e^{\pm it\sqrt{-\D}}P_kf\|_{L^p_tL^q_x}\lesssim_{\mathcal{Q}}\|\LOR P_kf\|_{L^2_x}
    \end{equation*}
\end{corollary}

\subsection{Function Spaces and Linear Estimates}\label{function_spaces_section}
Our function spaces are an adaptation of the usual Besov-type Strichartz spaces. Henceforth $\mathcal{Q}$ will denote a fixed set of radially admissible exponents as in Corollary \ref{angular_strichartz_c} to be determined throughout the proof, but certainly containing $(\infty,2)$. We then define the norm\footnote{We make the usual abuse of notation and write that $\phi\in S$ when really $\phi-p\in S$ for some fixed $p\in\Sp^2$, $S$ being the closure of the Schwarz functions under the given norm.}
$$
\|\phi\|_{S([0,T])}=\sum_{k\in\Z}\|\phi_k\|_{S_k([0,T])}
$$
with
\begin{equation}\label{norm}
\|\phi_k\|_{S_k([0,T])}:=\max_{(p,q)\in\mathcal{Q}}2^{(\f{1}{p}+\f{3}{q}-1)k}\|\LOR^{1-\dl(p,q)}\na_{t,x}P_k\phi\|_{L^p_tL^q_x([0,T]\x\R^3)}
\end{equation}
Here\footnote{Note that a more natural choice would be
$
\dl(p,q)=\max\{\f{2}{p}+\f{2}{q}-1+\eps(p,q;\eta),0\},
$
however we opt for the weaker norm above so as to encounter only full angular derivatives. Presumably it would be possible to to work with fractional angular derivatives by introducing a paradifferential calculus in the angular variable, see for instance  \cite{hirayama2024sharp,hong2022improved}.}
\[
\dl(p,q)=
\begin{cases}
0 \text{ if } \f{1}{p}+\f{1}{q}\leq \f{1}{2}\\
1 \text{ otherwise}
\end{cases}
\]

We will also work with the vector fields $L_n$ introduced in Section \ref{ang_com_chap}, however rather than incorporating these into the norm we will directly apply them to the solution we are working with.\footnote{This is for technical reasons to handle the non-trivial commutator $[L_n,P_k]$.}
Define 
\begin{align}\label{pl}
\pl:=L\phi=
\begin{pmatrix}
\phi\\
L_1\phi\\
L_2\phi\\
L_3\phi
\end{pmatrix},
&&
\bl:=
\begin{pmatrix}
1\\
L_1\\
L_2\\
L_3
\end{pmatrix},
\end{align}
so
\begin{equation}\label{norm}
\|P_k\phi^L\|_{S_k([0,T])}\simeq\max_{n=0,\ldots,3}\max_{(p,q)\in\mathcal{Q}}2^{(\f{1}{p}+\f{3}{q}-1)k}\|\LOR^{1-\dl(p,q)}\na_{t,x}P_kL_n\phi\|_{L^p_tL^q_x([0,T]\x\R^3)}
\end{equation}
with the convention $L_0:=1$.

We have the following linear estimate which is a straightforward application of Corollary \ref{angular_strichartz_c}:
\vspace{-0.5em}
\begin{theorem}[Linear Estimate]\label{linear_estimate}
Let $\phi$ satisfy the linear wave equation $\Box \phi=F$ with initial data $\phi[0]\equiv(\phi(0,\cdot),\dd_t\phi(0,\cdot))$ on the interval $[0,T]$.
It holds
\begin{align*}
\|\phi_k\|_{S_k([0,T])}\lesssim\|\LOR P_k\phi[0]\|_{\dot{H}^{3/2}\x \dot{H}^{1/2}}+\|\LOR F_k\|_{L^1_t \dot{H}^{1/2}_x([0,T]\x\R^3)}
\end{align*}
and as a corollary
\begin{align*}
\|P_k\phi^L\|_{S_k([0,T])}&\lesssim\|\LOR P_k\phi[0]\|_{\dot{H}^{3/2}\x \dot{H}^{1/2}}+\|\LOR P_k(x\cdot\na)\phi[0]\|_{\dot{H}^{3/2}\x \dot{H}^{1/2}}+\|\LOR P_k(x\cdot\Box\phi(0))\|_{\dot{H}^{1/2}}\\
&\hspace{22em}+\|\LOR P_kF^L\|_{L^1_t \dot{H}^{1/2}_x([0,T]\x\R^3)}
\end{align*}
where $F^L$ is as in \eqref{pl}.
\end{theorem}

\subsection{Angular Multipliers}\label{angular_multipliers_subsection}
To conclude this section, we introduce the angular multipliers which will play a key role in the main estimates of this article (see Lemma \ref{angular_sep_lem}). For fixed $\rho\leq0$, we first introduce a smooth partition of unity on the sphere, $(\sg^\beta_\rho)_{\beta\in\Ss_\rho}$, given by
\begin{equation}\label{spatial_partition}
\sg^\beta_\rho(x):=\f{\sg(2^{-\rho}\|\hat{x}\x\beta\|)}{\sum_{\beta'\in\Ss_\rho}\sg(2^{-\rho}\|\hat{x}\x\beta'\|)}
\end{equation}
for $\sg\in C_c^\infty$ supported on $[0,101/100]$ and equal to $1$ on $[0,1]$. $\Ss_\rho$ is a set of $\sim 2^{-2\rho}$ points on the sphere such that for every $\hat{x}\in\Sp^2$ there exists $\beta\in\Ss_\rho$ such that $\|\hat{x}\x\beta\|\leq 2^\rho$.
We choose our functions in such a way as to ensure the almost-orthogonality relation
$$
\|u\|_{L^2_x}\simeq\left(\sum_{\beta\in\Ss_\rho}\|\sg^\beta_\rho(x)u(x)\|_{L^2_x}^2\right)^\half
$$
holds uniformly in $\rho\leq0$.

For each $\rho$ sufficiently small, $\beta\in\Ss_\rho$, we also introduce a Whitney-type decomposition of the sphere in Fourier space. This consists of functions $\eta^{(r,l)}_r$ cutting off to discs of radius $\sim 2^r$ at distance $\sim 2^r$ from $\beta$, made precise in the following proposition. These cut-offs are turned into operators by defining, for example,
$$
\eta_r^{(r,l)}(D)\phi(x):=\F^{-1}(\eta_r^{(r,l)}(\xi)\hat{\phi}(\xi))(x)
$$

\begin{proposition}\label{partition_prop}
There exist absolute constants $C_1,C_2,C_3,C_4,N>0$ such that the following holds. For any $\rho\leq-C_1$, $\beta\in\Ss_\rho$ there is a partition of unity consisting of functions 
\begin{equation}\label{nbr_def}
\eta^\beta_{\rho+C_2} \text{ and } \eta_r^{(r,l)} \text{ } (\rho+C_1\leq r\leq0, l=1,\ \ldots,N)
\end{equation}
with the following properties:
\begin{enumerate}
\item There are points $\al_{r,l}\in\Sp^2$ and functions $\tilde{\eta}^\beta_{\rho+C_1}$, $\tilde{\eta}_r^{(r,l)}$ of the form $\tilde{\eta}^\beta_{\rho+C_2}(\xi)=\sg(2^{-(\rho+C_2)}\|\hxi\x\beta\|)$ and $\tilde{\eta}_r^{(r,l)}(\xi)=\sg(2^{-(r-C_3)}\|\hxi\x\al_{r,l}\|)$ for $\sg$ as before such that
$$
\eta^\beta_{\rho}=\f{\tilde{\eta}^\beta_{\rho}}{\tilde{\eta}^\beta_{\rho+C_2}+\sum\tilde{\eta}_r^{(r,l)}} \text{ and } \eta_r^{(r,l)}=\f{\tilde{\eta}_r^{(r,l)}}{\tilde{\eta}^\beta_{\rho+C_2}+\sum\tilde{\eta}_r^{(r,l)}}
$$
We allow for a smaller constant $C_3$ when $r=0$.
\item $\|\al_{r,l}\x\beta\|\simeq2^r$ for all $(r,l)$ 
and $\|\hx\x\hxi\|\gtrsim2^r$ for all $\hx\in\spt(\sg^\beta_\rho)$ and all $(r,l)$.
\item $\spt(\eta^{(r',l')}_{r'})\cap\spt(\eta^{(r,l)}_r)=\emptyset$ for all $|r-r'|\geq C_4$ and $\spt(\eta^\beta_{\rho+C_2})\cap\spt(\eta^{(r,l)}_r)=\emptyset$ for all $r\geq\rho+C_4$.
\end{enumerate}
\end{proposition}

We then have the following lemma concerning the boundedness of these multipliers.
\begin{lemma}\label{boundedness_of_angular_multipliers}
Let $1\leq q\leq\infty$. For any $j\in\Z$, $\rho\leq -C_1$, $\beta\in\Ss_\rho$, $\rho+C_1\leq r\leq0$ and $l=1,\ldots,N$ it holds
$$
\|\eta^\beta_{\rho+C_1}(D)P_j\phi\|_{L^q_x},\, \|\eta^{(r,l)}_r(D)P_j\phi\|_{L^q_x}\lesssim_q\|P_j\phi\|_{L^q_x}
$$
\end{lemma}
\section{Reduction to main proposition}\label{reduction_chap}
We will work with frequency envelopes to reduce our critical global result to the subcritical local result of Theorem \ref{LWP_thm} (proved in Section \ref{LWP_proof}). This section is largely based on Section 3 of \cite{Tao_1}.

In what follows we fix $\sg\in(0,1)$ (which will need to be taken sufficiently small), $s\in(3/2,3/2+\sg)$ and $0<\eps\ll 1$ which may depend on $\sg$. We also need the following definition from \cite{Tao_1}.
\begin{definition}[Frequency envelope]
We call $c=(c_k)_{k\in\Z}\in\ell^1$ a \textit{frequency envelope} if $$\|c\|_{\ell^1}\lesssim\eps$$ and $$2^{-\sg|k-k'|}c_{k'}\lesssim c_k\lesssim 2^{\sg|k-k'|}c_{k'}$$
We say that $(f,g)$ \textit{lies underneath} the envelope $c$ if 
\begin{equation*}
\|P_k(f,g)\|_{\dot{H}^{3/2}\x\dot{H}^{1/2}}\leq c_k
\end{equation*}
for all $k\in\Z$.
\end{definition}

Our first step in proving Theorem \ref{thm1.2} is to make the following reduction, saying that the frequency profile of the solution stays roughly constant along the evolution.
\begin{proposition}[Main Proposition]\label{main_prop}
Let $0<T<\infty$, $c$ be a frequency envelope and $\phi$ a smooth solution to \eqref{eqn1.2} on $[0,T]\times\R^3$ with initial data $\phi[0]$ satisfying the smallness condition
\begin{equation}\label{small1}
\|\LOR P_k\phi[0]\|_{\dot{H}^{3/2}\x\dot{H}^{1/2}}+\|\LOR (x\cdot\na) P_k\phi[0]\|_{\dot{H}^{3/2}\x\dot{H}^{1/2}}\leq c_k
\end{equation}
for all $k$. Let $\pl$ be as in \eqref{pl}. Then if $\eps$ is sufficiently small it holds
\begin{equation}\label{result}
\|P_k\pl\|_{S_k([0,T])}\leq C_0c_k
\end{equation}
for all $k\in\Z$, where $C_0\gg1$ is an absolute constant. In particular, $\phi[t]$ lies underneath the frequency envelope $C_0c$ for all $t\in[0,T]$.
\end{proposition}

We quickly outline how Theorem \ref{thm1.2} follows from this proposition (and Theorem \ref{LWP_thm}). Given data $\phi[0]$ as in the statement with $\eps$ sufficiently small, define a frequency envelope
\begin{equation*}
c_k:=\sum_{j\in\Z}2^{-\sg|j-k|}(\|\LOR  P_j\phi[0]\|_{\dot{H}^{3/2}\x\dot{H}^{1/2}}+\|\LOR  (x\cdot\na) P_j\phi[0]\|_{\dot{H}^{3/2}\x\dot{H}^{1/2}})
\end{equation*}
It is then clear that \eqref{small1} holds, so we see from the proposition that the local solution $\phi:[0,T]\x\R^3\rightarrow\Sp^2$ of Theorem \ref{LWP_thm} satisfies
$$
\|P_k\phi\|_{S_k([0,T])}\leq C_0c_k
$$
for all $k$. It follows that for any $s>3/2$ with $|s-3/2|<\sg$  we have
\begin{align*}
&\|P_k\phi\|_{L^\infty \dot{H}^s([0,T]\x\R^3)}+\|P_k\dd_t\phi\|_{L^\infty \dot{H}^{s-1}([0,T]\x\R^3)}\\
&\quad\quad\lesssim 2^{(s-3/2)k}\|P_k\phi\|_{S_k([0,T])}\\
&\quad\quad\lesssim2^{(s-3/2)k}C_0c_k\\
&\quad\quad\lesssim C_0\sum_{j\in\Z}2^{(|s-3/2|-\sg)|k-j|}(\|P_j\LOR \phi[0]\|_{\dot{H}^s\x\dot{H}^{s-1}}+\|P_j\LOR (x\cdot\na)\phi[0]\|_{\dot{H}^s\x\dot{H}^{s-1}})
\end{align*}
for all $k\in\Z$, from which we see that
$$
\|\phi[t]\|_{\dot{B}^s_{2,1}\x\dot{B}^{s-1}_{2,1}}\lesssim C_0(\|\LOR \phi[0]\|_{\dot{B}^s_{2,1}\x\dot{B}^{s-1}_{2,1}}+\|\LOR (x\cdot\na)\phi[0]\|_{\dot{B}^s_{2,1}\x\dot{B}^{s-1}_{2,1}})
$$
for all $t\in[0,T]$. The low-frequency portion of $\phi[t]$ is straightforward to bound using energy estimates, however it is something of a distraction at this point so we postpone this to Appendix \ref{appendixN}.

In summary we obtain uniform bounds on the ${B}^s_{2,1}\x {B}^{s-1}_{2,1}$ norm of the solution. Since Proposition \ref{main_prop} also shows that smallness in the critical space is (almost) conserved, it follows from the local theory that the solution extends globally.

Using the same argument as in \cite{Tao_1}, we see that Proposition \ref{main_prop} can be further reduced to the following statement, to whose proof the bulk of this paper is dedicated.
\begin{proposition}[Reduced Main Proposition]\label{reduced_main_prop}
Let $c$ be a frequency envelope, $0<T<\infty$ and $\phi$ be a smooth half-wave map on $[0,T]\times\R^3$ such that $\phi[0]$ satisfies \eqref{small1}. Suppose that
\begin{equation}\label{assumption1}
\|P_k\pl\|_{S_k([0,T])}\leq 2C_0c_k
\end{equation}
for all $k\in\Z$. Then in fact
\begin{equation}\label{main_goal}
\|P_k\pl\|_{S_k([0,T])}\leq C_0c_k
\end{equation}
for all $k\in\Z$ (assuming that $C_0$ is sufficiently large and $\eps$ is sufficiently small).
\end{proposition}

\section{Discarding some error terms}\label{discarding}
We now begin the first step in the proof of Proposition \ref{reduced_main_prop}, where we will show that some terms in the forcing of equation \eqref{eqn1.2} can be ignored.

Fix $c$, $T$ and $\phi$ satisfying the hypotheses of  the proposition. We need to show (\ref{main_goal}). By scaling invariance it suffices to prove that $$\|\psi^L\|_{S_0([0,T])}\leq C_0c_0$$
for $\psi^L:=P_0(\phi^L)$. We will use the notation $\psi^L\equiv(\psi_0,\psi_1,\psi_2,\psi_3)$ so that $\psi_n=P_0(L_n\phi)$.

Thanks to the linear estimate, it would be sufficient to show that
\begin{equation*}
\|\LOR  \Box\psi^L\|_{L^1_tL^2_x}\lesssim C_0^3c_0\eps
\end{equation*}
and take $\eps(C_0)$ sufficiently small (the initial data term involving $\Box\phi$ can be bounded straightforwardly, see \eqref{x_nbox}). Unfortunately, it will not be possible to show this directly, however after some transformations we will be able to achieve a similar form, as we shall see in the coming sections.

This motivates the definition of an ``error" term:
\begin{definition}[Error terms]
A function $F=(F_0,\ldots,F_3)$ on $[0,T]\times\R^n $ is said to be an \emph{acceptable error} if $$\|\langle\Omega\rangle F\|_{L^1_tL^2_x}\lesssim C_0^3c_0\eps$$
In this case we write $F=error$. We will also denote by \emph{error} the components of such a vector.
\end{definition}

Applying $P_0$ to equation \eqref{eqn1.2} we find
\begin{equation}\label{Boxpsi}
\Box\psil=P_0\bl(-\phi\dd_\alpha\phi^T\dd^\alpha\phi+HWM(\phi))
\end{equation}
where
$$
HWM(\phi):=\Phip[(\dphi)(\phi\cdot\dphi)]\nonumber+\phi\x[\dhalf(\phi\x\dphi)-(\phi\x(-\D)\phi)]
$$
  
Our first step is to remove the most simple frequency interactions from the wave-maps source term. We introduce the following notation for some commonly used Strichartz pairs, with $M=\infty-$ a large constant which $\eta$, $\sg$ and $\eps$ will all depend on. 
\begin{gather}\label{notation}
\begin{aligned}
    &(2+,\infty-):=(\f{2M}{M-1},2M), && \|\LOR P_k\phi^L\|_{2+,\infty-}\lesssim2^{-(\half+\f{1}{M})k}\|P_k\phi^L\|_{S_k}\\
    &(\infty-,2+):=(M,\f{2M}{M-2}), && \|\LOR P_k\phi^L\|_{\infty-,2+}\lesssim 2^{-(\f{3}{2}-\f{2}{M})k}\|P_k\phi^L\|_{S_k}
\end{aligned}
\end{gather}

\begin{proposition}\label{removing_trivial_terms}
We have 
\begin{align}
P_0(\phi\dd_\alpha\phi^T\dd^\alpha\phi)=2P_0(\phi_{\leq-10}\dd_\alpha\phi_{\leq-10}^T\dd^\alpha\phi_{>-10})+P_0(\phi_{\leq-10}\dd_\alpha\phi_{>-10}^T\dd^\alpha\phi_{>-10})+error\label{big_thing1}
\end{align}
and
\begin{align}
P_0L_n(\phi\dd_\alpha\phi^T\dd^\alpha\phi)&=
2P_0((\lnp)_{\leq-10}\dd_\alpha\phi_{\leq-10}^T\dd^\alpha\phi_{>-10})\nonumber\\
&\quad\quad+2P_0(\phi_{\leq-10}\dad(L_n\phi)_{\leq-10}^T\dd^\alpha\phi_{>-10})\\
&\quad\quad+2P_0(\phi_{\leq-10}\dad\phi_{\leq-10}^T\dau(L_n\phi)_{>-10})\nonumber\\
&\quad+P_0((\lnp)_{\leq-10}\dd_\alpha\phi_{>-10}^T\dd^\alpha\phi_{>-10})\nonumber\\
&\quad\quad+P_0(\phi_{\leq-10}\dd_\alpha(L_n\phi)_{>-10}^T\dd^\alpha\phi_{>-10})\nonumber\\
&\quad\quad+P_0(\phi_{\leq-10}\dd_\alpha\phi_{>-10}^T\dd^\alpha(L_n\phi)_{>-10})\nonumber\\
&\quad+error \label{big_thing2}
\end{align}
for $n=1,2,3$.
\end{proposition}
\begin{proof}
We will only show \eqref{big_thing2}, \eqref{big_thing1} being similar. We start with the following observation, using that $L_n$ commutes with the wave operator and satisfies the Leibniz rule:
$$
L_n(\dad\phi^T\dau\phi)=\half L_n(\Box(\phi^T\phi)-2\phi^T\Box\phi)=2\dad(L_n\phi)^T\dau\phi
$$
Applying this property and the Leibniz rule on the whole nonlinearity we have
$$
P_0L_n(\phi\dd_\alpha\phi^T\dd^\alpha\phi)
=P_0((L_n\phi)\dd_\alpha\phi^T\dd^\alpha\phi)+2P_0(\phi\dad(L_n\phi)^T\dau\phi)
$$
Henceforth we restrict our attention to the first term, the other term being treated identically. We also drop the subscript on $L_n$.

Decomposing each factor of $\phi$ into low and high frequencies, and noting that the term vanishes when all three factors are at low frequency, we write
\begin{align}
P_0((L\phi)\dd_\alpha\phi^T\dd^\alpha\phi)=&2P_0((L\phi)_{>-10}\dd_\alpha\phi_{>-10}^T\dd^\alpha\phi_{\leq-10})\label{term1}\\
&+{P}_0((L\phi)_{>-10}\dd_\alpha\phi_{\leq-10}^T\dd^\alpha\phi_{\leq-10})\label{term2}\\
&+{P}_0((L\phi)_{>-10}\dd_\alpha\phi_{>-10}^T\dd^\alpha\phi_{>-10})\label{term3}\\
&+2{P}_0((L\phi)_{\leq-10}\dd_\alpha\phi_{\leq-10}^T\dd^\alpha\phi_{>-10})\label{term4}\\
&+{P}_0((L\phi)_{\leq-10}\dd_\alpha\phi_{>-10}^T\dd^\alpha\phi_{>-10})\label{term5}
\end{align}
Of these terms, \eqref{term4} and \eqref{term5} appear in \eqref{big_thing2} so we have to show that $\eqref{term1}=\eqref{term2}=\eqref{term3}=error$.
\begin{itemize}
\item \underline{\eqref{term1}, $\hhl$:}

We have
\begin{align*}
&\|\LOR P_0((L\phi)_{>-10}\dd_\alpha\phi_{>-10}^T\dd^\alpha\phi_{\leq-10})\|_{1,2}\\
&\lesssim\|\LOR (L\phi)_{>-10}\|_{\tpim}\|\LOR \dd^\alpha \phi_{>-10}\|_{\imtp}\|\LOR \dd_\alpha \phi_{\leq-10}\|_{\tpim}
\end{align*}
where we used the Leibniz rule to spread the angular derivatives across the 3 terms, followed by the monotonicity of the angular Sobolev spaces. Using the definition of the S-norm and the local constancy of the frequency envelope we therefore see that
\begin{align*}
\|\LOR \eqref{term1}\|_{1,2}&\lesssim\left(\sum_{j>-10}2^{-(\half+\f{1}{M})j}\|P_j\phi^L\|_{S_j}\right)\left(\sum_{k>-10}2^{-(\half-\f{2}{M})k}\|P_k\phi\|_{S_k}\right)\left(\sum_{l\leq-10}2^{(\half-\f{1}{M})l}\|P_l\phi\|_{S_l}\right)\\
&\lesssim C_0^3\eps^2 c_0
\end{align*}

\item \underline{\eqref{term2}, $\hll$:}

We similarly estimate
\begin{align*}
&\|\LOR P_0((L\phi)_{>-10}\dd_\alpha \phi_{\leq-10}^T\dd^\alpha \phi_{\leq-10})\|_{1,2}\\
&\lesssim\|\LOR (L\phi)_{>-10}\|_{\imtp}\|\LOR \dd^\alpha \phi_{\leq-10}\|_{\tpim}\|\LOR \dd_\alpha \phi_{\leq-10}\|_{\tpim}\\
&\lesssim C_0^3\eps^2 c_0
\end{align*}

\item \underline{\eqref{term3}, $\hhh$:}

This term cannot be handled in the standard Strichartz spaces. Observe that when $\LOR$ spreads over the three terms according to the Leibniz rule, in each case there will be at least one differentiated term which is not hit by an angular derivative. This term can then be placed into a non-standard Strichartz space. For example, when $\LOR$ falls on the first differentiated factor we have
\begin{align*}
\|P_0((L\phi)_{>-10}\LOR\dd_\alpha \phi_{>-10}^T\dd^\alpha \phi_{>-10})\|_{1,2}\lesssim&\|(L\phi)_{>-10}\|_{\f{18}{7},\infty}\|\LOR\dd_\alpha \phi_{>-10}\|_{9,\f{10}{3}}\|\dd^\alpha \phi_{>-10}\|_{2,5}\nonumber\\
\lesssim&\sum_{j,k,l>-10}2^{-\f{7}{18}j}\|P_j\phi^L\|_{S_j}\cdot 2^{-\f{1}{90}k}\|\phi_k\|_{S_k}\cdot 2^{-\f{1}{10}l}\|\phi_l\|_{S_l}\nonumber\\
\lesssim&C_0^3c_0^2\eps
\end{align*}
\end{itemize}
\end{proof}

Thanks to this proposition we can rewrite the frequency-localised equation as
\begin{align}
\Box\psi_0=&-2\phi_{\leq-10}\dd_\alpha\phi_{\leq-10}^T\dd^\al\psi_0-2[P_0(\phi_{\leq-10}\dd_\al\phi_{\leq-10}^T\dd^\al\phi_{>-10})-\phi_{\leq-10}\dd_\alpha\phi_{\leq-10}^T\dd^\al\psi_0]\nonumber\\
&-P_0(\phi_{\leq-10}\dd_\alpha\phi_{>-10}^T\dd^\alpha\phi_{>-10})+P_0(HWM(\phi))+error\label{new_wm_eqn}
\end{align}
and
\begin{align}
\Box\psi_n&=-2(L_n\phi)_{\leq-10}\dd_\alpha\phi_{\leq-10}^T\dd^\al\psi_0\nonumber\\
&\quad\quad-2\phi_{\leq-10}\dd_\alpha(L_n\phi)_{\leq-10}^T\dd^\al\psi_0\nonumber\\
&\quad\quad-2\phi_{\leq-10}\dd_\alpha\phi_{\leq-10}^T\dd^\al\psi_n\nonumber\\
&\quad-2[P_0((L_n\phi)_{\leq-10}\dd_\al\phi_{\leq-10}^T\dd^\al\phi_{>-10})-(L_n\phi)_{\leq-10}\dd_\alpha\phi_{\leq-10}^T\dd^\al\psi_0]\nonumber\\
&\quad\quad-2[P_0(\phi_{\leq-10}\dd_\al(L_n\phi)_{\leq-10}^T\dd^\al\phi_{>-10})-\phi_{\leq-10}\dd_\alpha(L_n\phi)_{\leq-10}^T\dd^\al\psi_0]\nonumber\\
&\quad\quad-2[P_0(\phi_{\leq-10}\dd_\al\phi_{\leq-10}^T\dd^\al(L_n\phi)_{>-10})-\phi_{\leq-10}\dd_\alpha\phi_{\leq-10}^T\dd^\al\psi_n]\nonumber\\
&\quad-P_0((L_n\phi)_{\leq-10}\dd_\alpha\phi_{>-10}^T\dd^\alpha\phi_{>-10})\nonumber\\
&\quad\quad-P_0(\phi_{\leq-10}\dd_\alpha(L_n\phi)_{>-10}^T\dd^\alpha\phi_{>-10})\nonumber\\
&\quad\quad-P_0(\phi_{\leq-10}\dd_\alpha\phi_{>-10}^T\dd^\alpha(L_n\phi)_{>-10})\nonumber\\
&\quad+P_0L_n(HWM(\phi))+error\label{new_wm_eqn_n}
\end{align}
for $n=1,2,3$.
We have now clearly identified the troublesome frequency interactions in the wave maps source term. In the next chapter we will show that the half-wave maps terms are acceptable, and in Chapter \ref{normal_forms_chapter} we will discard the second and third terms (or groups of terms) via normal transformations. Lastly in Chapters \ref{approx_PT_chapter} and \ref{cancellation_chapter} we will show that the remaining $\llh$ term can be gauged away using Tao's approximate parallel transport.

\section{The half-wave maps contributions are negligible}\label{HWM_chapter}
We decompose the half-wave maps forcing into two terms:
$$
HWM(\phi)=HWM_1(\phi)+HWM_2(\phi)
$$
with
$$
HWM_1(\phi):=\Phip[(\dphi)(\phi\cdot\dphi)]
$$
and
$$
HWM_2(\phi):=\phi\x[\dhalf(\phi\x\dphi)-(\phi\x(-\D)\phi)]
$$
As discussed in the introduction, we are able to discard of these terms entirely due to their geometric structures. We largely use techniques from \cite{KK}, with a novel ingredient for handling the $\llowesth$ frequency interactions (see Lemma \ref{SIL}).

Before preceding to the estimates, we present some lemmas which will be used frequently in the sequel. Denote
\begin{equation}\label{Lc}
\Lc_k(u_{\ko},v_{\kt}):=\int_{\R^3}\int_{\R^3}m_k(\xi,\eta)e^{ix\cdot(\xi+\eta)}\chi_{\ko}(\xi)\hat{u}(\xi)\chi_{\kt}(\eta)\hat{v}(\eta)d\xi d\eta
\end{equation}
for $m_k$ any smooth multiplier satisfying the pointwise bounds
$$
|m_k(\xi,\eta)|\lesssim 2^k,\text{ }|(2^{\ko}\na_{\xi})^i(2^{\kt}\na_{\eta})^jm_k(\xi,\eta)|\lesssim_{i,j}2^k
$$
on the support of $\chi_{\ko}(\xi)\chi_{\kt}(\eta)$. 
Note that such a multiplier can be expanded as a Fourier series with rapidly decaying coefficients on the support of $\chi_{\ko}(\xi)\chi_{\kt}(\eta)$:
\begin{align}\label{FS}
m_k(\xi,\eta)=\sum_{a,b\in\Z^3}c^{(k)}_{a,b}e^{-i(2^{-\ko}a\cdot\xi+2^{-\kt}b\cdot\eta)}&& \text{ with } |c^{(k)}_{a,b}|\lesssim_N 2^k\japa^{-N}\japb^{-N} \text{ for any }N\in\mathbb{N}.
\end{align}
We can therefore, at least formally, write
\begin{align}\label{FS2}
\Lc_k(u_{\ko},v_{\kt})=\sum_{a,b\in\Z^3}c^{(k)}_{a,b}\uko(x-2^{-\ko}a) v_{\kt}(x-2^{-\kt}b)
\end{align}
Operators of this form arise in studying cancellations in $HWM_2(\phi)$, and an important property is given by the following 
\begin{lemma}[Lemma 3.1, \cite{KS}]\label{useful_lemma}
Let $\Lc_k$ be as above. Then if $\|\cdot\|_Z$, $\|\cdot\|_X$, $\|\cdot\|_Y$ are translation invariant norms with the property that
$$
\|u\cdot v\|_Z\leq \|u\|_X\|v\|_Y
$$
it holds
$$
\|\Lc_k(u_{\ko},v_{\kt})\|_Z\lesssim 2^k\|u_{\ko}\|_X\|v_{\kt}\|_Y
$$
\end{lemma}
In particular this lemma tells us we can (and should) think of $\Lc_j(\phi_j,\phi_k)$ as $\dd\phi_j\cdot\phi_k$.

Due to the generally nonlocal nature of these operators, they interact non-trivially with the non-translation invariant commuting vector fields. In fact for $k,\ko,\kt\in\Z$, $n=1,2,3$, $i,j=1,2,3$ it holds
\begin{equation}\label{L_LL}
L_n(\Lc_k(u_{\ko},v_{\kt}))=\Lc_k(L_nu_{\ko},v_{\kt})+\Lc_k(u_{\ko},L_nv_{\kt})+\Lc_{k-\ko}(\dd_t\uko,\ukt)+\Lc_{k-\kt}(\uko,\dd_t\ukt)
\end{equation}
and
\begin{align}
\Om_{ij}(\Lc_k(u_{\ko},v_{\kt}))=\Lc_k(\Om_{ij} u_{\ko},v_{\kt})+\Lc_k(u_{\ko},\Om_{ij}v_{\kt})+\Lc_k(u_{\ko},v_{\kt})\label{om_exp}
\end{align}
The $\Lc_k$ in these expressions need not all correspond to the same multiplier $m_k$.

Lastly, we note the following basic facts which follow from the commutation relations between $\Om$, $L$ and $P_k$. For any $(p,q)\in\mathcal{Q}$ and $a\in\R$ it holds\footnote{This may be interpreted as saying that translation does not affect the norm provided we translate on scales at most comparable to the natural oscillation length of $u_k$.}
\begin{equation}\label{translation_bound}
\|\LOR^{1-\dl(p,q)} L(u_k(\cdot+a))\|_{p,q}\lesssim \langle 2^ka \rangle^2 2^{-(\f{1}{p}+\f{3}{q})k}\|P_{\sim k}u^L\|_{S_{\sim k}}
\end{equation}
and if further $q\neq\infty$ we also have\footnote{The restriction to $q<\infty$ comes from the need to bound the Riesz transform appearing in $[L,\dhalf]$. In practice this is not important since there is usually enough flexibility in the estimates to lower $q$ using Bernstein's inequality.}
$$
\|\LOR^{1-\dl(p,q)}L\dhalf u_k\|_{p,q}\lesssim2^{(1-\f{1}{p}-\f{3}{q})k}\|P_{\sim k}u^L\|_{S_{\sim k}}
$$

\bigskip

We now come to the most important lemma of this paper, which is the key ingredient for handling the $\llowesth$ terms in 3 dimensions. For notation used in the statement we refer the reader to Section \ref{angular_multipliers_subsection}.

\begin{lemma}[Angular Separation Estimate]\label{angular_sep_lem}
Fix $\rho\leq-C_1$. Then for any $\la,k\in\Z$, $\rho+C_1\leq r\leq0$, $l=0,\ldots,N$ it holds
\begin{align*}
\|\vpl(x)\sg_{\rho}^\beta(x)\eta^{(r,l)}_r(D)\phi_k\|_{L^q_x}\lesssim 2^{-(\la+k+2r)}(\|\phi_k\|_{L^q_x}+\max_{i,j}\|\Om_{ij} \phi_k\|_{L^q_x})
\end{align*}
for any $1\leq q\leq\infty$.
\end{lemma}
The intuition for this estimate was discussed in the introduction, and leads us to expect a preferable loss of $2^{-(\la+j+r)}$. Unfortunately, we were not quite able to achieve this. This comes from the $\Om_{ij}$ being non-translation invariant.

\begin{proof}
We may assume without loss of generality that $\alpha_{r,l}=e_1$, and by scaling it suffices to consider $k=0$. We may further assume that $\beta$ lies in the $x-y$ plane so that $|\hat{x}_1\hxi_2-\hx_2\hxi_1|\gtrsim 2^r$ for any $\hx,\,\hxi\in\Sp^2$ in the supports of $\sg^\beta_\rho,\,\eta^{(r,l)}_r$ respectively.


Write
$$
\vpl(x)\sg_\rho^\beta(x)\eta^{(r,l)}_r(D)\phi_0=\int_{\xi}e^{ix\cdot\xi}m(x,\xi)\cdot(x_1\xi_2-x_2\xi_1)\hat{\phi}_0(\xi)d\xi
$$
for
$$
m(x,\xi):=\vpl(x)\sg_\rho^\beta(x)\tilde{\chi}_0(\xi)\eta^{(r,l)}_r(\xi)(x_1\xi_2-x_2\xi_1)^{-1}
$$
Expand $m$ as a Fourier series in $\xi$. Since $\al_{r,l}=e_1$ we have $\text{supp}\,m\subset\{\xi_1\sim 1, |\xi_2|, |\xi_3|\lesssim 2^r \}$ so
$$
m(x,\xi)=\vpl(x)\sg_\rho^\beta(x)\sum_{p\in\Z^3}c_{p}(x) e^{2\pi i(\xi_1 p_1+2^{-r}\xi_2 p_2+2^{-r}\xi_3 p_3)}
$$
where
\begin{align*}
c_{p}(x)\simeq2^{-2r}\int_{\substack{|\xi_1|\sim 1\\|\xi_2|,|\xi_3|\lesssim 2^r}}\tilde{\chi}_0(\xi)\eta_r^{(r,l)}(\xi)(x_1\xi_2-x_2\xi_1)^{-1} e^{-2\pi i(\xi_1 p_1+2^{-r}\xi_2 p_2+2^{-r}\xi_3 p_3)}d\xi  
\end{align*}
We want to integrate by parts so need bounds on the derivatives of the integrand. A calculation yields
$$
|\na_\xi^\gamma \eta^{(r,l)}_r(\xi)|\lesssim_\gamma 2^{-(\gamma_2+\gamma_3)r}
$$
for all $\gamma\in \mathbb{N}^3$, $\xi\in\text{supp}(\tilde{\chi}_0)\cap\spt(\eta^{(r,l)}_r)$.

Furthermore, for $x$, $\xi$ in the support of $m$, we have $|x_2|\leq|x|\|\hx\x\al_{r,l}\|\lesssim2^{\la+r}$ and so
\begin{align*}
|\dd_{\xi_1}^{\gamma_1}\dd_{\xi_2}^{\gamma_2}(x_1\xi_2-x_2\xi_1)^{-1}|=&|x_1\xi_2-x_2\xi_1|^{-(\gamma_1+\gamma_2+1)}|x_2|^{\gamma_1}|x_1|^{\gamma_2}\\
\lesssim& 2^{-(\la+r)(\gamma_1+\gamma_2+1)}2^{(\la+r) \gamma_1}2^{\la\gamma_2}=2^{-(\la+r)}2^{-r\gamma_2}
\end{align*}

It follows that
\begin{align*}
&|\na_\xi^\gam(\tilde{\chi}_0(\xi)\eta^{(r,l)}_r(\xi)(x_1\xi_2-x_2\xi_1)^{-1})|\\
&\lesssim\sum_{\gam^{(1)}+\gam^{(2)}+\gam^{(3)}=\gam}|\na_\xi^{\gam^{(1)}}\tilde{\chi}_0(\xi)\cdot \na_\xi^{\gam^{(2)}}\eta^{(r,l)}_r(\xi) \cdot\na_\xi^{\gam^{(3)}}((x_1\xi_2-x_2\xi_2)^{-1})|\\
&\lesssim\sum_{\substack{\gam^{(1)}+\gam^{(2)}+\gam^{(3)}=\gam,\\ \gam^{(3)}_3=0}}1 \cdot 2^{-(\gam^{(2)}_2+\gam^{(2)}_3)r}\cdot 2^{-(\la+r)}2^{-\gam^{(3)}_2r}\\
&\lesssim 2^{-(\la+r)}\sum_{\substack{\gam^{(1)}+\gam^{(2)}+\gam^{(3)}=\gam,\\ \gam^{(3)}_3=0}} 2^{-(\gam^{(2)}_2+\gam^{(3)}_2+\gam^{(2)}_3)r}
\end{align*}
For $r\leq0$, the right hand side of this expression is largest when $\gam^{(2)}_2+\gam^{(3)}_2=\gam_2$ and $\gam^{(2)}_3=\gam_3$, leading to
$$
|\na_\xi^\gam(\tilde{\chi}_0(\xi)\eta^{(r,l)}_r(\xi)(x_1\xi_2-x_2\xi_1)^{-1})| \lesssim 2^{-(\la+r)}2^{-(\gam_2+\gam_3)r}
$$

Integrating by parts in the expression for $c_p(x)$ we therefore obtain
\begin{align*}
|c_{p}(x)|\lesssim \f{2^{-2r}}{p_1^{\gam_1}(2^{-r}p_2)^{\gam_2}(2^{-r}p_3)^{\gam_3}}\int_{\substack{|\xi_1|\sim1\\ |\xi_2|,|\xi_3|\lesssim 2^r}}2^{-(\la+r)}2^{-(\gam_2+\gam_3)r}d\xi
\lesssim\f{2^{-(\la+r)}}{\langle p\rangle^{|\gam|}}
\end{align*}

With this bound we calculate, for any $N\in\mathbb{N}$,
\begin{align*}
&\|\mu_2(x)\vpl(x)\sg_\rho^\beta(x)\eta^{(r,l)}_r(D)\phi_0\|_{L^q_x}\\
&\lesssim_N\sum_{p}\f{2^{-(\la+r)}}{\langle p\rangle^{N}}\|\int_{\xi}
e^{i((x_1 +2\pi p_1)\xi_1+(x_2+2\pi 2^{-r}p_2)\xi_2+(x_3+2\pi 2^{-r}p_3)\xi_3)}
(x_1\xi_2-x_2\xi_1)\hat{\phi}_0(\xi)d\xi \|_{L^q_x}\\
&\lesssim_N\sum_{p}\f{2^{-(\la+r)}}{\langle p\rangle^{N}}\|(x_1\dd_2-x_2\dd_1)(\phi_0(x_1+2\pi p_1,x_2+2^{-r}2\pi p_2,x_3+2^{-r}2\pi p_3))\|_{L^q_x}\\
&\lesssim_N\sum_{p}\f{2^{-(\la+r)}}{\langle p\rangle^{N}}(\|(x_1\dd_2-x_2\dd_1)\phi_0\|_{L^{q}_x}+2^{-r}|p|\|\na \phi_0\|_{L^{q}_x})\\
&\lesssim_N\sum_{p}\f{2^{-(\la+2r)}}{\langle p\rangle^{N-1}}(\|\Om_{1,2}\phi_0\|_{L^{q}_x}+\|\phi_0\|_{L^{q}_x})
\end{align*}
Choosing $N$ sufficiently large and summing over $p$ gives the desired result.
\end{proof}

As a consequence of this lemma we can bound certain trilinear terms as follows.
\begin{lemma}\label{SIL}
Let $m,j,k\in\Z$ and fix $M$ sufficiently large. Then for any (scalar) functions $\phi^{(1)}_j$, $\phi^{(2)}_k$, $\phi^{(3)}_m$ we have the following estimates:
\begin{enumerate}
\item If $j\lesssim k$, $j\lesssim m$ we have
\begin{align*}
&\|\phi^{(1)}_j\cdot\phi^{(2)}_k\cdot \phi^{(3)}_m\|_{1,2}\\
&\lesssim 2^{-j/M}2^{3k/2M}\|\phi^{(1)}_j\|_{\f{2M}{M-1},\infty}\|\LOR\pk^{(2)}\|_{\f{2M}{M-1},2M}(\|\phi^{(3)}_m\|_{\infty,2}+2^{j-m}\|L\phi^{(3)}_m\|_{\infty,2}+2^{-m}\|\dd_t\phi^{(3)}_m\|_{\infty,2})\\
&\quad+2^{-j/M}2^{3k\f{M-1}{4M}}\|\phi^{(1)}_j\|_{\f{2M}{M-1},\infty}\|\pk^{(2)}\|_{\f{2M}{M-1},\f{4M}{M-1}}(\|\phi^{(3)}_m\|_{\infty,2}+2^{j-m}\|L\phi^{(3)}_m\|_{\infty,2}+2^{-m}\|\dd_t\phi^{(3)}_m\|_{\infty,2})
\end{align*}

\item If $j\lesssim m$ we have
\begin{align*}
&\|\phi^{(1)}_j\cdot\phi^{(2)}_k\cdot \phi^{(3)}_m\|_{1,2}\\
&\lesssim 2^{-j/M}2^{3j/2M}\|\LOR\phi^{(1)}_j\|_{\f{2M}{M-1},2M}\|\pk^{(2)}\|_{\f{2M}{M-1},\infty}(\|\phi^{(3)}_m\|_{\infty,2}+2^{j-m}\|L\phi^{(3)}_m\|_{\infty,2}+2^{-m}\|\dd_t\phi^{(3)}_m\|_{\infty,2})\\
&\quad+2^{-j/M}2^{3j\f{M-1}{4M}}\|\phi^{(1)}_j\|_{\f{2M}{M-1},\f{4M}{M-1}}\|\pk^{(2)}\|_{\f{2M}{M-1},\infty}(\|\phi^{(3)}_m\|_{\infty,2}+2^{j-m}\|L\phi^{(3)}_m\|_{\infty,2}+2^{-m}\|\dd_t\phi^{(3)}_m\|_{\infty,2})
\end{align*}

\item If $j\lesssim k$ we have
\begin{align*}
&\|\phi^{(1)}_j\cdot\phi^{(2)}_k\cdot \phi^{(3)}_m\|_{1,2}\\
&\lesssim 2^{-j/M}2^{3k/2M}\|\phi^{(1)}_j\|_{\f{2M}{M-1},\infty}\\
&\quad\quad\quad\quad\enspace\cdot(\|\LOR\pk^{(2)}\|_{\f{2M}{M-1},2M}+2^{j-k}\|\LOR L\pk^{(2)}\|_{\f{2M}{M-1},2M}+2^{-k}\|\LOR\dd_t\pk^{(2)}\|_{\f{2M}{M-1},2M})\|\phi^{(3)}_m\|_{\infty,2}\\
&\quad+2^{-j/M}2^{3k\f{M-1}{4M}}\|\phi^{(1)}_j\|_{\f{2M}{M-1},\infty}\\
&\quad\quad\quad\quad\enspace\cdot(\|\pk^{(2)}\|_{\f{2M}{M-1},\f{4M}{M-1}}+2^{j-k}\| L\pk^{(2)}\|_{\f{2M}{M-1},\f{4M}{M-1}}+2^{-k}\|\dd_t\pk^{(2)}\|_{\f{2M}{M-1},\f{4M}{M-1}})\|\phi^{(3)}_m\|_{\infty,2}
\end{align*}

\item If $j\lesssim k$ we have
\begin{align*}
&\|\phi^{(1)}_j\cdot\phi^{(2)}_k\cdot \phi^{(3)}_m\|_{1,2}\\
&\lesssim 2^{-j/M}2^{3j/2M}\|\LOR\phi^{(1)}_j\|_{\f{2M}{M-1},2M}\\
&\quad\quad\quad\quad\quad\quad\quad\quad\quad\quad\quad\quad\cdot(\|\pk^{(2)}\|_{\f{2M}{M-1},\infty}+2^{j-k}\|L\pk^{(2)}\|_{\f{2M}{M-1},\infty}+2^{-k}\|\dd_t\pk^{(2)}\|_{\f{2M}{M-1},\infty})\|\phi^{(3)}_m\|_{\infty,2}\\
&\quad+2^{-j/M}2^{3j\f{M-1}{4M}}\|\phi^{(1)}_j\|_{\f{2M}{M-1},\f{4M}{M-1}}\\
&\quad\quad\quad\quad\quad\quad\quad\quad\quad\quad\quad\quad\cdot(\|\pk^{(2)}\|_{\f{2M}{M-1},\infty}+2^{j-k}\|L\pk^{(2)}\|_{\f{2M}{M-1},\infty}+2^{-k}\|\dd_t\pk^{(2)}\|_{\f{2M}{M-1},\infty})\|\phi^{(3)}_m\|_{\infty,2}
\end{align*}
\end{enumerate}
We note that since the above holds for arbitrary scalar functions, in our setting the same will hold for vector functions independent of the type of multiplication used (dot product, cross product,...) and the order of the terms. 
\end{lemma}

Before going into the proof, let us shed some light on the relevance of this lemma. It will be applied to terms of the form 
\begin{equation}\label{keyterm}
\sum_{j\leq k\leq-10}\|L\LOR (\phi_k\text{ }\na\phi_j\cdot\na\phi_{\sim0})\|_{1,2}
\end{equation}
which are beyond the reach of the standard Strichartz estimates. For example, when $L$ and $\LOR $ both fall on $\na\pj$ we use point 1 of the above lemma to find
\begin{align}
&\sum_{j\leq k\leq-10}\|\phi_k\text{ }L\LOR \na\phi_j\cdot\na\phi_{\sim0}\|_{1,2}\nonumber\\
&\lesssim\sum_{j\leq k\leq-10}2^{-j/M}2^{3k/2M}\|\LOR\pk\|_{\f{2M}{M-1},2M}\|L\LOR \na\phi_j\|_{\f{2M}{M-1},\infty}(\|\na\pso\|_{\infty,2}+2^{j}\|L\na\pso\|_{\infty,2}+\|\dd_t\na\pso\|_{\infty,2})\nonumber\\
&\hspace{4.6em}+2^{-j/M}2^{3k\f{M-1}{4M}}\|\pk\|_{\f{2M}{M-1},\f{4M}{M-1}}\|L\LOR \na\phi_j\|_{\f{2M}{M-1},\infty}(\|\na\pso\|_{\infty,2}+2^{j}\|L\na\pso\|_{\infty,2}+\|\dd_t\na\pso\|_{\infty,2})\nonumber\\
&\lesssim\sum_{j\leq k\leq-10}
2^{(\f{1}{2}-\f{1}{2M})(j-k)}C_0^3c_jc_kc_0\lesssim C_0^3\eps^2c_0\label{example}
\end{align}
where we were able to place $\pk$ into $S_k$ in both cases since it only appears in a non-standard Strichartz space when not accompanied by an angular derivative.

The same argument works for any combination in which $\LOR$ falls on $\pj$ or $\pso$, and $L$ on $\pk$ or $\pj$. If $\LOR$ falls on $\pk$, and $L$ still doesn't hit $\pso$, we obtain the same result using point 2 of the lemma. Points 3 and 4 are for when $L$ hits $\pso$ and $\LOR$ avoids or hits $\pk$ respectively.

Due to the non-local nature of our equation, frequency interactions of the type discussed will appear in many different guises, which is why we give the lemma in such generality.

\begin{proof}[Proof of Lemma \ref{SIL}]
We focus on point 1, noting the adaptations needed for the other cases at the end.

Using the notation introduced in Section \ref{preliminaries_section} we first split the term over regions where $|x|$ is large or small compared to the natural oscillations of $\phi^{(1)}_j$:
\begin{equation*}
\|\phi^{(1)}_j\cdot\phi_k^{(2)}\cdot \phi^{(3)}_m\|_{1,2}
\lesssim\underbrace{\|\vp_{<-j}(x)(\phi^{(1)}_j\cdot\phi_k^{(2)}\cdot \phi^{(3)}_m)\|_{1,2}}_{(A)}+\underbrace{\|\vp_{\geq -j}(x)(\phi^{(1)}_j\cdot\phi_k^{(2)}\cdot \phi^{(3)}_m)\|_{1,2}}_{(B)}
\end{equation*}
Starting with (A), we further split the norm depending on the size of $|t|$:
\begin{equation*}
(A)\lesssim\underbrace{\|\vp_{<-j}(x)\vp_{<-j}(t)(\phi^{(1)}_j\cdot\phi_k^{(2)}\cdot \phi^{(3)}_m)\|_{1,2}}_{(A.I)}+\underbrace{\|\vp_{<-j}(x)\vp_{\geq-j}(t)(\phi^{(1)}_j\cdot\phi_k^{(2)}\cdot \phi^{(3)}_m)\|_{1,2}}_{(A.II)}
\end{equation*}

To reduce notation, we will often omit the space/time cut-offs.
Starting with (A.I), we use H\"older in the time variable to obtain
\begin{align*}
(A.I)&\lesssim\|\phi^{(1)}_j\|_{2,\infty}\|\phi_k^{(2)}\|_{2,\infty}\|\phi^{(3)}_m\|_{\infty,2}
\lesssim2^{-j/M}\|\phi^{(1)}_j\|_{\f{2M}{M-1},\infty}\|\phi_k^{(2)}\|_{\f{2M}{M-1},\infty}\|\phi^{(3)}_m\|_{\infty,2}
\end{align*}
Using Bernstein's inequality on $\pk^{(2)}$ and the monotonicity of the angular Sobolev spaces we see that this term is as required.

We now study (A.II). To counteract the loss coming from the use of H\"older's inequality in time we use a trick that will come up frequently in the sequel: since $\phi^{(3)}_m$ has not yet been acted on by any Lorentz boost we can write
\begin{align}
\phi^{(3)}_m=&t^{-1}\D^{-1}\dd_nL_n\phi^{(3)}_m-\D^{-1}\dd_n[(t^{-1}x_n)\dd_t\phi^{(3)}_m]\label{F_decomposition}
\end{align}
with the implicit sum over $n=1,2,3$. 
A simple computation using that the spatial localisation passes through the Fourier multipliers up to exponentially decaying tails (since $m\gtrsim j$) yields
\begin{equation}\label{large_t_bound}
\|\vp_{<-j}(x)\vp_{\ko}(t)\phi^{(3)}_m\|_{\infty,2}\lesssim2^{-\ko-m}\|L\phi^{(3)}_m\|_{\infty,2}+2^{-\ko-j-m}\|\dd_t\phi^{(3)}_m\|_{\infty,2}
\end{equation}

Therefore
\begin{align*}
(A.II)&\lesssim\sum_{k_1\geq -j}\|\phi^{(1)}_j\|_{2,\infty}\|\phi_k^{(2)}\|_{2,\infty}\|\vp_{<-j}(x)\vp_{\ko}(t)\phi^{(3)}_m\|_{\infty,2}\\
&\lesssim\sum_{k_1\geq -j}2^{k_1/M}\|\phi^{(1)}_j\|_{\f{2M}{M-1},\infty}\|\phi^{(2)}_k\|_{\f{2M}{M-1},\infty}(2^{-\ko-m}\|L\phi^{(3)}_m\|_{\infty,2}+2^{-\ko-j-m}\|\dd_t\phi^{(3)}_m\|_{\infty,2})\\
&\lesssim2^{-j/M}\|\phi^{(1)}_j\|_{\f{2M}{M-1},\infty}\|\phi^{(2)}_k\|_{\f{2M}{M-1},\infty}(2^{j-m}\|L\phi^{(3)}_m\|_{\infty,2}+2^{-m}\|\dd_t\phi^{(3)}_m\|_{\infty,2})
\end{align*}
which is as required.

We now turn to
$$
(B)=\|\vp_{\geq -j}(x)(\phi^{(1)}_j\cdot\phi_k^{(2)}\cdot \phi^{(3)}_m)\|_{1,2}
$$ 
It is here that we need to invoke the angular multipliers. Write
\begin{align*}
(B)\leq&\sum_{k_1\geq-j}\sum_{\kt\in\Z}\underbrace{\|\vp_{\ko}(x)\vp_{\kt}(t)(\phi^{(1)}_j\cdot\phi_k^{(2)}\cdot \phi^{(3)}_m)\|_{1,2}}_{(B)_{\ko,\kt}}
\end{align*}
Fix $k_1\geq-j$. Let $(\sg^\beta_{-(j+k_1)/3})_{\beta\in\Ss_{j,k_1}}$ be a partition of unity on $\Sp^2$ as in \eqref{spatial_partition} (denoting $\Ss_{j,\ko}:=\Ss_{-(j+\ko)/3}$). Then for fixed $k_1$, $\kt$ we can split $(B)_{\ko,\kt}$ into a square-sum
\begin{align*}
(B)_{\ko,\kt}\lesssim&\left\|\left(\sum_{\beta\in\Ss_{j,k_1}}\|\sg^\beta_{-(j+\ko)/3}(x)(\phi^{(1)}_j\cdot\phi_k^{(2)}\cdot \phi^{(3)}_m)\|_{L^2_x(|x|\sim 2^{\ko})}^2\right)^\half\right\|_{L^1_t(|t|\sim 2^{\kt})}
\end{align*}

Now, for each $\beta\in\mathcal{S}_{j,\ko}$, use the Fourier multipliers introduced in (\ref{nbr_def}) to write
\begin{align}
\phi^{(2)}_k
=&\eta^\beta_{-(j+\ko)/3}(D)\phi^{(2)}_k
+\sum_{l=1}^N\sum_{-(j+\ko)/3\ll r\leq0}\eta^{(r,l)}_r(D)\phi_k^{(2)}\label{far_angle110}
\end{align}
(dropping the ``$+C_2$'' in the subscript of $\eta^\beta_{-(j+\ko)/3}$ to reduce notation).
Let's start with the first term, where the angular localisations in Fourier space and physical space are forced to be close. We find
\begin{align}
&\left\|\left(\sum_{\beta\in\Ss_{j,k_1}}\|\sg^\beta_{-(j+\ko)/3}(x)(\phi^{(1)}_j\cdot \eta^\beta_{-(j+\ko)/3}(D)\phi_k^{(2)}\cdot \phi^{(3)}_m)\|_{L^2_x(|x|\sim 2^{\ko})}^2\right)^\half\right\|_{L^1_t(|t|\sim 2^{\kt})}\nonumber\\
&\quad\quad\lesssim\left\|\left(\sum_{\beta\in\mathcal{S}_{j,k_1}}(\|\phi_j^{(1)}\|_\infty \|\eta^\beta_{-(j+\ko)/3}(D) \phi^{(2)}_k\|_\infty\|\sg^\beta_{-(j+\ko)/3}(x)\phi^{(3)}_m\|_2)^2\right)^\half\right\|_{L^1_t(|t|\sim 2^{\kt})}
\end{align}
We then use Bernstein's inequality on the middle term to benefit from the close angular localisation, and thereby bound the above by
\begin{align}
&\|\phi_j^{(1)}\|_{2,\infty}\left\|\left(\sum_{\beta\in\Ss_{j,k_1}}(2^{(3k-2(j+\ko)/3)(\f{M-1}{4M})}\|\eta^\beta_{-(j+\ko)/3}(D) \phi^{(2)}_k\|_\f{4M}{M-1}\|\sgb_{-(j+\ko)/3}(x)\phi^{(3)}_m\|_2)^2\right)^\half\right\|_2\nonumber\\
&\lesssim\|\phi_j^{(1)}\|_{2,\infty}\cdot 2^{(3k-2(j+\ko)/3)(\f{M-1}{4M})}\|\pk^{(2)}\|_{2,\f{4M}{M-1}}\cdot\|\phi^{(3)}_m\|_{\infty,2}\nonumber\\
&\lesssim2^{\kt/M}2^{(3k-2(j+\ko)/3)(\f{M-1}{4M})}\|\phi_j^{(1)}\|_{\f{2M}{M-1},\infty} \|\pk^{(2)}\|_{\f{2M}{M-1},\f{4M}{M-1}}\|\phi^{(3)}_m\|_{\infty,2}
\label{katie-1}
\end{align}
where we used that the operator $\eta^\beta_{-(j+\ko)/3}(D)$ is bounded and square summed over the $\sg^\beta_{-(j+\ko)/3}$. 

Now, fixing $M$ sufficiently large this gives an acceptable bound in the range $\kt\leq\ko$:
\begin{multline*}
    \sum_{\ko\geq-j}\sum_{\kt\leq\ko}\left\|\left(\sum_{\beta\in\Ss_{j,k_1}}\|\sg^\beta_{-(j+\ko)/3}(x)(\phi^{(1)}_j\cdot \eta^\beta_{-(j+\ko)/3}(D)\phi_k^{(2)}\cdot \phi^{(3)}_m)\|_{L^2_x(|x|\sim 2^{\ko})}^2\right)^\half\right\|_{L^1_t(|t|\sim 2^{\kt})}\\
    \lesssim 2^{-j/M}2^{3k\f{M-1}{4M}}\|\phi_j^{(1)}\|_{\f{2M}{M-1},\infty} \|\pk^{(2)}\|_{\f{2M}{M-1},\f{4M}{M-1}}\|\phi^{(3)}_m\|_{\infty,2}
\end{multline*}

For $\kt>\ko$, we obtain decay in $t$ via an estimate analogous to \eqref{large_t_bound}, 
\begin{equation}
\|\vp_{\ko}(x)\vp_{\kt}(t)\phi^{(3)}_m\|_{\infty,2}\lesssim2^{-\kt-m}\|L\phi^{(3)}_m\|_{\infty,2}+2^{\ko-\kt-m}\|\dd_t\phi^{(3)}_m\|_{\infty,2}\label{P_mF_bound0}
\end{equation}
and the desired result follows.

We now turn to the second term in (\ref{far_angle110}), the ``far-angle case". We use the angular separation estimate, Lemma \ref{angular_sep_lem}, followed by Bernstein's inequality and the Riesz estimate for angular derivatives (which only holds for finite exponents) to bound
\begin{align*}
\|\vp_{\ko}(x)\sgb_{-(j+k_1)/3}(x)\eta^{(r,l)}_{r}(D)\phi_k^{(2)}\|_{\infty}&\lesssim 2^{-(\ko+k+2r)}(\|\phi^{(2)}_k\|_\infty+\max_{i,j}\|\Om_{ij}\phi^{(2)}_k\|_{\infty})\\
&\lesssim 2^{-(\ko+k+2r)}2^{3k/2M}\|\LOR\pk^{(2)}\|_{2M}
\end{align*}
Therefore
\begin{align}
\sum_{l=1}^N\sum_{-(j+\ko)/3\ll r}\|\sg^\beta_{-(j+\ko)/3}(x)(\phi^{(1)}_j\cdot\eta^{(r,l)}_r(D)\phi_k^{(2)}\cdot\phi^{(3)}_m)\|_{L^2_x(|x|\sim 2^{\ko})}\nonumber\\
\lesssim 2^{-(\ko+k)}2^{2(j+\ko)/3}2^{3k/2M}\|\phi^{(1)}_j\|_{\infty}\|\LOR\pk^{(2)}\|_{2M}\|\sgb_{-(j+k_1)/3}(x)\phi^{(3)}_m\|_{2} \label{sum_eqn1}
\end{align}
Lastly the $L^1_{(|t|\sim 2^{\kt})}$ norm of the square-sum of (\ref{sum_eqn1}) over $\beta\in \Ss_{j,\ko}$ is bounded by
\begin{align}
2^{\kt/M}2^{-(\ko+k)}2^{2(j+\ko)/3}2^{3k/2M}\|\phi^{(1)}_j\|_{\f{2M}{M-1},\infty}\|\LOR\pk^{(2)}\|_{\f{2M}{M-1},2M}\|\phi^{(3)}_m\|_{\infty,2}\label{joanna-1}
\end{align}
Summed over $\kt\leq\ko$ and $k_1\geq -j$ this gives
\begin{equation*}
2^{-j/M}2^{j-k}2^{3k/2M}\|\phi^{(1)}_j\|_{\f{2M}{M-1},\infty}\|\LOR\pk^{(2)}\|_{\f{2M}{M-1},2M}\|\phi^{(3)}_m\|_{\infty,2}
\end{equation*}
which is acceptable since $j\lesssim k$. In the case of large $t$ we again apply \eqref{P_mF_bound0} before summing over $\kt>\ko$.

\bigskip

To prove point 2 of the Lemma, we perform the same argument but carry out the angular decomposition in Fourier space on $\pj^{(1)}$ instead of $\pk^{(2)}$. In the far-angle case we no longer gain a factor of $2^{j-k}$ so the restriction $j\lesssim k$ is not necessary.
For point 3, we do not change the angular decomposition but get the gain in $|t|^{-1}$ from $\pk^{(2)}$ rather than $\phi_m^{(3)}$, using the estimate
$$
\|\vp_{\ko}(x)\vp_{\kt}(t)\phi^{(2)}_k\|_{2,\infty}\lesssim2^{-\kt-k}\|L\phi^{(2)}_k\|_{2,\infty}+2^{\ko-\kt-k}\|\dd_t\phi^{(2)}_k\|_{2,\infty}$$
for $k+\ko\gtrsim0$.
For point 4 we use both of the adaptations described above.
\end{proof}

Due to some error terms which appear as a result of the commutation relations \eqref{commr1}-\eqref{commr3}, we will also need the following form of Lemma \ref{SIL}.
\begin{corollary}\label{SIC}
Let $K_j$ be a convolution operator given by a Schwarz kernel $k_j(x):=2^{3k}k(2^jx)$. Then for $m,r,s\gtrsim j$ it holds
\begin{enumerate}
\item \begin{align*}
&\|K_j(\phi^{(1)}_r\cdot\phi^{(2)}_s)\cdot\phi^{(3)}_m\|_{1,2}\\
&\lesssim 2^{-j/M}2^{3j\f{M-1}{4M}}2^{\f{21}{4M}(r-j)}\|\phi^{(1)}_r\|_{\f{2M}{M-1},\f{4M}{M-1}}\|\phi^{(2)}_s\|_{\f{2M}{M-1},\infty}(\|\phi^{(3)}_m\|_{\infty,2}+2^{j-m}\|L\phi^{(3)}_m\|_{\infty,2}+2^{-m}\|\dd_t\phi^{(3)}_m\|_{\infty,2})\\
&\quad+2^{-j/M}2^{3r/2M}2^{j-r}\|\LOR\pho_r\|_{\f{2M}{M-1},2M}\|\pht_s\|_{\f{2M}{M-1},\infty}(\|\phi^{(3)}_m\|_{\infty,2}+2^{j-m}\|L\phi^{(3)}_m\|_{\infty,2}+2^{-m}\|\dd_t\phi^{(3)}_m\|_{\infty,2})
\end{align*}
\item \begin{align*}
&\|K_j(\phi^{(1)}_r\cdot\phi^{(2)}_s)\cdot\phi^{(3)}_m\|_{1,2}\\
&\lesssim 2^{-j/M}2^{3j\f{M-1}{4M}}2^{\f{21}{4M}(r-j)}\|\phi^{(1)}_r\|_{\f{2M}{M-1},\f{4M}{M-1}}(\|\pht_s\|_{\f{2M}{M-1},\infty}+2^{j-s}\|L\pht_s\|_{\f{2M}{M-1},\infty}+2^{-b}\|\dd_t\pht_s\|_{\f{2M}{M-1},\infty})\|\phi^{(3)}_m\|_{\infty,2}\\
&\quad+2^{-j/M}2^{3r/2M}2^{j-r}\|\LOR\pho_r\|_{\f{2M}{M-1},2M}(\|\pht_s\|_{\f{2M}{M-1},\infty}+2^{j-s}\|L\pht_s\|_{\f{2M}{M-1},\infty}+2^{-s}\|\dd_t\pht_s\|_{\f{2M}{M-1},\infty})\|\phi^{(3)}_m\|_{\infty,2}
\end{align*}
\end{enumerate}
The important thing to note here is the gain in powers of $2^j$ rather than $2^r$ (up to a small amount of leakage).
\end{corollary}
The proof of this corollary relies on the following simple proposition, the main point of which is that angular localisation passes through convolution up to exponentially decaying tails.
\begin{proposition}\label{comm_prop}
    Let $1\leq q\leq p\leq\infty$. Let $K_j$ be as in the corollary. Then the following commutator estimates hold for any $N\in\mathbb{N}$.
    \begin{enumerate}
        \item Let $l>\ko+5$. Then 
        $
        \|\varphi_{\ko}\cdot K_j(\varphi_l F)\|_{L^p_x}\lesssim_N 2^{-(l+j)N}2^{3(\f{1}{q}-\f{1}{p})j}\|F\|_{L^q_x}
        $.
        \item $\|\varphi_{\ko}\cdot K_j(\varphi_{<\ko-5} F)\|_{L^p_x}\lesssim_N 2^{-(\ko+j)N}2^{3(\f{1}{q}-\f{1}{p})j}\|F\|_{L^q_x}$.
        \item Let $r\geq-(j+\ko)/3+C_1$. Then 
        $$
        \|\sg^{\beta}_{-\f{(j+\ko)}{3}}(x)\varphi_{\ko}(x)\cdot K_j(\eta^{(r,l)}_r(x)\widetilde{\vp}_{\ko}(x) F)\|_{L^p_x}\lesssim_N 2^{-(j+\ko+r)N}2^{3(\f{1}{q}-\f{1}{p})j}\|F\|_{L^p_x}.
        $$
    \end{enumerate}
\end{proposition}
\begin{proof}[Proof of Corollary \ref{SIC}]
We show only point (1), the adaptations for (2) being as in the previous proof. As in the proof of Lemma \ref{SIL} we decompose
    \begin{align*}              \|K_j(\pho_a\cdot\pht_b)\cdot\phth_m\|_{1,2}&\lesssim\|\cdot\|_{L^1_tL^2_x(|x|\lesssim 2^{-j})}+\|\cdot\|_{L^1_tL^2_x(|x|\gg 2^{-j})}=:(A)+(B)
    \end{align*}
Here (A) can be treated similarly to in the lemma, so we focus on (B).
Further decompose 
    \begin{align*}
        (B)\leq\sum_{\ko\gg -j}\sum_{\kt\in\Z}\|\cdot\|_{L^1_t L^2_x(|x|\sim 2^{\ko},|t|\sim 2^{\kt})}
    \end{align*}
Performing an angular decomposition in the physical variable as in the previous proof and moving the spatial localisation inside the convolution, we have
    \begin{align*}
        \|\cdot\|_{L^1_t L^2_x(|x|\sim 2^{\ko},|t|\sim 2^{\kt})}&\lesssim\left\|\left(\sum_{\beta\in\mathcal{S}_{j,\ko}}\left\|\sg^{\beta}_{-(j+\ko)/3}(x)\vp_{\ko}(x)\cdot K_j(\pho_a\cdot\pht_b)\cdot\phth_m\right\|_{L^2_x}^2\right)^\half\right\|_{L^1_t(|t|\sim 2^{\kt})}\\
        &\lesssim\left\|\left(\sum_\beta\left\|\sg^{\beta}_{-(j+\ko)/3}(x){\vp}_{\ko}(x)\cdot K_j(\vp_{[\ko-5,\ko+5]}\cdot\pho_a\cdot\pht_b)\cdot\phth_m\right\|_{L^2_x}^2\right)^\half\right\|_{L^1_t(|t|\sim 2^{\kt})}\\
        &\quad+\left\|\left(\sum_\beta\left\|\sg^{\beta}_{-(j+\ko)/3}(x){\vp}_{\ko}(x)\cdot K_j(\vp_{<\ko-5}\cdot\pho_a\cdot\pht_b)\cdot\phth_m\right\|_{L^2_x}^2\right)^\half\right\|_{L^1_t(|t|\sim 2^{\kt})}\\        &\quad+\left\|\left(\sum_\beta\left\|\sg^{\beta}_{-(j+\ko)/3}(x){\vp}_{\ko}(x)\cdot K_j(\vp_{>\ko+5}\cdot\pho_a\cdot\pht_b)\cdot\phth_m\right\|_{L^2_x}^2\right)^\half\right\|_{L^1_t(|t|\sim 2^{\kt})}
    \end{align*}
The second and third terms above are error terms bounded using Proposition \ref{comm_prop}. 
 For the first term we still have to exchange the angular localisation and the convolution. Denoting $\vp_{[\ko-5,\ko+5]}=:\widetilde{\vp}_{\ko}$, we have (again dropping the $``+C_2"$ notation)
\begin{align*}
    &\left\|\left(\sum_\beta\left\|\sg^{\beta}_{-(j+\ko)/3}(x){\vp}_{\ko}(x)\cdot K_j(\widetilde{\vp}_{\ko}\cdot\pho_a\cdot\pht_b)\cdot\phth_m\right\|_{L^2_x}^2\right)^\half\right\|_{L^1_t(|t|\sim 2^{\kt})}\nonumber\\
    &\lesssim \left\|\left(\sum_\beta\left\|\sg^{\beta}_{-(j+\ko)/3}(x){\vp}_{\ko}(x)\cdot K_j({\eta}^{\beta}_{-(j+\ko)/3}(x)\widetilde{\vp}_{\ko}(x)\cdot\pho_a\cdot\pht_b)\cdot\phth_m\right\|_{L^2_x}^2\right)^\half\right\|_{L^1_t(|t|\sim 2^{\kt})}\\
    &\quad+\sum_{\substack{r\gg-(j+\ko)/3\\l=1,\ldots,N}}\left\|\left(\sum_\beta\left\|\sg^{\beta}_{-(j+\ko)/3}(x){\vp}_{\ko}(x)\cdot K_j(\eta^{(r,l)}_r(x)\widetilde{\vp}_{\ko}(x)\cdot\pho_a\cdot\pht_b)\cdot\phth_m\right\|_{L^2_x}^2\right)^\half\right\|_{L^1_t(|t|\sim 2^{\kt})}
\end{align*}
Note that here the multipliers $\eta$ are acting on the space variable. 
 Here the second term is an error and again treated using Proposition \ref{comm_prop}, while for the main term we use the close/far angle decomposition from the previous proof on $\pho$ to reduce to 
\begin{align}
    &\left\|\left(\sum_\beta\left\|\sg^\beta_{-(j+\ko)/3}(x)\vp_{\ko}(x)K_j(\eta^{\beta}_{-(j+\ko)/3}(x)\widetilde{\vp}_{\ko}(x)\cdot\tilde{\eta}^{\beta}_{-(j+\ko)/3}(D)\pho_a\cdot\pht_b)\cdot\phth_m\right\|_{L^2_x}^2\right)^\half\right\|_{L^1_t}\tag{C-A}\label{C-A}\\
    &\quad+\sum_{\substack{r\gg-(j+\ko)/3\\l=1,\ldots,N}}\left\|\left(\sum_\beta\left\|\sg^\beta_{-(j+\ko)/3}(x)\vp_{\ko}(x)K_j(\eta^{\beta}_{-(j+\ko)/3}(x)\widetilde{\vp}_{\ko}(x)\cdot\tilde{\eta}^{(r,l)}_r(D)\pho_a\cdot\pht_b)\cdot\phth_m\right\|_{L^2_x}^2\right)^\half\right\|_{L^1_t}\tag{F-A}\label{F-A}
\end{align}
where $\tilde{\eta}^{\beta}_{-(j+\ko)/3}+\sum_{r,l}\tilde{\eta}^{(r,l)}_r$ is an angular decomposition as in Proposition \ref{partition_prop} with $\eta^{\beta}_{-(j+\ko)/3}$ playing the role of $\sg^\beta_{-(j+\ko)/3}$. The restriction to $|t|\sim 2^{\kt}$ is implicit.
For the close-angle term we proceed almost as in the proof of Lemma \ref{SIL}, however we must be careful applying Bernstein's inequality. We first apply Bernstein in the form of Young's inequality on the convolution $K_j$ then directly on the term $\tilde{\eta}^{\beta}_{-(j+\ko)/3}(D)\pho_a$ to find
\begin{align*}
    \eqref{C-A}&\lesssim\left\|\left(\sum_{\beta}(\|\sg^{\beta}_{-(j+\ko)/3}\phth_m\|_2\,2^{3j\f{M-8}{4M}}\|\tilde{\eta}^{\beta}_{-(j+\ko)/3}(D)\pho_a\cdot\pht_b\|_{\f{4M}{M-8}})^2\right)^\half\right\|_{L^1_t}\\
    &\lesssim2^{\kt/M}2^{3j\f{M-8}{4M}}2^{(3a-2(j+\ko)/3)\f{7}{4M}}\|\pho_a\|_{\f{2M}{M-1},\f{4M}{M-1}} \|\pht_b\|_{\f{2M}{M-1},\infty}\|\phth_m\|_{\infty,2}
\end{align*}
Summing first over $\kt\leq\ko$ this may be summed over $\ko\gg-j$ (due to our careful application of Bernstein) yielding
$$
2^{-j/M}2^{3j\f{M-1}{4M}}2^{3(a-j)\f{7}{4M}}\|\pho_a\|_{\f{2M}{M-1},\f{4M}{M-1}} \|\pht_b\|_{\f{2M}{M-1},\infty}\|\phth_m\|_{\infty,2}
$$
as required. We make the usual adaptation involving \eqref{P_mF_bound0} in the case $\kt>\ko$.

For the far-angle term, we apply the angular separation lemma, Lemma \ref{angular_sep_lem}, to bound
\begin{align*}
    \eqref{F-A}&\lesssim\sum_{r,l}\left\|\left(\sum_\beta(\|\sg^\beta_{-(j+\ko)/3}\phth_m\|_2
    \|\eta^{\beta}_{-(j+\ko)/3}(x)\widetilde{\vp}_{\ko}(x)\cdot\tilde{\eta}^{(r,l)}_r(D)\pho_a\|_{\infty}\|\pht_b\|_{\infty})^2\right)^\half\right\|_{L^1_t}\\
    &\lesssim 2^{\kt/M}\sum_{r,l}2^{-(a+\ko+2r)}2^{3a/2M}\|\LOR\pho_a\|_{\f{2M}{M-1},2M}\|\pht_b\|_{\f{2M}{M-1},\infty}\|\phth_m\|_{\infty,2}\\
    &\lesssim 2^{\kt/M}2^{-(a+\ko)}2^{2(j+\ko)/3}2^{3a/2M}\|\LOR\pho_a\|_{\f{2M}{M-1},2M}\|\pht_b\|_{\f{2M}{M-1},\infty}\|\phth_m\|_{\infty,2}
\end{align*}
which yields
$$
2^{-j/M}2^{j-a}2^{3a/2M}\|\LOR\pho_a\|_{\f{2M}{M-1},2M}\|\pht_b\|_{\f{2M}{M-1},\infty}\|\phth_m\|_{\infty,2}
$$
when summed over $\kt\leq\ko$, $\ko\gg-j$, and
$$
2^{-j/M}2^{j-a}2^{3a/2M}\|\LOR\pho_a\|_{\f{2M}{M-1},2M}\|\pht_b\|_{\f{2M}{M-1},\infty}(2^{j-m}\|L\phth_m\|_{\infty,2}+2^{-m}\|\dd_t\phth_m\|_{\infty,2})
$$
for $\kt>\ko$ after applying \eqref{P_mF_bound0}.
\end{proof}

\bigskip
\subsection{Showing that $P_0(HWM_1(\phi))=error$}\label{HWM1_chapter}
\hfill\\
The goal of this subsection is the following
\begin{proposition}\label{HWM1_prop}
We have
$$
P_0\bl(HWM_1(\phi))=error
$$
\end{proposition}
\bigskip
First note the following Moser-type estimate:
\begin{lemma}\label{Moser}
Let $g:\R^3\rightarrow \R^3$ be a smooth function with bounded derivatives up to order 4. Then for any $k\in\Z$ it holds
$$
\max_{(p,q)\in\mathcal{Q}}2^{(\f{1}{p}+\f{3}{q})k}\|P_k\LOR^{1-\dl(p,q)} L g(\phi)\|_{p,q}\lesssim_{\mathcal{Q}} \|\phi^L\|_S(1+\|\phi^L\|_S)^3
$$
for $\mathcal{Q}$ as in the definition of $S$. Moreover,
$$
\max_{(p,q)\in\mathcal{Q}}2^{(\f{1}{p}+\f{3}{q}-1)k}\|P_k\LOR^{1-\dl(p,q)} \dd_t g(\phi)\|_{p,q}\lesssim_{\mathcal{Q}}\|\phi\|_S(1+\|\phi\|_S)^3
$$
\end{lemma}
\begin{proof}
We first recall the standard form of this inequality, requiring only bounded derivatives up to second order:
    \begin{equation}\label{no_fields}
    \|P_k g(\phi)\|_{p,q}\simeq\tk\|P_k(\na\phi\cdot g'(\phi))\|_{p,q}\lesssim 2^{-(\f{1}{p}+\f{3}{q})k}\|\phi\|_S(1+\|\phi\|_S)
    \end{equation}
To incorporate the vector fields we apply the chain rule to find (omitting the angular derivative for a radially admissible pair)
    \begin{align}
    P_k\Om_{ij}L_n g(\phi)&=P_k(\omij L_n\phi \cdot g'(\phi))+P_k(L_n\phi\cdot\omij\phi\cdot g''(\phi))\label{big_expand}
    \end{align}
Then since $g'$ satisfies the hypotheses for \eqref{no_fields} we have
    \begin{align*}
        \|P_k(\omij L_n\phi\cdot g'(\phi))\|_{p,q}&\lesssim\|P_{<k-10}\omij L_n\phi\|_{\infty,\infty}\|P_{\sim k}g'(\phi)\|_{p,q}+\|P_{\geq k-10}\omij L_n\phi\|_{p,q}\| g'\|_{\infty}\\
        &\lesssim\|\phi^L\|_S\cdot2^{-(\f{1}{p}+\f{3}{q})k}\|\phi\|_S(1+\|\phi\|_S)+2^{-(\f{1}{p}+\f{3}{q})k}\|\phi^L\|_{S}
    \end{align*}
which is acceptable. A similar argument works for the remaining term in \eqref{big_expand} and the second estimate in the statement can be proved similarly.
\end{proof}

From this lemma we can deduce the following result which effectively allows us to ignore the projection when estimating $P_k(\Phip(\dhalf\phi))$:
\begin{lemma}\label{projection_lemma}
Let $\phi$ such that $\|\phi^L\|_S\leq1$. Then there exists a constant $C_{\mathcal{Q}}>0$ such that
$$
\max_{(p,q)\in\mathcal{Q}}2^{(\f{1}{p}+\f{3}{q}-1)k}\|\LOR^{1-\dl(p,q)} LP_k(\Pi_{\phi^{\perp}}(\dhalf\phi))\|_{p,q}\lesssim_{\mathcal{Q}} \sum_{\ko\in\Z}2^{-C_{\mathcal{Q}}|k-\ko|}\|P_{\ko}\phi^L\|_{\sko}
$$
Under the same conditions we also have
$$
\max_{(p,q)\in\mathcal{Q}}2^{(\f{1}{p}+\f{3}{q}-2)k}\|\dd_t\LOR^{1-\dl(p,q)} P_k(\Phip(\dhalf\phi))\|_{p,q}\lesssim_{\mathcal{Q}} \sum_{\ko\in\Z}2^{-C_{\mathcal{Q}}|k-\ko|}\|\phi_{\ko}\|_{\sko}
$$
\end{lemma}


Our final preparation for the proof of Proposition \ref{HWM1_prop} is the following lemma, a variant of Lemma 4.3 \cite{Tao_1}, which allows us to apply the geometric identity \eqref{geometric_identity_4} in a more general setting:
\begin{lemma}\label{moving_loc_lemma}
Let $r\in\Z$, $p,q\geq1$ with $p^{-1}=p_1^{-1}+p_2^{-1}$ and $q^{-1}=q_1^{-1}+q_2^{-1}$. 
It holds
\begin{equation*}
u(x)\cdot P_rv(x)-P_r(u\cdot v)(x)=2^{-r}\int_y\int_{\tht=0}^1(\check{\chi}_r(y)(2^ry)^T)\na u(x-\tht y) v(x-y)dyd\tht
\end{equation*}
from which
\begin{align*}
\|u\cdot P_rv-P_r(u\cdot v)\|_{q}\lesssim2^{-r}\|\na u\|_{q_1}\|v\|_{q_2}
\end{align*}
and more generally
\begin{align*}
    \|\LOR L[u\cdot P_r v-P_r(u\cdot v)]\|_q\lesssim 2^{-r}\|\LOR\na_{t,x}u^L\|_{q_1}\|\LOR v^L\|_{q_2}+2^{-2r}\|\LOR\na u\|_{q_1}\|\LOR\dd_t v\|_{q_2}
\end{align*}
These statements also hold for $\Pp_r$ as in Section \ref{ang_com_chap}.
\end{lemma}

Henceforth we will use the following shorthand, adding to that introduced in \eqref{notation}:
\begin{gather}
    \begin{aligned}
        &(2+,\infty):=(\f{2M}{M-1},\infty), && \|\LOR P_k\phi^L\|_{2+,\infty}\lesssim 2^{-(\half-\f{1}{2M})k}\|P_k\phi^L\|_{S_k}\\
        &(2+,4+):=(\f{2M}{M-1},\f{4M}{M-1}), && \|\LOR P_k\phi^L\|_{2+,4+}\lesssim 2^{-(\half-\f{1}{2M})k}2^{-3\f{M-1}{4M}k}\|P_k\phi^L\|_{S_k}
    \end{aligned}
\end{gather}

\bigskip

\begin{proof}[Proof of Proposition \ref{HWM1_prop}]
First note that it is sufficient to find some $\dl>\la>0$ such that
\begin{align}
\|\LOR  P_0L[P_m(\Phip(\dhalf\phi))\text{ }(\phi\cdot\dhalf\phi)]\|_{1,2}\lesssim C_0^3\eps^22^{-\dl|m|}\sum_{k\in\Z}2^{-\la|k-m|}c_k\label{sufficient0}
\end{align}
for every $m\in\Z$, provided we then fix $\sg<\la$.

We start by studying \eqref{sufficient0} with $m>-10$. Further decompose
\begin{align}
&\|\LOR  P_0L[P_{m}(\Phip(\dhalf\phi))\text{ }(\phi\cdot\dhalf\phi)]\|_{1,2}\nonumber\\
&\lesssim\|\LOR  L[P_{m}(\Phip(\dhalf\phi))\text{ }\sum_{k\in\Z}(\phi_k\cdot\dhalf\phi_{<k+10})]\|_{1,2}\label{mgmt1}\\
&\quad+\|\LOR  L[P_{m}(\Phip(\dhalf\phi))\text{ }\sum_{j\in\Z}(\phi_{\leq j-10}\cdot\dhalf\phi_j)]\|_{1,2}\label{mgmt2}
\end{align}

We first study \eqref{mgmt1}. For the sum over $k\geq-10$ we use Lemma \ref{projection_lemma} to see that
\begin{align*}
&\sum_{k\geq-10}\|\LOR  L[P_{m}(\Phip(\dhalf\phi))\text{ }(\phi_k\cdot\dhalf\phi_{<k+10})]\|_{1,2}\\
&\lesssim \|\LOR  L(P_{m}(\Phip(\dhalf\phi)))\|_{\infty-,2+}\sum_{k\geq-10}\|\LOR L \pk\|_{2+,\infty-}\|\LOR L\dhalf \phi_{< k+10}\|_{2+,\infty-}\\
&\lesssim 2^{-(\half-\f{2}{M})m}C_0^3\eps^2\sum_{\ko}2^{-\la|m-\ko|}c_{\ko}
\end{align*}
for some $\la>0$.

The case $k<-10$ is handled by a direct application of Lemma \ref{SIL}. For example if $\LOR $ and $L $ both fall on $\phi_k$ we have by point 2 of said lemma that
\begin{align*}
&\sum_{k<-10}\|P_{m}(\Phip(\dhalf\phi))\text{ }(\LOR L \phi_k\cdot\dhalf\phi_{<k+10})\|_{1,2}\\
&\lesssim\sum_{k<-10}\sum_{j<k+10}2^{(\half-\f{1}{2M})(j-k)}C_0^2c_jc_k(\|P_m(\Phip(\dhalf\phi))\|_{\infty,2}\\
&\quad\quad\quad\quad\quad+2^{j-m}\|LP_m(\Phip(\dhalf\phi))\|_{\infty,2}+2^{-m}\|\dd_tP_m(\Phip(\dhalf\phi))\|_{\infty,2})
\end{align*}
which, again thanks to Lemma \ref{projection_lemma}, is bounded by
$$
C_0^3\eps^22^{-m/2}\sum_{k_1}2^{-\la|m-k_1|}c_{k_1}
$$
for some $\la>0$. When $\LOR $ and $L$ distribute in other combinations, we apply the other parts of Lemma \ref{SIL}.

We now turn to \eqref{mgmt2}, in which there are no derivatives falling on the lowest frequency factor. To handle this delicate situation we use the geometric relation \eqref{geometric_identity_4}.
As an example we consider only the case where $L$ and $\Om$ both fall on the product $\phi_{\leq j-10}\cdot\dphi_j$. This case presents the most technical difficulties from interchanging $L$ with the nonlocal derivative $\dhalf$ and the frequency projections. We have 
\begin{align}
&\|P_m(\Phip(\dhalf\phi))\text{ }\sum_{j\in\Z}\LOR L(\phi_{\leq j-10}\cdot\dhalf\phi_j)\|_{1,2}\nonumber\\
&\leq\sum_{j\in\Z}\|P_m(\Phip(\dhalf\phi))\text{ }\LOR L(\phi_{\leq j-10}\cdot\dhalf\phi_j-\dhalf(\phi_{\leq j-10}\cdot\pj))\|_{1,2}\label{mgmt2a}\\
&\hspace{2em}+\|P_m(\Phip(\dhalf\phi))\text{ }\LOR L\dhalf(\phi_{\leq j-10}\cdot\phi_j-P_j(\phi_{\leq j-10}\cdot\phi_{>j-10}))\|_{1,2}\label{mgmt2b}\\
&\hspace{2em}+\|P_m(\Phip(\dhalf\phi))\text{ }\LOR L\dhalf P_j(\phi_{\leq j-10}\cdot\phi_{>j-10})\|_{1,2}\label{mgmt2c}
\end{align}
The idea for handling these terms is that the first two effectively see a derivative moved onto the low frequency factor, and for the third we can apply \eqref{geometric_identity_4}. We start by rewriting \eqref{mgmt2a} as
$$
\sum_{j\in\Z}\sum_{k\leq j-10}\|P_m(\Phip(\dhalf\phi))\text{ }\LOR L\Lc_k(\pk,\pj)\|_{1,2}
$$
which is similar to \eqref{mgmt1}. Indeed when $j\geq-10$ we can bound this using the same estimates as for \eqref{mgmt1} combined with identities \eqref{L_LL}-\eqref{om_exp} and Lemma \ref{useful_lemma}. In the case $j<-10$ we need a small adaptation before applying Lemma \ref{SIL} due to the nonlocal operator $\Lc_k$. As in \eqref{FS2} we expand the corresponding multiplier $m_k$ as a Fourier series to write
\begin{align*}
\Lc_k(\pk,\pj)=\sum_{a,b}c^{(k)}_{a,b}\pk(x-2^{-k}a)\cdot\pj(x-2^{-j}b)
\end{align*}
Since we are applying Lemma \ref{SIL} we must then pay attention to where the vector derivatives fall on this expression. For example if both fall on the low frequency term $\pk$ we have to bound
\begin{align}
    \sum_{\substack{j\leq-10\\k\leq j-10}}\sum_{a,b}|c^{(k)}_{a,b}|\|P_m(\Phip(\dhalf\phi))\ \LOR L(\pk(x-2^{-k}a))\cdot\pj(x-2^{-j}b)\|_{1,2}\label{have_to}
\end{align}
Then by point 1 of the Lemma we have
\begin{align*}
    &\|P_m(\Phip(\dhalf\phi))\ \LOR L(\pk(x-2^{-k}a))\cdot\pj(x-2^{-j}b)\|_{1,2}\\
    &\lesssim2^{-k/M}\|\LOR L(\pk(x-2^{-k}a))\|_{2+,\infty}(2^{3j/2M}\|\LOR(\pj(x-2^{-j}b))\|_{2+,\infty-}+2^{3j\f{M-1}{4M}}\|\pj\|_{2+,4+})\\
    &\hspace{25em}\cdot(\|P_m(\Phip(\dhalf\phi))\|_{\infty,2}+\ldots)\\
    &\lesssim 2^{-k}2^{(\half-\f{1}{2M})(k-j)}\langle a \rangle^2\langle b \rangle ^2C_0^3c_jc_k\cdot 2^{-m/2}\sum_{\ko}2^{-\la|m-\ko|}c_{\ko}
\end{align*}
where we used \eqref{translation_bound} to achieve the final line. This is acceptable when summed as in \eqref{have_to}.

For \eqref{mgmt2b} we first consider $j>-10$, using the commutation relations $[\Om,L]\simeq L$, $[L,\dhalf]\simeq R\dd_t$ and $[\Om,R]\simeq R$ to find
\begin{align}
&\sum_{j>-10}\|P_m(\Phip(\dhalf\phi))\text{ }\LOR L\dhalf(\phi_{\leq j-10}\cdot\phi_j-P_j(\phi_{\leq j-10}\cdot\phi_{>j-10}))\|_{1,2}\nonumber\\
&\lesssim\sum_{j>-10}\|P_m(\Phip(\dhalf\phi))\text{ }\dhalf \LOR  L(\phi_{\leq j-10}\cdot \phi_j-P_j(\phi_{\leq j-10}\cdot\phi_{>j-10}))\|_{1,2}\label{pp'}\\
&\quad\quad\quad+\|P_m(\Phip(\dhalf\phi))\text{ }R\LOR \dd_t(\phi_{\leq j-10}\cdot\phi_j-P_j(\phi_{\leq j-10}\cdot\phi_{>j-10}))\|_{1,2}\label{ppp'}
\end{align}
These terms can both be handled using Lemma \ref{moving_loc_lemma}. For \eqref{pp'} we have
\begin{align*}
    \eqref{pp'}&\lesssim\sum_{j>-10}\|P_m(\Phip(\dhalf\phi))\|_{\infty-,2+}\|\dhalf\LOR L(\phi_{\leq j-10}\cdot \phi_j-P_j(\phi_{\leq j-10}\cdot\phi_{\sim j}))\|_{\f{M}{M-1},M}\\
    &\lesssim\sum_{j>-10}\|P_m(\Phip(\dhalf\phi))\|_{\infty-,2+}\cdot 2^j(2^{-j}\|\LOR\na_{t,x} L\phi_{\leq j-10}\|_{2+,\infty-}\|\LOR L\phi_{\sim j}\|_{2+,\infty-}\\
    &\hspace{21em}+2^{-2j}\|\LOR\na\phi_{\leq j-10}\|_{2+,\infty-}\|\LOR\dd_t\phi_{\sim j}\|_{2+,\infty-})\\
    &\lesssim\sum_{j>-10}2^{-(\half-\f{2}{M})m}\left(\sum_{\ko}2^{-\la|m-\ko|}C_0c_{\ko}\right)\cdot 2^{-2j/M}C_0^2\eps^2
\end{align*}
which is acceptable. \eqref{ppp'} is similar.

For \eqref{mgmt2b} with $j\leq-10$ we will use Lemma \ref{SIL} but must incorporate the cancellation structure via Lemma \ref{moving_loc_lemma}. We write, for $\Phi_j(\xi):=|2^{-j}\xi|\tilde{\chi}_0(2^{-j}\xi)$, $\tilde{\chi}_0$ as in Section \ref{notation_section},
\begin{multline*}
\dhalf(\phi_{\leq j-10}\cdot\phi_j-P_j(\phi_{\leq j-10}\cdot\phi_{>j-10}))\\
=\int_{z,y,\tht}\check{\Phi}_j(y)(2^jz\check{\chi}_j(z))\na\phi_{\leq j-10}(x-y-\tht z)\cdot\phi_{\sim j}(x-y-z)dyd\tht dz
\end{multline*}
Thus by Lemma \ref{SIL}, using the notation $\LOR_x$ and $L_x$ to emphasise that these fields act with respect to the $x$ variable only, we have
\begin{align*}
&\sum_{j\leq-10}\|P_m(\Phip(\dhalf\phi))\ \LOR L\dhalf(\phi_{\leq j-10}\cdot\phi_j-P_j(\phi_{\leq j-10}\cdot\phi_{>j-10}))\|_{1,2}\\
&\lesssim\sum_{j\leq-10}
\int_{z,y,\tht}dyd\tht dz|\check{\Phi}_j(y)(2^jz\check{\chi}_j(z))|\\
&\quad\quad\quad\quad\quad\cdot\|P_m(\Phip(\dhalf\phi))(x)\ \LOR_x L_x[\na \phi_{\leq j-10}(x-y-\tht z)]\cdot\phi_{\sim j}(x-y-z)\|_{L^1_tL^2_x}\\
&\quad+\text{similar terms}\\
&\lesssim\sum_{\substack{j\leq-10\\k\leq j-10}}\int_{z,y,\tht}dyd\tht dz|\check{\Phi}_j(y)(2^jz\check{\chi}_j(z))|\cdot2^{-k/M}\|\LOR_x L_x(\na \pk(x-y-\tht z))\|_{2+,\infty}\\
&\quad\quad\cdot(2^{3j/2M}\|\LOR_x(\phi_{\sim j}(x-y-z))\|_{2+,\infty-}+2^{3j\f{M-1}{4M}}\|\phi_{\sim j}\|_{2+,4+})(\|P_m(\Phip(\dhalf\phi))\|_{\infty,2}+\ldots)\\
&\quad+\text{similar terms}
\end{align*}
Then using \eqref{translation_bound} we bound this by
\begin{align*}
    &\sum_{\substack{j\leq-10\\k\leq j-10}}2^{-m/2}2^{(\half-\f{1}{2M})(k-j)}C_0^3c_jc_k\left(\sum_{\ko}2^{-\la|m-\ko|}c_{\ko}\right)\\
    &\quad\quad\quad\quad\cdot\int_{z,y,\tht}|\check{\Phi}_j(y)(2^jz\check{\chi}_j(z))|\langle 2^k(y+\tht z)\rangle^2\langle 2^j(y+z)\rangle^2dyd\tht dz
\end{align*}
Thanks to the scaling of $\check{\Phi}_j$ and $\check{\chi}_j$, both of which are rapidly decaying, we see that the integral above is $O(1)$ and this term is acceptable.

To complete the case $m>-10$ we need to study \eqref{mgmt2c}. Applying \eqref{geometric_identity_4} and commuting $\LOR  L$ through $\dhalf P_j$ we have
\begin{align}
&\sum_{j\in\Z}\|P_m(\Phip(\dhalf\phi))\text{ }\LOR  L \dhalf P_j(\phi_{\leq j-10}\cdot\phi_{>j-10})\|_{1,2}\nonumber\\
&\lesssim\sum_{j\in\Z}\|P_m(\Phip(\dhalf\phi))\text{ } \dhalf P_j\LOR  L(\phi_{> j-10}\cdot\phi_{>j-10})\|_{1,2}\label{oooo1}\\
&\quad\quad\quad+\|P_m(\Phip(\dhalf\phi))\text{  }RP_j\LOR  \dd_t(\phi_{> j-10}\cdot\phi_{>j-10})\|_{1,2}\label{tttt2}\\
&\quad\quad\quad+\|P_m(\Phip(\dhalf\phi))\text{ }\dhalf (2^{-j}\dd_t\Pp_j)\LOR (\phi_{> j-10}\cdot\phi_{>j-10})\|_{1,2}\label{tttt3}
\end{align}
This is easiest to handle when $j\geq-10$. For instance the sum of \eqref{oooo1} over $j\geq-10$ is bounded by
\begin{align*}
\sum_{j\geq-10}\|P_m(\Phip(\dhalf\phi))\|_{\infty-,2+}\cdot 2^j\|\LOR  L\phi_{>j-10}\|_{2+,\infty-}\|\LOR  L\phi_{>j-10}\|_{2+,\infty-}
\end{align*}
which is fine. The commutator terms \eqref{tttt2} and \eqref{tttt3} correspond to $\hhh$ interactions and can be treated like \eqref{term3}. 

For $j<-10$ we use Corollary \ref{SIC} with $K_j=2^{-j}\dhalf P_j$. 
For \eqref{oooo1} we have
\begin{align*}
&\sum_{j<-10}\| P_m(\Phip(\dhalf\phi))\text{ } \dhalf P_j\LOR  L(\phi_{> j-10}\cdot\phi_{>j-10})\|_{1,2}\\
&\lesssim\sum_{j<-10}\| P_m(\Phip(\dhalf\phi))\text{ } \dhalf P_j(  L\phi_{> j-10}\cdot\LOR L\phi_{>j-10})\|_{1,2}\\
&\lesssim\sum_{\substack{j<-10\\r,s>j-10}}2^j2^{-j/M}\left(2^{3j\f{M-1}{4M}}2^{\f{21}{4M}(r-j)}\|L\phi_{\sim r}\|_{2+,4+}+2^{3r/2M}2^{j-r}\|\LOR L\phi_{\sim r}\|_{2+,\infty-}\right)\\
&\hspace{14em}\cdot\|\LOR L\phi_{\sim s}\|_{2+,\infty}(\|P_m(\Phip(\dhalf\phi))\|_{\infty,2}+\ldots)\\
&\lesssim C_0^3\eps2^{-m/2}\left(\sum_{\substack{j<-10\\r>j-10}}(2^{(\half-\f{1}{2M}+3\f{M-1}{4M}-\f{21}{4M})(j-r)}+2^{(\half-\f{1}{2M}+1-\f{3}{2M})(j-r)})c_r\right)\left(\sum_{\ko}2^{-\la|m-\ko|}c_{\ko}\right)
\end{align*}
which is acceptable. \eqref{tttt2} and \eqref{tttt3} can be treated in the same way, using the additional information that $P_j$ localises the two factors of $\phi_{>j-10}$ to comparable frequencies ($r\sim s$) in order to handle the high frequency time derivative which appears.
This completes the case $m>-10$.

\bigskip

The case $m\leq-10$ is actually easier to handle and we do not need to invoke Lemma \ref{SIL}, since the geometry rules out any $\llh$ interactions. When the lone factor of $\phi$ appears at high frequency ($\geq 2^{-10}$), we refer to \eqref{term1} and \eqref{term2} of Proposition \ref{removing_trivial_terms} for the cases when $\dphi$ appears at high or low frequency respectively. It thus remains to study the case when $\phi$ is at very low frequency. Here we have
\begin{align}
&\|P_0\LOR  L[P_{m}(\Phip(\dhalf\phi))\text{ }(\phi_{\leq-10}\cdot\dhalf\phi_{>-10})]\|_{1,2}\\
&\leq\|P_0\LOR  L[P_{m}(\Phip(\dhalf\phi))\text{ }(\phi_{\leq-10}\cdot\dhalf\phi_{>-10}-\dhalf(\phi_{\leq-10}\cdot\phi_{>-10}))]\|_{1,2}\nonumber\\
&\quad+\|P_0\LOR  L [P_{m}(\Phip(\dhalf\phi))\text{ }\dhalf(\phi_{\leq-10}\cdot\phi_{>-10})]\|_{1,2}\label{TT4}
\end{align}
The first term above is of the form
\begin{align*}
\sum_{j\leq-10}\sum_{k>-10}\|P_0\LOR L[P_{m}(\Phip(\dhalf\phi))\text{ }\Lc_j(\pj,\pk))]\|_{1,2}
\end{align*}
which can be handled like \eqref{term2} from Proposition \ref{removing_trivial_terms} (or directly when $|j-k|\sim 0$).

For the second term in \eqref{TT4} we use \eqref{geometric_identity_4} to replace the low frequency term with a high one. We may also insert a projection $\tilde{P}_0$ before the $\dhalf$ since $m$ is very small and the whole term is restricted to frequency $\sim 2^0$.  We thus bound
\begin{align*}
&\|P_0\LOR  L [P_{m}(\Phip(\dhalf\phi))\text{ }\dhalf(\phi_{\leq-10}\cdot\phi_{>-10})]\|_{1,2}\\
&\lesssim\|\LOR  L P_{m}(\Phip(\dhalf\phi))\text{ }\dhalf\LOR  L(\phi_{>-10}\cdot\phi_{>-10})\|_{1,2}\\
&\quad+\|\LOR P_{m}(\Phip(\dhalf\phi))\text{ }\tilde{P}_0R(\dd_t\phi_{>-10}\cdot\LOR\phi_{>-10})]\|_{1,2}\\
&\quad+\|\LOR P_{m}(\Phip(\dhalf\phi))\text{ }\tilde{P}_0R(\phi_{>-10}\cdot\LOR\dd_t\phi_{>-10})]\|_{1,2}
\end{align*}

The first of these lines is straightforwardly bounded by
\begin{align*}
2^{(\half-\f{1}{M})m}\left(\sum_{\ko\in\Z}2^{-\la|m-\ko|}C_0c_{\ko}\right)\|\LOR  L\phi_{>-10}\|_{2+,\infty-}\|\LOR  L\phi_{>-10}\|_{\infty-,2+}
\end{align*}
For the second and third we must again invoke Corollary \ref{SIC} to see, for example,
\begin{align*}
&\|\LOR P_m(\Phip(\dhalf\phi))\text{ }\tilde{P}_0R(\dd_t\phi_{>-10}\cdot\LOR\phi_{>-10})]\|_{1,2}\\
&\lesssim\sum_{r\sim s>-10}(2^{\f{21}{4M}r}\|\dd_t\phi_r\|_{2+,4+}+2^{3r/2M}2^{-r}\|\LOR\dd_t\phi_r\|_{2+,\infty-})\|\LOR\phi_s\|_{2+,\infty}\\
&\hspace{20em}\cdot(\|\LOR P_m(\Phip(\dhalf\phi))\|_{\infty,2}+\ldots)\\
&\lesssim C_0^3 \eps \cdot 2^{-m/2}\left(\sum_{\ko}2^{-\la|m-\ko|}c_{\ko}\right)\sum_{r>-10}(2^{(\f{21}{4M}-3\f{M-1}{4M}+\f{1}{M})r}+2^{-(1-\f{1}{M})r})c_r
\end{align*}
which is acceptable for $M$ sufficiently large. The third term can be treated identically and this completes the proof.
\end{proof}

\bigskip
\subsection{Showing that $P_0(HWM_2(\phi))=error$}\label{HWM2_chapter}
\hfill\\
In this section we will prove that the remaining nonlocal terms in the forcing are acceptable. We will use the notation
$$
X\lesssim_{a,b}Y
$$
to mean that $X\leq C_{a,b}Y$ where $C_{a,b}$ grows at most polynomially in $a,\,b$. This is specific to both this section and the letters $a,\,b$.
\begin{proposition}\label{HWM2_prop}
We have
$$
P_0\bl(HWM_2(\phi))=error
$$
\end{proposition}
\begin{proof}
We decompose
$$
HWM_2(\phi)=\sum_{k\in\Z}\phi\x[\dhalf(\phi_{k}\x\dphi)-(\phi_{ k }\x(-\D)\phi)]
$$
and study the regions $ k <-10$, $ k \in[-10,10]$ and $ k >10$ separately.

\begin{itemize}[leftmargin=*]
\item \underline{$ k <-10$:}
We make the decomposition
\begin{align}
&\|P_0\LOR L[\phi\x[\dhalf(\phi_{ k }\x\dphi)-(\phi_{ k }\x(-\D)\phi)]]\|_{1,2}\nonumber\\
&\leq\|P_0\LOR L[\phi_{< k -10}\x[\dhalf(\phi_{ k }\x\dphi)-(\phi_{ k }\x(-\D)\phi)]]\|_{1,2}\label{HWMI}\\
&\quad+\|P_0\LOR L[\phi_{[ k -10,-15]}\x[\dhalf(\phi_{ k }\x\dphi)-(\phi_{ k }\x(-\D)\phi)]]\|_{1,2}\label{HWMII}\\
&\quad+\|P_0\LOR L[\phi_{\geq-15}\x[\dhalf(\phi_{ k }\x\dphi)-(\phi_{ k }\x(-\D)\phi)]]\|_{1,2}\label{HWMIII}
\end{align}
The last term above is the easiest to handle, splitting
\begin{align*}
\eqref{HWMIII}&\lesssim\|\LOR L[\phi_{\geq-15}\x[\dhalf(\phi_{ k }\x\dphi_{>k+10})-(\phi_{ k }\x(-\D)\phi_{>k+10})]]\|_{1,2}\\
&\quad+\|\LOR L[\phi_{\geq-15}\x[\dhalf(\phi_{ k }\x\dphi_{\leq k+10})-(\phi_{k}\x(-\D)\phi_{\leq k+10})]]\|_{1,2}
\end{align*}
The first term here can be handled like \eqref{term1} upon writing 
\begin{equation}\label{expa}
\dhalf(\phi_{ k }\x\dphi_{>k+10})-(\phi_{ k }\x(-\D)\phi_{>k+10})=\sum_{j>k+10}\Lc_{k+j}(\pk,\pj)
\end{equation}
and using \eqref{L_LL}-\eqref{om_exp}, while the second term can be handled like \eqref{term2}.

Next consider \eqref{HWMII}. This term is of the form $\llowesth$, so we will rely heavily on Lemma \ref{SIL}. Since the outer projection $P_0$ almost passes through the operators $\LOR$ and $L$, the third factor of $\phi$ is restricted to $\pso$. Then using \eqref{FS} to write the commutator expression $\Lc_k(\pk,\dhalf\pso)=\Lc_k(\pk,\pso)$ as a Fourier series, we find
\begin{align*}
    \eqref{HWMII}&\lesssim\sum_{j=k-10}^{-15}\sum_{a,b}|c^{(k)}_{a,b}|\|\LOR L[\pj(x)\x[\pk(x-2^{-k}a)\x\pso(x-b)]]\|_{1,2}
\end{align*}
Then for instance if the derivatives $\LOR L$ both fall on $\pj$ we can apply point 2 of Lemma \ref{SIL} and bound
\begin{align*}
    &\sum_{j=k-10}^{-15}\sum_{a,b}c^{(k)}_{a,b}\|\LOR L\pj(x)\x[\pk(x-2^{-k}a)\x\pso(x-b)]\|_{1,2}\\
    &\lesssim\sum_{j=k-10}^{-15}\sum_{a,b}c^{(k)}_{a,b}2^{-k/M}\|\LOR L\pj\|_{2+,\infty}(2^{3k/2M}\|\LOR(\pk(x-2^{-k}a))\|_{2+,\infty-}+2^{3k\f{M-1}{4M}}\|\pk\|_{2+,4+})\\
    &\hspace{20em}\cdot(\|\pso\|_{\infty,2}+\|L(\pso(x-b))\|_{\infty,2}+\|\dd_t\pso\|_{\infty,2}) 
    \\
    &\lesssim_N\sum_{j=k-10}^{-15}\sum_{a,b}\langle a\rangle^{-N}\japb^{-N}2^{(\half-\f{1}{2M})(k-j)}C_0c_j\cdot \langle a\rangle C_0c_k\cdot \langle b\rangle C_0c_0\\
    &\lesssim C_0^3c_0\sum_{j=k-10}^{-15}2^{(\half-\f{1}{2M})(k-j)}c_jc_k
\end{align*}
which is acceptable when summed over $k<-10$.

To complete the case $k<-10$ it remains to study \eqref{HWMI}. Here there are no derivatives falling on the lowest frequency term, so we must use that $\phi$ lies on the sphere. Observe that the third factor of $\phi$ is restricted to frequency $\sim 2^0$ by the outer projection and write
\begin{align*}
\dhalf(\phi_k\x\dhalf\pso)-\phi_k\x(-\D)\pso&=\Lc_{ k }(\phi_k,\pso)
\end{align*}
to find
\begin{align}
\eqref{HWMI}&\leq\sum_{a,b}|c_{a,b}^{( k )}|\|\LOR  L[\phi_{< k -10}(x)\x[\phi_k(x-2^{-k}a)\x\pso(x-b)]]\|_{1,2}\label{big_sum_2}
\end{align}
We then invoke the vector identity
\begin{equation}\label{vector_id}
a\x(b\x c)=b(a\cdot c)-c(b\cdot a)
\end{equation}
to rewrite
\begin{align}
&\|\LOR  L[\phi_{< k -10}(x)\x[\phi_k(x-2^{- k }a)\x\pso(x-b)]]\|_{1,2}\nonumber\\
&\leq\|\LOR  L[\phi_k(x-2^{- k }a) \text{ }\phi_{< k -10}(x)\cdot\pso(x-b)]\|_{1,2}\nonumber\\
&\quad+\|\LOR  L[\pso(x-b)\text{ }\phi_{< k -10}(x)\cdot\phi_k(x-2^{- k }a)]\|_{1,2}\label{two_dots}
\end{align}
Let's start with the first term. In order to use \eqref{geometric_identity_4} we need the two terms in the dot product to be evaluated at the same point, so write
$$
\phi_{<k-10}(x)=\phi_{< k -10}(x-b)+\int_0^1b\cdot \na\phi_{<k-10}(x-\tht b)d\tht
$$
Putting the integral expression into \eqref{two_dots} we get a term of the form $\hll$ which is easily handled. Indeed, borrowing the factor of $2^{ k }$ from $c^{( k )}_{a,b}$, we may write
\begin{align*}
&2^{ k }\|\LOR L[\phi_k(x-2^{- k }a) \text{ }(\int_0^1b\cdot \na\phi_{< k -10}(x-\tht b)d\tht)\cdot\pso(x-b)]\|_{1,2}\\
&\lesssim_b 2^{ k }\int_0^1\|\LOR L(\phi_k(x-2^{- k }a))\|_{2+,\infty-}\|\LOR L(\na\phi_{<k-10}(x-\tht b))\|_{2+,\infty-}\|\LOR L(\pso(x-b))\|_{\infty-,2+}\\
&\lesssim_{a,b} C_0^3\sum_{j< k -10}\int_0^1c_jc_{ k }c_02^{(\half-\f{1}{M})(j+k)}\langle 2^j\tht b\rangle^2d\tht
\end{align*}
which is acceptable when summed as in \eqref{big_sum_2} and over $k<-10$.

We then come to
\begin{align}
&2^{ k }\|\LOR L[\phi_k(x-2^{-k}a)(\phi_{< k -10}\cdot\pso)(x-b)]\|_{1,2}\nonumber\\
&\lesssim2^{ k }\|\LOR L[\phi_k(x-2^{-k}a)(\phi_{[ \kmt,-10]}\cdot\pso)(x-b)]\|_{1,2}\nonumber\\
&\quad+2^{ k }\|\LOR L[\phi_k(x-2^{-k}a)(\phi_{\leq-10}\cdot\pso)(x-b)]\|_{1,2}\label{jk}
\end{align}
The first of these terms can be handled by a straightforward application of Lemma \ref{SIL}. For the second we use \eqref{geometric_identity_4} to bound
\begin{align*}
&2^{ k }\|\LOR L[\phi_k(x-2^{-k}a)(\phi_{\leq-10}\cdot\pso)(x-b)]\|_{1,2}\\
&\lesssim2^{ k }\|\LOR L[\phi_k(x-2^{-k}a)(\phi_{\leq-10}\cdot\pso-P_{\sim0}(\phi_{\leq-10}\cdot\phi_{\sim0}))(x-b)]\|_{1,2}\\
&\quad+2^{ k }\|\LOR L[\phi_k(x-2^{-k}a)P_{\sim0}(\phi_{>-10}\cdot\phi_{>-10})(x-b)]\|_{1,2}
\end{align*}
The first line is easy to handle using Lemma \ref{moving_loc_lemma} to move a derivative onto the low frequency term, so it remains to consider the second line. From the estimate
\begin{align*}
    &2^{ k }\|\LOR L[\phi_k(x-2^{-k}a)P_{\sim0}(\phi_{>-10}\cdot\phi_{>-10})(x-b)]\|_{1,2}\\
    &\lesssim2^{ k }\|\LOR L(\phi_k(x-2^{-k}a))\|_{2+,\infty-}\|\LOR L (P_{\sim0}(\phi_{>-10}\cdot\phi_{>-10})(x-b))\|_{\f{2M}{M+1},\f{2M}{M-1}}
\end{align*}
it remains to show that
\begin{equation}\label{jojo}
\|\LOR L (P_{\sim0}(\phi_{>-10}\cdot\phi_{>-10})(x-b))\|_{\f{2M}{M+1},\f{2M}{M-1}}\lesssim_b C_0^2\eps^2
\end{equation}
First permuting the vector derivatives and the translation by $b$ we find
\begin{align}
    \|\LOR L (P_{\sim0}(\phi_{>-10}\cdot\phi_{>-10})(x-b))\|_{\f{2M}{M+1},\f{2M}{M-1}}
    &\lesssim_b \|\LOR L P_{\sim0}(\phi_{>-10}\cdot\phi_{>-10})\|_{\f{2M}{M+1},\f{2M}{M-1}}\nonumber\\
    &\quad+\|\LOR P_{\sim0}(\dd_t\phi_{>-10}\cdot\phi_{>-10})\|_{\f{2M}{M+1},\f{2M}{M-1}}\label{aligne}
\end{align}
For the first term we commute $\LOR L$ and $P_0$ to see
\begin{align*}
    \|\LOR L P_{\sim0}(\phi_{>-10}\cdot\phi_{>-10})\|_{\f{2M}{M+1},\f{2M}{M-1}}&\lesssim\|P_{\sim 0}\LOR L(\pgt\cdot\pgt)\|_{\f{2M}{M+1},\f{2M}{M-1}}\\
    &\quad+\|\Pp_{\sim 0}\LOR \dd_t(\pgt\cdot\pgt)\|_{\f{2M}{M+1},\f{2M}{M-1}}\\
    &\lesssim\|\LOR L\pgt\|_{2+,\infty-}\|\LOR L\pgt\|_{\infty-,2+}\\
    &\quad+\|\LOR\dd_t\pgt\|_{\infty-,2+}\|\LOR\pgt\|_{2+,\infty-}
\end{align*}
which is as required. This calculation also covered the second term in \eqref{aligne} so \eqref{jojo} is shown, completing the study of \eqref{jk}.

It remains to study the second term in \eqref{two_dots}. Again we write
\begin{equation}\label{intt}
\phi_{<k-10}(x)=\phi_{<k-10}(x-2^{- k }a)+2^{- k }a\cdot\int_0^1\na\phi_{<k-10}(x-2^{- k }a\tht)d\tht
\end{equation}
For the term involving the integral we have
\begin{align*}
&\|\LOR L[\pso(x-b)(2^{- k }a\cdot\int_0^1\na\phi_{<k-10}(x-2^{- k }a\tht)d\tht)\cdot\phi_k(x-2^{- k }a)]\|_{1,2}\\
&\lesssim_a2^{-k}\int_0^1\|\LOR L[\pso(x-b)\na\phi_{<k-10}(x-2^{- k }a\tht)\cdot\phi_k(x-2^{- k }a)]\|_{1,2}d\tht
\end{align*}
Remembering that we can absorb the $2^{- k }$ into $c^{( k )}_{a,b}$ from \eqref{big_sum_2}, we see that this can be treated directly using Lemma \ref{SIL} after splitting $\LOR L$ over the three factors.

For the remaining part of $\phi_{<k-10}$ we use the geometry to bound
\begin{align}
&\|\LOR L[\pso(x-b)(\phi_{<k-10}\cdot\phi_k)(x-2^{- k }a)]\|_{1,2}\nonumber\\
&\lesssim\|\LOR L[\pso(x-b)(\phi_{<k-10}\cdot\phi_k-P_k(\phi_{<k-10}\cdot\phi_{\geq\kmt}))(x-2^{- k }a)]\|_{1,2}\label{plp22}\\
&\quad+\|\LOR L[\pso(x-b)P_k(\phi_{\geq\kmt}\cdot\phi_{\geq\kmt})(x-2^{- k }a)]\|_{1,2}\label{plp21}
\end{align}
First consider \eqref{plp21}. We consider only the more difficult case where $\LOR L$ falls on the $P_{ k }$. By a series of calculations as in \eqref{aligne}, we reduce to studying terms of the form
\begin{align*}
    \|\pso(x-b)(\Pp_{ k }\LOR L(\phi_{\geq\kmt}\cdot\phi_{\geq\kmt}))(x-2^{- k }a)\|_{1,2}
\end{align*}
and
\begin{align*}
    2^{- k }\|\pso(x-b)(\Pp_{ k }\LOR \dd_t(\phi_{\geq\kmt}\cdot\phi_{\geq\kmt}))(x-2^{- k }a)\|_{1,2}
\end{align*}
We restrict our attention to the more delicate second term, as the first can be treated similarly. Considering for example the case in which the angular and time derivative fall on different factors, we use Corollary \ref{SIC} to bound
\begin{align*}
    &2^{- k }\|\pso(x-b)(\Pp_{ k }(\dd_t\phi_{\geq\kmt}\cdot\LOR\phi_{\geq\kmt}))(x-2^{- k }a)\|_{1,2}\\
    &\lesssim2^{- k }\sum_{r,s\geq\kmt}\|\pso(x-b+2^{- k }a)\Pp_{ k }(\dd_t\phi_r\cdot\LOR\phi_s)(x)]\|_{1,2}\\
    &\lesssim2^{- k }\sum_{r\sim s\geq\kmt}2^{- k /M}(2^{3 k \f{M-1}{4M}}2^{\f{21}{4M}(r- k )}\|\dd_t\phi_r\|_{2+,4+}+2^{ k -r}2^{3r/2M}\|\LOR\dd_t\phi_r\|_{2+,\infty-})\\ 
    &\hspace{8em}\cdot\|\LOR\phi_s\|_{2+,\infty}(\|\pso\|_{\infty,2}+2^{ k }\|L(\pso(x-b+2^{- k }a))\|_{\infty,2}+\|\dd_t\pso\|_{\infty,2})\\
    &\lesssim_{a,b} 2^{- k }C_0^3\eps c_0\sum_{r\geq k -10}(2^{(3\f{M-1}{4M}-\f{21}{4M}-\f{1}{M})( k -r)}+2^{(1-\f{1}{M})( k -r)})c_r
\end{align*}
which is acceptable when multiplied by $c^{( k )}_{a,b}$ and summed over $a,b$ and $ k <-10$. Note that here the gain of $2^{ k }$ before the factor $L(\pso(x-b+2^{- k }a))$ was necessary in order to cancel out the loss from the translation by $2^{ k }a$.

For \eqref{plp22} we expand
$$
\phi_{<k-10}\cdot\phi_k-P_k(\phi_{<k-10}\cdot\phi_{\geq k-10})=2^{- k }\int_0^1\int_y(2^{ k }y\check{\chi}_{ k }(y))\na\phi_{<k-10}(x-\tht y)\phi_{\sim k }(x-y)dyd\tht
$$
It therefore remains to bound
\begin{align*}
    \sum_{a,b}2^{- k }|c^{( k )}_{a,b}|\int_{y,\tht}|2^{ k }y\check{\chi}_{ k }(y)|\|\LOR L[\pso(x-b) \na\phi_{<k-10}(x-2^{- k }a-\tht y)\phi_{\sim k }(x-2^{- k }a-y)]\|_{1,2}dyd\tht
\end{align*}
which can be handled using Lemma \ref{SIL}.

\bigskip

\item\underline{$ k >10$:}
We study
$$
\sum_{j\in\Z}\|P_0\LOR L[\phi\x[\dhalf(\phi_{ k }\x\dphi_j)-(\phi_{ k }\x(-\D)\phi_j)]]\|_{1,2}
$$
If the third factor of $\phi$ is restricted to $j< k -10$, then the first $\phi$ is restricted to $\phi_{\sim k }$ by the outer projection $P_0$. The term is therefore of type $\hhl$ and may be treated as \eqref{term1} upon writing
$$
\dhalf(\phi_{ k }\x\dphi_j)-(\phi_{ k }\x(-\D)\phi_j)=\Lc_{ k }(\phi_k,\dhalf\pj)
$$
and appealing to the identities \eqref{L_LL}-\eqref{om_exp}.

When $j> k +10$, again the first factor is restricted to $\phi_{\sim j}$ and this may be treated similarly.

It thus remains to study $j\sim k $. Here there is nothing to be gained by cancellation so we split the term up into its two parts:
\begin{align}
&\|P_0\LOR L[\phi\x[\dhalf(\phi_{ k }\x\dphi_{\sim k })-(\phi_{ k }\x(-\D)\phi_{\sim k })]]\|_{1,2}\nonumber\\
&\leq\underbrace{\|P_0\LOR L[\phi\x\dhalf(\phi_{ k }\x\dphi_{\sim k })]\|_{1,2}}_{(A)}+\underbrace{\|P_0\LOR L[\phi\x(\phi_{ k }\x(-\D)\phi_{\sim k })]\|_{1,2}}_{(B)}\label{2_parts}
\end{align}
We first study (A). This term presents some more complications due to its nonlocal expression, however it also has the advantage that when the remaining $\phi$ is at low frequency, the outer derivative $\dhalf$ is acting at frequency $\sim1$. First split $\phi$ into low and high frequencies:
\begin{align*}
(A)&\leq\underbrace{\|P_0\LOR L[\phi_{\leq-10}\x\dhalf(\phi_{ k }\x\dphi_{\sim k })]\|_{1,2}}_{(A)_{\leq-10}}\\
&\quad+\underbrace{\|P_0\LOR L[\phi_{>-10}\x\dhalf(\phi_{ k }\x\dphi_{\sim k })]\|_{1,2}}_{(A)_{>-10}}
\end{align*}
Here $(A)_{>-10}$ is of type $\hhh$ so can be handled like \eqref{term3} using the radially admissible Strichartz spaces.

$(A)_{\leq-10}$ is of the form $\lhh$ so must be handled using the geometry. In this case it is especially important to keep track of how the vector derivatives are falling as the commutator terms can rapidly cause a build up of derivatives if treated too crudely.

To clarify the calculations we then fix a particular $\Om_{ij}$ and $L_n$ (the inhomogeneous parts of $\LOR$ and $L$ are easier to handle) and make a very precise decomposition. Note that we are free to switch the order of $\omij$ and $L_n$ up to a term of the same form. Writing $\tilde{P}_0\dhalf=\mathcal{P}_0$, a radial operator, and using the Leibniz rule, we have
\begin{equation*}
\begin{aligned}
P_0L_n\Om_{ij}[\plt\x\dhalf(\phi_k\x\dhalf\phi_{\sim k })]&=P_0[L_n\Om_{ij}\plt\x\mathcal{P}_0(\phi_k\x\dhalf\phi_{\sim k })]\\
    &\quad+P_0[L_n\plt\x\mathcal{P}_0(\Om_{ij}\phi_k\x\dhalf\phi_{\sim k })]\\
    &\quad+P_0[L_n\plt\x\mathcal{P}_0(\phi_k\x\dhalf\Om_{ij}\phi_{\sim k })]\\
    &\quad+P_0[\Om_{ij}\plt\x\mathcal{P}_0(L_n\phi_k\x\dhalf\phi_{\sim k })]\\
    &\quad+P_0[\plt\x\mathcal{P}_0(L_n\Om_{ij}\phi_k\x\dhalf\phi_{\sim k })]\\
    &\quad+P_0[\plt\x\mathcal{P}_0(L_n\phi_k\x\dhalf\Om_{ij}\phi_{\sim k })]\\
    &\quad+P_0[\Om_{ij}\plt\x\mathcal{P}_0(\phi_k\x L_n\dhalf \phi_{\sim k })]\\
    &\quad+P_0[\plt\x\mathcal{P}_0(\Om_{ij}\phi_k\x L_n\dhalf \phi_{\sim k })]\\
    &\quad+P_0[\plt\x\mathcal{P}_0(\phi_k\x L_n\Om_{ij}\dhalf \phi_{\sim k })]\\
    &\quad+P_0[\Om_{ij}\plt\x\mathcal{P}'_0\dd_t(\phi_k\x\dhalf\phi_{\sim k })]\\
    &\quad+P_0[\plt\x\mathcal{P}'_0\dd_t(\Om_{ij}\phi_k\x\dhalf\phi_{\sim k })]\\
    &\quad+P_0[\plt\x\mathcal{P}'_0\dd_t(\phi_k\x\dhalf\Om_{ij}\phi_{\sim k })]\\
\end{aligned}
\begin{aligned}
&\left.\vphantom{\begin{aligned}
P_0L_n\Om_{ij}[\plt\x\dhalf(\phi_k\x\dhalf\phi_{\sim k })]&=P_0[L_n\Om_{ij}\plt\x\mathcal{P}_0(\phi_k\x\dhalf\phi_{\sim k })]\\
    &\quad+P_0[L_n\plt\x\mathcal{P}_0(\Om_{ij}\phi_k\x\dhalf\phi_{\sim k })]\\
    &\quad+P_0[L_n\plt\x\mathcal{P}_0(\phi_k\x\dhalf\Om_{ij}\phi_{\sim k })]\\
  \end{aligned}}\right\rbrace\quad\text{(A1)}\\
  &\left.\vphantom{\begin{aligned}
    &\quad+P_0[\Om_{ij}\plt\x\mathcal{P}_0(L_n\phi_k\x\dhalf\phi_{\sim k })]\\
    &\quad+P_0[\plt\x\mathcal{P}_0(L_n\Om_{ij}\phi_k\x\dhalf\phi_{\sim k })]\\
    &\quad+P_0[\plt\x\mathcal{P}_0(L_n\phi_k\x\dhalf\Om_{ij}\phi_{\sim k })]\\
  \end{aligned}}\right\rbrace\quad\text{(A2)}\\
  &\left.\vphantom{\begin{aligned}
    &\quad+P_0[\Om_{ij}\plt\x\mathcal{P}_0(\phi_k\x L_n\dhalf \phi_{\sim k })]\\
    &\quad+P_0[\plt\x\mathcal{P}_0(\Om_{ij}\phi_k\x L_n\dhalf \phi_{\sim k })]\\
    &\quad+P_0[\plt\x\mathcal{P}_0(\phi_k\x L_n\Om_{ij}\dhalf \phi_{\sim k })]\\
  \end{aligned}}\right\rbrace\quad\text{(A3)}\\
  &\left.\vphantom{\begin{aligned}
    &\quad+P_0[\Om_{ij}\plt\x\mathcal{P}'_0\dd_t(\phi_k\x\dhalf\phi_{\sim k })]\\
    &\quad+P_0[\plt\x\mathcal{P}'_0\dd_t(\Om_{ij}\phi_k\x\dhalf\phi_{\sim k })]\\
    &\quad+P_0[\plt\x\mathcal{P}'_0\dd_t(\phi_k\x\dhalf\Om_{ij}\phi_{\sim k })]\\
  \end{aligned}}\right\rbrace\quad\text{(A4)}\\
\end{aligned}
\end{equation*}
In the above $\mathcal{P}_0'$ is another operator of the type described in Section \ref{ang_com_chap}, which may not be radial.

We start by considering the group (A4) which is the most difficult since there is an additional derivative falling on the high frequency terms. Consider first the case where $\dd_t$ falls on $\phi_k$. Writing the operator $\mathcal{P}'_0$ as an explicit convolution by some $\mathcal{K}_0$, we have for the first line
\begin{align}
&\Om_{ij} \phi_{\leq-10}\x \mathcal{P}'_0( \dd_t\phi_{ k }\x\dphi_{\sim k })\nonumber\\
&=\Om_{ij}\plt(x)\x\int_y\mathcal{K}_0(y)(\dd_t\phi_k(x-y)\x\dphi_{\sim k }(x-y))dy\nonumber\\
&=\int_y\mathcal{K}_0(y)\dd_t\phi_k(x-y)\text{ }\Om_{ij}\plt(x)\cdot\dphi_{\sim k }(x-y)dy\label{fs1}\\
&\quad-\int_y\mathcal{K}_0(y)\dphi_{\sim k }(x-y)\text{ }\Om_{ij}\plt(x)\cdot\dd_t\phi_k(x-y)dy\label{fs2}
\end{align}
We must then split
$$
\Om_{ij}\plt(x)=(\Om_{ij}\plt)(x-y)+y\cdot\int_{\tht=0}^1\na(\Om_{ij}\plt)(x-\tht y)d\tht
$$
Putting the integral term into \eqref{fs1} we find
\begin{align*}
&\left\|\int_y\mathcal{K}_0(y)\dd_t\phi_k(x-y)\text{ }(y\cdot\int_{\tht=0}^1\na(\Om_{ij}\plt)(x-\tht y)d\tht)\cdot\dphi_{\sim k }(x-y)dy\right\|_{1,2}\\
&\lesssim\int_y\int_0^1|\mathcal{K}_0(y)||y|\|\dd_t\phi_k\|_{9,\f{10}{3}}\|\na\Om_{ij}\plt\|_{\f{18}{7},\infty}\|\dphi_{\sim k }\|_{2,5}d\tht dy
\end{align*}
which is acceptable when summed over $ k >10$ since the factor of $|y|$ is absorbed by the kernel $\mathcal{K}_0$. 
The same argument works for \eqref{fs2}, and for the corresponding terms in the second and third lines of (A4).

For \eqref{fs1} it therefore remains to consider
\begin{align*}
&\int_y\mathcal{K}_0(y)[\dd_t\phi_k\text{ }\Om_{ij}\plt\cdot\dphi_{\sim k }](x-y)dy
\end{align*}
The corresponding terms from the second and third lines of (A4) are
\begin{align}
&\int_y\mathcal{K}_0(y)[\Om_{ij}\dd_t\phi_k\text{ }\plt\cdot\dphi_{\sim k }](x-y)dy\label{a22}
\end{align}
and
\begin{align}
&\int_y\mathcal{K}_0(y)[\dd_t\phi_k\text{ }\plt\cdot\dhalf\Om_{ij}\phi_{\sim k }](x-y)dy\label{a23}
\end{align}
and we have to bound the sum of these in $L^1_tL^2_x$.

In order to use \eqref{geometric_identity_4} we rewrite $\plt=\phi_{\leq \kmt}-\phi_{[-10, \kmt]}$. Then for the high frequency part we can again use a bound as for \eqref{term3} to see e.g.
\begin{align*}
&\|\dd_t\phi_k\text{ }\Om_{ij}\phi_{[-10, \kmt]}\cdot\dphi_{\sim k }\|_{1,2}\lesssim\|\dd_t\phi_k\|_{9,\f{10}{3}}\|\Om_{ij}\phi_{[-10, \kmt]}\|_{\f{18}{7},\infty} \|\dphi_{\sim k }\|_{2,5}
\end{align*}
which is acceptable. The same argument works for \eqref{a22} and \eqref{a23}.

For the low frequency part $\plkmt$ we want to use \eqref{geometric_identity_4}. We have
\begin{align}
&\dd_t\phi_k\text{ }\Om_{ij}\phi_{\leq \kmt}\cdot\dphi_{\sim k }\nonumber\\
&=\dd_t\phi_k\text{ }(\Om_{ij}\phi_{\leq \kmt}\cdot\dphi_{\sim k }-\dhalf(\Om_{ij}\phi_{\leq \kmt}\cdot\phi_{\sim k }))\label{pp1}\\
&\quad+\dd_t\phi_k\text{ }\dhalf(\Om_{ij}\phi_{\leq \kmt}\cdot\phi_{\sim k }-P_{\sim k }(\Om_{ij}\phi_{\leq \kmt}\cdot\phi_{> \kmt}))\label{pp2}\\
&\quad+\dd_t\phi_k\text{ }\dhalf P_{\sim k }(\Om_{ij}\phi_{\leq \kmt}\cdot\phi_{> \kmt})\label{pp3}
\end{align}
For the first line we use Lemma \ref{useful_lemma} to bound
\begin{align*}
\|\eqref{pp1}\|_{1,2}&\lesssim\sum_{j\leq \kmt}\|\dd_t\phi_k\text{ }\Lc_j(\Om_{ij}\pj,\phi_{\sim k })\|_{1,2}\\
&\lesssim\sum_{j\leq \kmt}2^j\|\dd_t\phi_k\|_{2+,\infty-} \|\Om_{ij}\pj\|_{2+,\infty-} \|\phi_{\sim k }\|_{\infty-,2+}
\end{align*}
and for the second line Lemma \ref{moving_loc_lemma} to find
\begin{align*}
\|\eqref{pp2}\|_{1,2}&\lesssim\|\dd_t\phi_k\|_{2+,\infty-} 2^{ k }2^{- k }\|\na\Om_{ij}\plkmt\|_{2+,\infty-} \|\phi_{> \kmt}\|_{\infty-,2+}
\end{align*}
Both of these bounds are acceptable, and we can treat the corresponding parts of \eqref{a22} and \eqref{a23} in the same way.

We at last come to the interesting part, \eqref{pp3}. We want to use \eqref{geometric_identity_4}, but are obstructed by the presence of the $\Om_{ij}$. The solution is to combine this term with the corresponding part of \eqref{a23}. We have
\begin{align*}
&\dd_t\phi_k\text{ }\dhalf P_{\sim k }(\Om_{ij}\phi_{\leq \kmt}\cdot\phi_{> \kmt})+\dd_t\phi_k\text{ }\dhalf P_{\sim k }(\phi_{\leq \kmt}\cdot\Om_{ij}\phi_{> \kmt})\\
&=\dd_t\phi_k\text{ }\dhalf \Om_{ij}P_{\sim k }(\phi_{\leq \kmt}\cdot\phi_{> \kmt})\\
&=-\half\dd_t\phi_k\text{ }\dhalf \Om_{ij}P_{\sim k }(\phi_{> \kmt}\cdot\phi_{> \kmt})
\end{align*}
We then bound
\begin{align*}
\|\dd_t\phi_k\text{ }\dhalf \Om_{ij}P_{\sim k }(\phi_{> \kmt}\cdot\phi_{> \kmt})\|_{1,2}&\lesssim2^{ k }\|\dd_t\phi_k\|_{2+,\infty-} \|\Om_{ij}\phi_{> \kmt}\|_{2+,\infty-} \|\phi_{> \kmt}\|_{\infty-,2+}
\end{align*}
which is acceptable. The corresponding term in \eqref{a22} can be handled similarly on its own.

The term \eqref{fs2} can be handled in the same way, using the Leibniz rule on the time-derivative in place of Lemma \ref{useful_lemma}.

To complete the study of (A4) we have to consider the case where the time derivative falls on $\dhalf\phi_{\sim k }$ instead of $\pk$. In this case the argument carries through identically until it comes to handling the term analogous to \eqref{pp3},
\begin{align*}
\phi_k\text{ }\dhalf P_{\sim k }(\Om_{ij}\phi_{\leq \kmt}\cdot\dd_t\phi_{> \kmt})
\end{align*}
with similar contributions
\begin{align}
\Om_{ij}\phi_k\text{ }\dhalf P_{\sim k }(\phi_{\leq \kmt}\cdot\dd_t\phi_{> \kmt})\label{a222}
\end{align}
and
\begin{align}
\phi_k\text{ }\dhalf P_{\sim k }(\phi_{\leq \kmt}\cdot\dd_t\Om_{ij}\phi_{> \kmt})\label{a232}
\end{align}
from the second and third lines of (A4). First note that if the derivative were instead on $\phi_{\lkmt}$ we would be fine in all three cases, for instance
\begin{align*}
\|\phi_k\text{ }\dhalf P_{\sim k }(\Om_{ij}\dd_t\phi_{\leq \kmt}\cdot\phi_{> \kmt})\|_{1,2}\lesssim\|\phi_k\|_{2+,\infty-} 2^{ k }\|\Om_{ij}\dd_t\phi_{\leq \kmt}\|_{2+,\infty-} \|\phi_{> \kmt}\|_{\infty-,2+}
\end{align*}
It therefore remains to study
\begin{gather*}
\phi_k\text{ }\dhalf P_{\sim k }\dd_t(\Om_{ij}\phi_{\leq \kmt}\cdot\phi_{> \kmt})\\
\Om_{ij}\phi_k\text{ }\dhalf P_{\sim k }\dd_t(\phi_{\leq \kmt}\cdot\phi_{> \kmt})\\
\phi_k\text{ }\dhalf P_{\sim k }\dd_t(\phi_{\leq \kmt}\cdot\Om_{ij}\phi_{> \kmt})
\end{gather*}
Combining the first and last terms and using \eqref{geometric_identity_4} we bound
\begin{align*}
&\|\phi_k\text{ }\dhalf P_{\sim k }\dd_t(\Om_{ij}\phi_{\leq \kmt}\cdot\phi_{> \kmt})+\phi_k\text{ }\dhalf P_{\sim k }\dd_t(\phi_{\leq \kmt}\cdot\Om_{ij}\phi_{> \kmt})\|_{1,2}\\
&=\|\phi_k\text{ }\dhalf \dd_t\Om_{ij}P_{\sim k }(\phi_{> \kmt}\cdot\phi_{> \kmt})\|_{1,2}\\
&\lesssim2^{ k }\|\phi_k\|_{2+,\infty-} \|\dd_t\LOR\phi_{> \kmt}\|_{\infty-,2+}\|\LOR\phi_{>\kmt}\|_{2+,\infty-}
\end{align*}
which is acceptable when summed over $ k >10$. The middle term can be dealt with in the same way on its own. This completes the analysis for (A4).

The groups (A1), (A2) and (A3) must be treated simultaneously in order to preserve the structure for \eqref{geometric_identity_4}. In all cases, we can work as for (A4) up to the decomposition \eqref{pp1}-\eqref{pp3}. At this point for (A1) we will be studying
\begin{gather*}
\phi_k\text{ }L_n\Om_{ij}\phi_{\leq \kmt}\cdot\dphi_{\sim k }\\
\Om_{ij}\phi_k\text{ }L_n\phi_{\leq \kmt}\cdot\dphi_{\sim k }\\
\phi_k\text{ }L_n\phi_{\leq \kmt}\cdot \dhalf\Om_{ij}\phi_{\sim k }
\end{gather*}
for the first, second and third lines respectively. For (A2) we will have
\begin{gather*}
L_n\phi_k\text{ }\Om_{ij}\phi_{\leq \kmt}\cdot\dphi_{\sim k }\\
L_n\Om_{ij}\phi_k\text{ }\phi_{\leq \kmt}\cdot\dphi_{\sim k }\\
L_n\phi_k\text{ }\phi_{\leq \kmt}\cdot\dhalf \Om_{ij}\phi_{\sim k }
\end{gather*}
and for (A3)
\begin{gather*}
\phi_k\text{ }\Om_{ij}\phi_{\leq \kmt}\cdot L_n\dphi_{\sim k }\\
\Om_{ij}\phi_k\text{ }\phi_{\leq \kmt}\cdot L_n\dphi_{\sim k }\\
\phi_k\text{ }\phi_{\leq \kmt}\cdot L_n\Om_{ij}\dphi_{\sim k }
\end{gather*}
(as well as a second set of easier terms from the expansion of the cross product).
Adding these nine terms together and reversing the Leibniz rule on $\Om_{ij}$ and $L_n$ this comes to
\begin{align*}
L_n\Om_{ij}[\phi_k\text{ }\phi_{\leq \kmt}\cdot\dphi_{\sim k }]
\end{align*}
which we can split up as
\begin{align}
&L_n\Om_{ij}[\phi_k\text{ }(\phi_{\leq \kmt}\cdot\dphi_{\sim k }-\dhalf(\phi_{\leq \kmt}\cdot\phi_{\sim k }))]\nonumber\\
&+L_n\Om_{ij}[\phi_k\text{ 
}\dhalf(\phi_{\leq \kmt}\cdot\phi_{\sim k }-P_{\sim k }(\phi_{\leq \kmt}\cdot\phi_{> \kmt}))]\nonumber\\
&+L_n\Om_{ij}[\phi_k\text{ }\dhalf P_{\sim k }(\phi_{\leq \kmt}\cdot\phi_{> \kmt})]\label{kpk}
\end{align}
The first term is of the form
\begin{align*}
\sum_{j\leq \kmt}\Om_{ij}L_n[\phi_k \Lc_j(\pj,\phi_{\sim k })]
\end{align*}
which can be treated using \eqref{L_LL}-\eqref{om_exp} and placing $\phi_k$ and $\pj$ into $L^{2+}_tL^{\infty-}_x$, and $\phi_{\sim k }$ into $L^{\infty-}_tL^{2+}_x$.

For the second term, we use Lemma \ref{moving_loc_lemma} to write, for example when $L_n\Om_{ij}$ falls on the difference term
\begin{align*}
&\|\phi_k\text{ }L_n\Om_{ij}\dhalf(\phi_{\leq \kmt}\cdot\phi_{\sim k }-P_{\sim k }(\phi_{\leq \kmt}\cdot\phi_{> \kmt}))\|_{1,2}\\
&\lesssim\|\phi_k\|_{\f{2M}{M-1},2M}\left\|2^{-k}L_n\Om_{ij}\dhalf \int_{y,\tht}(2^{ k }y\check{\chi}_{ k }(y))\na\phi_{\leq \kmt}(x-\tht y)\phi_{\sim k }(x-y) dyd\tht\right\|_{\f{2M}{M+1},\f{2M}{M-1}}\\
&\lesssim 2^{- k }2^{-(\half+\f{1}{M}) k }C_0c_{ k }\int_{y,\tht}|2^{ k }y\check{\chi}_{ k }(y)|\left\|L_n\Om_{ij}\dhalf[\na\phi_{\leq \kmt}(x-\tht y)\phi_{\sim k }(x-y) ]\right\|_{\f{2M}{M+1},\f{2M}{M-1}}dyd\tht
\end{align*}
where
\begin{align*}
&\left\|L_n\Om_{ij}\dhalf[\na\phi_{\leq \kmt}(x-\tht y)\phi_{\sim k }(x-y) ]\right\|_{\f{2M}{M+1},\f{2M}{M-1}}\\
&\lesssim 2^k\|L\LOR(\na\phi_{\leq\kmt}(x-\tht y))\|_{2+,\infty-}\|L\LOR(\phi_{\sim k}(x-y))\|_{\infty-,2+}\\
&\quad+\|\LOR(\dd_t\na\phi_{\leq k-10}(x-\tht y))\|_{2+,\infty-}\|\LOR(\phi_{\sim k}(x-y))\|_{\infty-,2+}\\
&\quad+\|\LOR(\na\phi_{\leq k-10}(x-\tht y))\|_{2+,\infty-}\|\LOR(\dd_t\phi_{\sim k}(x-y))\|_{\infty-,2+}
\end{align*}
all of which are acceptable using \eqref{translation_bound}, because $|y|$ behaves like $2^{-k}$ in the integral.

The third term of \eqref{kpk} can be treated using \eqref{geometric_identity_4} and the commutation relation between $L_n$ and $\dhalf P_{\sim k }$. We place $\phi_k$ into $L^{2+}_tL^{\infty-}_x$, one of the high frequency factors into $L^{2+}_tL^{\infty-}_x$ and the other into $L^{\infty-}_tL^{2+}_x$, in particular the one accompanied by a $\dd_t$ when this arises from $[L_n,\dhalf P_{\sim k }]$.

To conclude the case $ k >10$, it remains to consider (B). When $\phi$ appears at high frequency this term is again easily handled like \eqref{term3}. In the low frequency case, $\plt$, the term can be treated analogously to the group (A4) which also contains two high frequency derivatives, but with significant simplifications.

\bigskip

\item\underline{$ k \in[-10,10]$:} This time we consider
$$
\|P_0\LOR L[\phi\x(\dhalf(\pso\x\dphi)-(\pso\x(-\D)\phi))]\|_{1,2}
$$
This term is easiest to handle when the outer factor of $\phi$ is at high frequency. Indeed we have
\begin{align*}
   &\|P_0\LOR L[\phi_{>-10}\x(\dhalf(\pso\x\dphi)-(\pso\x(-\D)\phi))]\|_{1,2} \\
   &\lesssim\|\LOR L[\phi_{>-10}\x(\dhalf(\pso\x\dphi_{\leq20})-(\pso\x(-\D)\phi_{\leq20}))]\|_{1,2} \\
   &\quad+\|\LOR L[\phi_{>-10}\x(\dhalf(\pso\x\dphi_{>20})-(\pso\x(-\D)\phi_{>20}))]\|_{1,2}
\end{align*}
Upon carefully commuting $\LOR L$ through the operators $\dhalf$, we can bound the first line above by placing $\pgt$ into $L^{2+}_tL^{\infty-}_x$, $\pso$ into $L^{\infty-}_tL^{2+}_x$ and $\phi_{<20}$ also into $L^{2+}_tL^{\infty-}_x$, without needing to use the cancellation structure. For the second line we do need the cancellation, since we cannot handle two derivatives falling on a high frequency factor, so bound this by
\begin{align*}
    \sum_{j>20}\|\LOR L[\pgt\x\Lc_0(\pso,\dhalf\pj)]\|_{1,2}
\end{align*}
which can be dealt with by placing $\pgt$ and $\pso$ into $L^{2+}_tL^{\infty-}_x$ and $\pj$ into $L^{\infty-}_tL^{2+}_x$.

The case $\plt$ is more delicate. Note that in this case the final factor of $\phi$ is also restricted to frequency $\lesssim1$. First suppose it is at frequency $\sim1$. Write
$$
\tilde{P}_0(\dhalf(\pso\x\dphi_{\sim0})-(\pso\x(-\D)\phi_{\sim0}))=\Lc_0(\pso,\pso)
$$
Then using \eqref{FS2} we have
\begin{align*}
    &\|P_0\LOR L[\plt\x(\dhalf(\pso\x\dphi_{\sim0})-(\pso\x(-\D)\phi_{\sim0}))]\|_{1,2}\\
    &\lesssim\sum_{a,b}|c_{a,b}^{(0)}|\|\LOR L[\plt(x)\x(\pso(x-a)\x\pso(x-b))]\|_{1,2}\\
    &\lesssim\sum_{a,b}|c_{a,b}^{(0)}|\|\LOR L[\pso(x-a)\text{ }\plt(x)\cdot\pso(x-b)]\|_{1,2}+\text{similar term}
\end{align*}
We can then replace $\plt(x)$ with $\plt(x-b)$ up to an integral term of the form $\hhl$. 
Using Lemma \ref{moving_loc_lemma} we can then exchange $\plt\cdot\pso$ for $\tilde{P}_0(\plt\cdot\pgt)\simeq\tilde{P}_0(\pgt\cdot\pgt)$, and bound
\begin{align*}
    &\|\LOR L[\pso(x-a)\ \tilde{P}_0(\pgt\cdot\pgt)(x-b)]\|_{1,2}\\
    &\lesssim_{a,b}\|\LOR L\pso\|_{2+,\infty-}(\|\LOR L(\pgt\cdot\pgt)\|_{\f{2M}{M+1},\f{2M}{M-1}}+\|\LOR(\dd_t\pgt\cdot\pgt)\|_{\f{2M}{M+1},\f{2M}{M-1}})
\end{align*}
which can be handled by placing one of the high frequency factors (the differentiated one in the second case) into $L^{\infty-}_tL^{2+}_x$ and the other into the other into $L^{2+}_tL^{\infty-}_x$.

We now consider the case where the third factor of $\phi$ is at low frequency, say $\leq 2^{-20}$. We start by writing
$$
\dhalf(\pso\x\dphi_j)-(\pso\x(-\D)\phi_j)=\Lc_j(\pso,\pj)
$$
for $j<-20$. Then we have
\begin{align}
    &\|P_0 \LOR L [\plt\x(\dhalf(\pso\x\dphi_{<-20})-(\pso\x(-\D)\phi_{<-20}))]\|_{1,2}\nonumber\\
    &\lesssim\sum_{j<-20}\sum_{a,b}|c_{a,b}^{(j)}|\|P_0\LOR L[\pso(x-a)\text{ }\plt(x)\cdot\pj(x-2^{-j}b)]\|_{1,2}\label{pi0}\\
    &\quad+\sum_{j<-20}\sum_{a,b}|c_{a,b}^{(j)}|\|P_0\LOR L[\pj(x-2^{-j}b)\text{ }\plt(x)\cdot\pso(x-a)]\|_{1,2}\label{pi1}
\end{align}
We will study the first line above, the second being similar (in fact significantly easier). In order to use \eqref{geometric_identity_4}, we split $\plt$ into $\phi_{[j-10,-10]}+\phi_{<j-10}$. The first component here is handled by a straightforward application of Lemma \ref{SIL}, so we are left to study
\begin{align*}
    &\sum_{j<-20}\sum_{a,b}|c_{a,b}^{(j)}|\|P_0\LOR L[\pso(x-a)\text{ }\phi_{<j-10}(x)\cdot\pj(x-2^{-j}b)]\|_{1,2}
\end{align*}
We first replace $\phi_{<j-10}(x)$ with $\phi_{<j-10}(x-2^{-j}b)$ up to an acceptable integrable term of the form $\llowesth$. 
 We are then left with
\begin{align*}
    &\sum_{j<-20}\sum_{a,b}|c_{a,b}^{(j)}|\|P_0\LOR L[\pso(x-a)\text{ }(\phi_{<j-10}\cdot\pj)(x-2^{-j}b)]\|_{1,2}
\end{align*}
Similarly to before, we can replace $\phi_{<j-10}\cdot\pj$ with $P_j(\phi_{<j-10}\cdot\phi_{\geq j-10})$ up to the term 
\begin{align*}
    2^{-j}|c_{a,b}^{(j)}|\int_{y,\tht}|(2^jy)^T\check{\chi}_j(y)|\|\LOR L[\pso(x-a)\text{ }\na\phi_{<j-10}(x-2^{-j}b-\tht y)\cdot\phi_{\sim j}(x-2^{-j}b-y)]\|_{1,2}dyd\tht
\end{align*}
which is also of type $\llowesth$ and can be handled using Lemma \ref{SIL}. We can then finally invoke \eqref{geometric_identity_4} to bound
\begin{align}
    &\sum_{j<-20}\sum_{a,b}|c_{a,b}^{(j)}|\|P_0\LOR L[\pso(x-a)\text{ }P_j(\phi_{<j-10}\cdot\phi_{\geq j-10})(x-2^{-j}b)]\|_{1,2}\nonumber\\
    &\lesssim\sum_{r\sim s\geq j-10}\sum_{j<-20}\sum_{a,b}|c_{a,b}^{(j)}|\|P_0\LOR L[\pso(x-a)\text{ }P_j(\phi_r\cdot\phi_s)(x-2^{-j}b)]\|_{1,2}\label{pkpk}
\end{align}
This term can be handled using Corollary \ref{SIC}. For example, when $\LOR$ and $L$ both fall on $P_j$, we have (up to some terms which are symmetric in $r$ and $s$)
\begin{align*}
    &2^j\|\pso(x-a)\text{ }\LOR L(P_j(\phi_r\cdot\phi_s)(x-2^{-j}b))\|_{1,2}\\
    &\lesssim_b \|\pso(x-a+2^{-j}b)\Pp_j(\LOR \dd_t\phi_r\cdot\phi_s)(x)\|_{1,2}\\
    &\quad+\|\pso(x-a+2^{-j}b)\Pp_j( \dd_t\phi_r\cdot\LOR\phi_s)(x)\|_{1,2}\\
    &\quad+2^j\|\pso(x-a+2^{-j}b)\Pp_j(\LOR L\phi_r\cdot\phi_s)(x)\|_{1,2}\\
    &\quad+2^j\|\pso(x-a+2^{-j}b)\Pp_j(L\phi_r\cdot\LOR \phi_s)(x)\|_{1,2}\\
    &\lesssim2^{-j/M}2^{3j\f{M-1}{4M}}2^{\f{21}{4M}(s-j)}\|\phi_s\|_{2+,4+}\|\LOR\dd_t\phi_r\|_{2+,\infty}C_0c_0\\
    &\quad+2^{-j/M}2^{j-s}2^{3s/2M}\|\LOR\phi_s\|_{2+,\infty-}\|\LOR\dd_t\phi_r\|_{2+,\infty}C_0c_0\\
    &\quad+2^{-j/M}2^{3j\f{M-1}{4M}}2^{\f{21}{4M}(r-j)}\|\dd_t\phi_r\|_{2+,4+}\|\LOR\phi_s\|_{2+,\infty}C_0c_0\\
    &\quad+2^{-j/M}2^{j-r}2^{3r/2M}\|\LOR\dd_t\phi_r\|_{2+,\infty-}\|\LOR\phi_s\|_{2+,\infty}C_0c_0\\
    &\quad+2^j(\text{same terms with }L \text{ instead of }\dd_t)\\
    &\lesssim (2^{[-\f{1}{M}+3\f{M-1}{4M}-\f{21}{4M}](j-r)}+2^{(1-\f{1}{M})(j-r)})C_0^3c_0c_rc_s
\end{align*}
which is acceptable when summed as in \eqref{pkpk}. This concludes the study of \eqref{pi0}.

\end{itemize}
\end{proof}
\section{Normal Forms}\label{normal_forms_chapter}
The goal of this section is to perform a series of normal transformations to reduce the second and third terms on the right hand side of equation (\ref{new_wm_eqn}) to $error$.

\subsection{Low-high-high term}
To handle the third term, we make the transformation
\begin{align*}
\psi^L\mapsto\tilde{\psi}^L:=\psi^L+\half(\D_1) \quad\text{ for }\quad (\D_1)=
\begin{pmatrix}
\D_1^0\\
\D_1^{1,1}+\D_1^{1,2}+\D_1^{1,3}\\
\vdots\\
\D_1^{3,1}+\D_1^{3,2}+\D_1^{3,3}
\end{pmatrix}
\end{align*}
with
\begin{gather*}
\D_1^0:=P_0(\phi_{\leq-10}\phi_{>-10}^T\phi_{>-10})\\
\D_1^{n,1}:=P_0((L_n\phi)_{\leq-10}\phi_{>-10}^T\phi_{>-10})\\
\D_1^{n,2}:=P_0(\phi_{\leq-10}(L_n\phi)_{>-10}^T\phi_{>-10})\\
\D_1^{n,3}:=P_0(\phi_{\leq-10}\phi_{>-10}^T(L_n\phi)_{>-10})
\end{gather*}

We start by showing that this transformation is bounded in the following sense:
\begin{proposition}\label{D1_bounded_prop}
For $(\D_1)$ as above, it holds
\begin{equation*}
\|(\D_1)\|_{S_0}\lesssim C_0^2\eps c_0
\end{equation*}
and
\begin{equation*}
\|\LOR (\D_1)[0]\|_{\dot{H}^{3/2}\x \dot{H}^{1/2}}\lesssim c_0
\end{equation*}
\end{proposition}
\begin{proof}
By Bernstein's inequality we have
\begin{align*}
\|\D_1^0\|_{S_0}\simeq&\max_{\mathcal{Q}}\|\LOR^{1-\dl(p,q)} \na_{t,x}P_0(\phi_{\leq-10}\phi_{>-10}^T\phi_{>-10})\|_{p,q}
\lesssim\max_{\mathcal{Q}}\|\LOR  \na_{t,x}(\phi_{\leq-10}\phi_{>-10}^T\phi_{>-10})\|_{p,2}
\end{align*}
When the derivative falls on a high frequency term we have, noting that $p\neq2$ for $(p,q)\in\mathcal{Q}$,
\begin{align*}
\|\LOR (\plt \na_{t,x}\pgt^T\pgt)\|_{p,2}
\lesssim\|\LOR \plt\|_{\infty,\infty}\|\LOR \na_{t,x}\pgt\|_{\infty,2}\|\LOR \pgt\|_{p,\infty}
\lesssim C_0^2\eps c_0
\end{align*}
and in the same way
\begin{align*}
\|\LOR (\na_{t,x}\plt \pgt^T\pgt)\|_{p,2}
\lesssim\|\LOR \na_{t,x}\plt\|_{\infty,\infty}\|\LOR \pgt\|_{\infty,2}\|\LOR \pgt\|_{p,\infty}
\lesssim C_0^2\eps c_0
\end{align*}
The argument for the remaining $\D_1^{n,i}$ ($n,i=1,2,3$) is identical.

We now show the bound on the initial data. Recall the smallness assumption \eqref{small1}:
\begin{equation}
\|\LOR P_k\phi[0]\|_{\dot{H}^{3/2}\x\dot{H}^{1/2}}+\|\LOR (x\cdot\na)P_k\phi[0]\|_{\dot{H}^{3/2}\x\dot{H}^{1/2}}\leq c_k
\end{equation}
It immediately follows that
\begin{align*}
\|\LOR (\D^0_1)(0)\|_{\dot{H}^{3/2}}\lesssim\|\LOR \plt(0)\|_\infty\|\LOR \pgt(0)\|_\infty\|\LOR \pgt(0)\|_2\ll c_0
\end{align*}
with a similar argument for the initial velocity.

The terms involving $L$ are slightly more complicated. The initial bound in $\dot{H}^{3/2}$ presents no particular difficulties, however to study the initial velocity we have to iterate the equation. We consider only $\D^{n,1}_1$ as an example, in which case we have
\begin{align}
\|\LOR \dd_t(\D_1^{n,1})(0)\|_{\dot{H}^{1/2}}
&\lesssim \|\LOR \dd_t(L_n\phi)_{\leq-10}(0)\|_\infty \|\LOR \pgt(0)\|_\infty\|\LOR \pgt(0)\|_2\nonumber\\
&\quad+\|\LOR (L_n\phi)_{\leq-10}(0)\|_\infty \|\LOR \dd_t\pgt(0)\|_2\|\LOR \pgt(0)\|_\infty\nonumber\\
&\lesssim \eps^2\|\LOR \dd_t(L_n\phi)_{\leq-10}(0)\|_\infty+\eps^2c_0\label{jiji}
\end{align}
where we used
$$
\|\LOR (L_n\phi)_{\leq-10}(0)\|_\infty\lesssim\sum_{k\leq-10}2^{3k/2}\|\LOR P_k(x_n\dd_t\phi(0))\|_2\lesssim\sum_{k\leq-10}\|\LOR \langle x\cdot\na\rangle P_k\dd_t\phi(0)\|_{\dot{H}^{1/2}}\lesssim\eps
$$
To bound the term involving $\dd_t(L_n\phi)$ we need to refer back to the equation. Indeed, the necessary bound will follow and the proof will be complete given the following claim.

\begin{claim}\label{claim1}
Let $k\in\Z$. It holds
\begin{equation}\label{x_nbox}
\|\LOR  P_k(x_n \cdot \Box \phi)(0)\|_{L^2_x}\lesssim 2^{-k/2} \eps c_k
\end{equation}
for all $n=1,2,3$. It follows that
\begin{align*}
\|\LOR \dd_tP_k(L_n\phi)(0)\|_2
\lesssim\|\LOR P_k(x_n\dd_t^2+\dd_{x_n})\phi(0)\|_2\lesssim 2^{-k/2} c_k
\end{align*}
\end{claim}
\begin{proof}[Proof of claim \ref{claim1}]
By scaling it suffices to consider $k=0$.

We start with the wave maps source terms, placing high frequency factors are placed into $L^2$ and others into $L^\infty$. If a high frequency derivative is forced into $L^\infty$, we let it absorb the multiplier $x_n$ which scales like an inverse derivative. Explicitly, we have
\begin{align*}
\|\LOR P_0(x_n(\phi \dau\phi^T\dad\phi)(0))\|_{L^2}&\lesssim \|\LOR  (x_n(\phi\dau\phi_{\leq-10}^T\dad\phi_{>-10})(0)\|_{L^2}\\
&\quad+\|\LOR  (x_n(\phi\dau\phi_{>-10}^T\dad\phi_{>-10})(0)\|_{L^2}\\
&\quad+\|\LOR  (x_n(\phi\dau\phi_{\leq-10}^T\dad\phi_{\leq-10})(0)\|_{L^2}\\
&\lesssim \|\LOR P_0 \phi(0)\|_{\infty}\|\LOR \dau\phi_{\leq-10}(0)\|_{\infty}\|\LOR (x_n\dad\phi_{>-10}(0))\|_{2}\\
&\quad+\|\LOR P_0 \phi(0)\|_{\infty}\|\LOR \dau\phi_{>-10}(0)\|_{2}\|\LOR (x_n\dad\phi_{>-10}(0))\|_{\infty}\\
&\quad+\|\LOR P_0 (x_n\phi_{\sim0}(0))\|_{2}\|\LOR \dau\phi_{\leq-10}(0)\|_{\infty}\|\LOR \dad\phi_{\leq-10}(0)\|_{\infty}\\
&\lesssim\eps c_0
\end{align*}
The first half-wave maps terms, $HWM_1(\phi)$, can be treated similarly.
To study $HWM_2$ we decompose
\begin{align*}
\|\LOR  P_0(x_n \cdot HWM_2(\phi))(0)\|_{L^2_x}&\lesssim
\sum_{j,k\in\Z}\|\LOR  P_0(x_n \cdot [\phi\x(\dhalf(\phi_{ k }\x\dhalf\phi_j)-\phi_{ k }\x\D\phi_j)])(0)\|_{L^2_x}
\end{align*}
Note that if $j\gg k$ or $k\gg j$ we can write
$$
\dhalf(\phi_{ k }\x\dhalf\phi_j)-\phi_{ k }\x\D\phi_j=\Lc_{k+j}(\pk,\pj)
$$
In this symmetric form we see that it suffices to consider $j\gg k$. Starting with $k\geq-10$ we use the Fourier expansion \eqref{FS2} and find
\begin{align*}
&\sum_{\substack{ k \geq-10\\j\gg k }}\|\LOR  P_0(x_n \cdot (\phi\x\Lc_{j+k}(\pk,\pj)))(0)\|_{L^2_x}\\
&\lesssim\sum_{\substack{ k \geq-10\\j\gg k }}\sum_{a,b\in\Z^3}|c^{(k+j)}_{a,b}|\|\LOR  (x_n\cdot \phi\x(\pk (x-2^{- k }a)\x\pj(x-2^{-j}b)))\|_{L^2_x}\\
&\lesssim\sum_{\substack{ k \geq-10\\j\gg k }}\sum_{a,b\in\Z^3}|c^{(k+j)}_{a,b}|\|\LOR \phi(0)\|_\infty\|\LOR  (\pk (x-2^{- k }a))(0)\|_\infty\|\LOR (x_n\cdot\pj(x-2^{-j}b))(0)\|_2\\
&\lesssim\sum_{\substack{k\geq-10\\j\gg k }}2^{j+k}\cdot c_k\cdot2^{-5j/2} c_j\lesssim\eps c_0
\end{align*}
The case $k<-10$ can be handled similarly upon further localising the outer factor of $\phi$ to low and high frequencies.

It remains to study the case $j\simeq  k $. When this frequency is low we have
\begin{align*}
&\sum_{ k <-10}\|\LOR  P_0(x_n \cdot [\phi\x(\dhalf(\phi_{ k }\x\dhalf\phi_{\sim k })-\phi_{ k }\x\D\phi_{\sim k })])(0)\|_{L^2_x}\\
&\simeq\sum_{ k <-10}\|\LOR  P_0(x_n \cdot [\phi_{\sim0}\x(\dhalf(\phi_{ k }\x\dhalf\phi_{\sim k })-\phi_{ k }\x\D\phi_{\sim k })])(0)\|_{L^2_x}\\
&\lesssim\sum_{k <-10}\|\LOR (x_n\cdot\phi_{\sim0})(0)\|_2\cdot 2^{2 k }\|\LOR  \pk (0)\|_\infty\|\LOR \phi_{\sim k }(0)\|_\infty\lesssim \eps^2 c_0
\end{align*}
and when it is high we have
\begin{align*}
&\sum_{ k \geq-10}\|\LOR  P_0(x_n \cdot [\phi\x(\dhalf(\phi_{ k }\x\dhalf\phi_{\sim k })-\phi_{ k }\x\D\phi_{\sim k })])(0)\|_{L^2_x}\\
&\lesssim\sum_{ k \geq-10}\|\LOR  P_0(\phi\x  x_n\dhalf(\phi_{ k }\x\dhalf\phi_{\sim k }))(0)\|_{L^2_x}+\|\LOR  P_0(\phi \x ((x_n\phi_{ k })\x\D\phi_{\sim k })(0)\|_{L^2_x}
\end{align*}
Interchanging the $x_n$ and $\dhalf$ up to a term involving a Riesz transform this is seen to be acceptable upon placing $\phi$ and $\psk$ into $L^\infty$ and the remaining factor into $L^2$.
This completes the proof of the claim.
\end{proof}
\end{proof}

We now show that this transformation reduces the equations to the form
\begin{equation}\label{po_eqn}
\Box\tilde{\psi}_0=-2\plt\dad\plt^T\dau\psi_0-2[P_0(\plt\dad\plt^T\dau\phi_{>-10})-\plt\dad\plt^T\dau\psi_0]+error
\end{equation}
and
\begin{align}
\Box\tilde{\psi}_n&=-2(L_n\phi)_{\leq-10}\dd_\alpha\phi_{\leq-10}^T\dd^\al\psi_0\nonumber\\
&\quad\quad-2\phi_{\leq-10}\dd_\alpha(L_n\phi)_{\leq-10}^T\dd^\al\psi_0\nonumber\\
&\quad\quad-2\phi_{\leq-10}\dd_\alpha\phi_{\leq-10}^T\dd^\al\psi_n\nonumber\\
&\quad-2[P_0((L_n\phi)_{\leq-10}\dd_\al\phi_{\leq-10}^T\dd^\al\phi_{>-10})-(L_n\phi)_{\leq-10}\dd_\alpha\phi_{\leq-10}^T\dd^\al\psi_0]\nonumber\\
&\quad\quad-2[P_0(\phi_{\leq-10}\dd_\al(L_n\phi)_{\leq-10}^T\dd^\al\phi_{>-10})-\phi_{\leq-10}\dd_\alpha(L_n\phi)_{\leq-10}^T\dd^\al\psi_0]\nonumber\\
&\quad\quad-2[P_0(\phi_{\leq-10}\dd_\al\phi_{\leq-10}^T\dd^\al(L_n\phi)_{>-10})-\phi_{\leq-10}\dd_\alpha\phi_{\leq-10}^T\dd^\al\psi_n]\nonumber\\
&\quad+error\label{po_eqn_n}
\end{align}
Indeed, clearly
$$
\Box\po_0=\Box\psi_0+\half P_0\Box(\plt\pgt^T\pgt)
$$
where
\begin{align}
P_0\Box(\plt\pgt^T\pgt)=&P_0(\Box\plt \text{ }\pgt^T\pgt)\label{111}\\
&+4P_0(\dau\plt \text{ } \dad\pgt^T\pgt)\label{222}\\
&+2P_0(\plt  \text{ }(\Box\pgt)^T\pgt)\label{333}\\
&+2P_0(\plt  \text{ }\dau\pgt^T\dad\pgt))\label{444}
\end{align}
We then need to show that $\eqref{111}=\eqref{222}=\eqref{333}=error$, since \eqref{444} cancels the (LHH) term in the equation for $\psi_0$, \eqref{new_wm_eqn}. Arguing similarly for $\tilde{\psi}_n$ we see that \eqref{po_eqn} and \eqref{po_eqn_n} follow from the following proposition.
\begin{proposition}\label{big_error_prop}
Denote
\begin{gather*}
T_1(\vp^{(1)},\vp^{(2)},\vp^{(3)}):=P_0(\Box\vp^{(1)}_{\leq-10} \text{ }(\vp^{(2)}_{>-10})^T\vp^{(3)}_{>-10})\\
T_2(\vp^{(1)},\vp^{(2)},\vp^{(3)}):=P_0(\dau\vp^{(1)}_{\leq-10} \text{ } \dad(\vp^{(2)}_{>-10})^T\vp^{(3)}_{>-10})\\
T_3(\vp^{(1)},\vp^{(2)},\vp^{(3)}):=P_0(\vp^{(1)}_{\leq-10}  \text{ }(\Box\vp^{(2)}_{>-10})^T\vp^{(3)}_{>-10})
\end{gather*}
Then it holds
$$
T_1(\vp^{(1)},\vp^{(2)},\vp^{(3)})=T_2(\vp^{(1)},\vp^{(2)},\vp^{(3)})=T_3(\vp^{(1)},\vp^{(2)},\vp^{(3)})=error
$$
for any of 
$$
(\vp^{(1)},\vp^{(2)},\vp^{(3)})\in\{(\phi,\phi,\phi),(L_n\phi,\phi,\phi),(\phi,L_n\phi,\phi),(\phi,\phi,L_n\phi) \}
$$
\end{proposition}

\begin{proof}
For simplicity, we will only prove the statement for $(\phi,\phi,\phi)$ and the other cases follow in the same way.\footnote{The only difference when including the factors of $L_n$ comes in estimating the half-wave maps terms upon iterating the equation. Here one must simply pay a little attention when exchanging $L$ with the operators $\dhalf$ and $P_j$, however this causes no problems thanks to the commutation relations of Section \ref{ang_com_chap}.} Let's start with $T_1$. We have
\begin{align*}
&\|\LOR P_0(\Box\phi_{\leq-10}\phi_{>-10}^T\pgt)\|_{1,2}\lesssim\|\LOR \Box\phi_{\leq-10}\|_{2+,\infty-} \|\LOR \pgt\|_{\infty-,2+} \|\LOR \pgt\|_{2+,\infty-}
\end{align*}
so it suffices to show the following:
\begin{claim}
It holds
\begin{equation}\label{box_estimate}
\|\LOR  P_k\Box\phi\|_{\f{2M}{M-1},2M}\lesssim 2^{(\f{3}{2}-\f{1}{M})k}C_0c_k
\end{equation}
\end{claim}
\begin{proof}[Proof of claim]
It again suffices to consider $k=0$. By the usual frequency decomposition and Bernstein's inequality we have
\begin{align}
\|\LOR P_0(\phi\dau\phi^T\dad\phi)\|_{\f{2M}{M-1},2M}&\lesssim\|\LOR P_0(\phi\dau\phi_{\leq-10}^T\dad\phi_{>-10})\|_{\f{2M}{M-1},2}\label{T1}\\
&\quad+\|\LOR P_0(\phi\dau\phi_{>-10}^T\dad\phi_{>-10})\|_{\f{2M}{M-1},\f{2M}{M+1}}\label{T2}\\
&\quad+\|\LOR P_0(\phi_{\sim0}\dau\phi_{\leq-10}^T\dad\phi_{\leq-10})\|_{\f{2M}{M-1},2M}\label{T5}
\end{align}
Always placing the lone $\phi$ into $L^\infty_{t,x}$ we have
\begin{align*}
\eqref{T1}\lesssim\|\LOR \dau\plt\|_{\f{2M}{M-1},\infty} \|\LOR \dad\phi_{>-10}\|_{\infty,2}
\lesssim C_0^2 c_0^2
\end{align*}
Likewise
\begin{align*}
\eqref{T2}\lesssim\|\LOR \dau\pgt\|_{\f{4M}{M-1},\f{4M}{M+1}} \|\LOR \dad\pgt\|_{\f{4M}{M-1},\f{4M}{M+1}}
\lesssim C_0^2 c_0^2
\end{align*}
and lastly
\begin{align*}
\eqref{T5}\lesssim\|\LOR \phi_{\sim0}\|_{\f{2M}{M-1},2M} \|\LOR \dau\plt\|_{\infty,\infty} \|\LOR \dad\plt\|_{\infty,\infty}\lesssim C_0^3c_0^3
\end{align*}

The first half-wave maps terms can be treated in the same way and the remaining such terms can be handled analogously upon incorporating the modifications as in the proof of Claim \ref{claim1}.

\end{proof}

For $T_2$ we have
\begin{align*}
&\|\LOR P_0(\dau\plt\dad\pgt^T\pgt)\|_{1,2}\\
&\lesssim\|\LOR \dau\plt\|_{2+,\infty-}\|\LOR \dad\pgt\|_{\infty-,2+}  \|\LOR \pgt\|_{2+,\infty-}\lesssim C_0^3c_0^3
\end{align*}

Lastly, for $T_3$ we must expand the wave operator within the overall expression. Starting with the wave maps source terms we have
\begin{align}
\|\LOR P_0(&\plt P_{>-10}(\phi\dad\phi^T\dau\phi)^T\pgt)\|_{1,2}\nonumber\\
\lesssim\sum_{k>-10}&\|\LOR P_0(\plt P_k(\phi\dad\phi_{\leq k-10}^T\dau\phi_{> k-10})^T\phi_{\sim k})\|_{1,2}\label{tt1}\\
\quad\quad+&\|\LOR P_0(\plt P_k(\phi\dad\phi_{> k-10}^T\dau\phi_{> k-10})^T\phi_{\sim k})\|_{1,2}\label{tt2}\\
\quad\quad+&\|\LOR P_0(\plt P_k(\phi_{\sim k}\dad\phi_{\leq k-10}^T\dau\phi_{\leq k-10})^T\phi_{\sim k})\|_{1,2}\label{tt5}
\end{align}
which can be treated like \eqref{term1}, \eqref{term3} and \eqref{term2} respectively. The $HWM_1(\phi)$ terms are analogous.
For $HWM_2$ we have
\begin{align*}
\sum_{k>-10}\|\LOR P_0(&\plt P_k(HWM_2(\phi))^T\phi_{\sim k})\|_{1,2}\nonumber\\
\lesssim\sum_{k>-10}\sum_{l,j}\|\LOR &( P_k(\phi\x(\dhalf(\phil\x\dhalf\pj)-\phil\x\D\pj))^T\cdot\psk)\|_{1,2}
\end{align*}
First consider $j\gg l$, $l\geq k-10$, comparable to \eqref{tt2}. Using the expansion \eqref{om_exp} we have
\begin{align*}
&\sum_{k>-10}\sum_{\substack{l\geq k-10\\j\gg l}}\|\LOR ( P_k(\phi\x(\dhalf(\phil\x\dhalf\pj)-\phil\x\D\pj))^T\cdot\psk)\|_{1,2}\\
&\simeq\sum_{k>-10}\sum_{\substack{l\geq k-10\\j\gg l}}\|\LOR P_k(\phi\x\Lc_{l+j}(\phil,\pj))^T\cdot\LOR\psk)\|_{1,2}\\
&\lesssim\sum_{k>-10}\sum_{\substack{l\geq k-10\\j\gg l}}2^{l+j}\|\LOR \phil\|_{9,10/3}\|\pj\|_{2,5}\|\LOR \psk\|_{18/7,\infty}+2^{l+j}\|\phil\|_{2,5}\|\LOR \pj\|_{9,10/3}\|\LOR \psk\|_{18/7,\infty}\\
&\lesssim C_0^3\eps^2c_0
\end{align*}
as required. When $l<k-10$ the term behaves like \eqref{tt1} and we have
\begin{align*}
&\sum_{k>-10}\sum_{\substack{l< k-10\\j\gg l}}\|\LOR ( P_k(\phi\x(\dhalf(\phil\x\dhalf\pj)-\phil\x\D\pj))^T\cdot\psk)\|_{1,2}\\
&\lesssim\sum_{k>-10}\sum_{\substack{l< k-10\\j\gg l}}2^{l+j}\|\LOR \phil\|_{2+,\infty-} \|\LOR \pj\|_{\infty-,2+} \|\LOR \psk\|_{2+,\infty-}\lesssim C_0^3\eps^2 c_0
\end{align*}

The case $l\gg j$ can be treated in the same way. When $j\simeq l< k-10$ we are in the regimen of \eqref{tt5} and have
\begin{align*}
&\sum_{k>-10}\sum_{l< k-10}\|\LOR ( P_k(\phi\x(\dhalf(\phil\x\dhalf\phi_{\sim l})-\phil\x\D\phi_{\sim l}))^T\cdot\psk)\|_{1,2}\\
&\lesssim\sum_{k>-10}\sum_{l< k-10}2^{2l}\|\LOR  \phil\|_{2+,\infty-} \|\LOR \phi_{\sim l}\|_{2+,\infty-} \|\LOR \psk\|_{\infty-,2+}\lesssim C_0^3\eps^2c_0
\end{align*}
and finally if $j\simeq l\geq k-10$ we refer to \eqref{tt2} and find
\begin{align*}
&\sum_{k>-10}\sum_{l\geq k-10}\|\LOR ( P_k(\phi\x(\dhalf(\phil\x\dhalf\phi_{\sim l})-\phil\x\D\phi_{\sim l}))^T\cdot\psk)\|_{1,2}\\
&\lesssim\sum_{k>-10}\sum_{l\geq k-10}2^{2l}\|\LOR  \phi_{\sim l}\|_{9,10/3} \|\phi_{\sim l}\|_{2,5} \|\LOR \psk\|_{18/7,\infty}\lesssim C_0^3\eps^2c_0
\end{align*}
which completes the proof.
\end{proof}

\subsection{Low-low-high error term}
Write
\begin{multline}\label{diff_integral}
P_0(\phi_{\leq-10}\dd_\al\phi_{\leq-10}^T\dd^\al\phi_{>-10})-\phi_{\leq-10}\dd_\alpha\phi_{\leq-10}^T\dd^\al\psi_0\\
=-\tilde{P}_0\int_0^1\int_y\check{\chi}_0(y)y^T \na_x(\plt(x- \tht y)\dad\plt^T(x- \tht y))\dau\pso(x-y)d \tht dy
\end{multline}
This splits into two terms by Leibniz's rule: one where the derivative $\na_x$ falls on the non-differentiated term, and one where it falls on the $\dad\plt$. The first such term is unproblematic:
\begin{multline*}
\left\|\tilde{P}_0\int_0^1\int_y\check{\chi}_0(y)y^T \na_x\plt(x- \tht y\text{ })\dad\plt^T(x- \tht y)\dau\pso(x-y)d \tht dy\right\|_{1,2}\\
\lesssim\int_0^1\int_y|y\check{\chi}_0(y)|\, \|\na_x\plt\|_{2+,\infty-} \|\dad\plt\|_{2+,\infty-} \|\dau\pso\|_{\infty-,2+} dyd \tht\lesssim C_0^3\eps^2c_0
\end{multline*}
so that
\begin{multline}\label{int_p1_eq}
\Box\tilde{\psi}_0=-2\plt\dad\plt^T\dau\psi_0\\
+2\tilde{P}_0\int_0^1\int_y\check{\chi}_0(y)y^T \plt(x- \tht y)\text{ }\na_x\dau\plt^T(x- \tht y)\dad\pso(x-y)d \tht dy+error
\end{multline}
and similar for the $\tilde{\psi}_n$. This motivates our second transformation 
\begin{align*}
\tilde{\psi}^L\mapsto\Phi^L:=\tilde{\psi}^L-(\D_2)
\quad\text{ for }\quad
\D_2=
\begin{pmatrix}
\D_2^0\\
\D_2^{1,1}+\D_2^{1,2}+\D_2^{1,3}\\
\vdots\\
\D_2^{3,1}+\D_2^{3,2}+\D_2^{3,3}
\end{pmatrix}
\end{align*}
with
\begin{gather*}
\D_2^0:=\tilde{P}_0\int_0^1\int_y\check{\chi}_0(y)y^T\plt(x- \tht y)\na_x\plt^T(x- \tht y)\pso(x-y)d \tht dy\\
\D_2^{n,1}:=\tilde{P}_0\int_0^1\int_y\check{\chi}_0(y)y^T(L_n\phi)_{\leq-10}(x- \tht y)\na_x\plt^T(x- \tht y)\pso(x-y)d \tht dy\\
\D_2^{n,2}:=\tilde{P}_0\int_0^1\int_y\check{\chi}_0(y)y^T\plt(x- \tht y)\na_x(L_n\phi)_{\leq-10}^T(x- \tht y)\pso(x-y)d \tht dy\\
\D_2^{n,3}:=\tilde{P}_0\int_0^1\int_y\check{\chi}_0(y)y^T\plt(x- \tht y)\na_x\plt^T(x- \tht y)(L_n\phi)_{\sim0}(x-y)d \tht dy
\end{gather*}
Henceforth we drop the $\tilde{P}_0$ since it does not affect the calculations.

As usual write $\Phi^L=(\Phi_0,\Phi_1,\Phi_2,\Phi_3)$ and start by noting the boundedness of the transformation, the proof of which is very similar to Proposition \ref{D1_bounded_prop} and so omitted.
\begin{proposition}\label{LLH_NT}
It holds
\begin{equation*}
\|(\D_2)\|_{S_0}\lesssim C_0^2\eps c_0
\end{equation*}
and moreover
$$
\|\LOR (\D_2)[0]\|_{\dot{H}^{3/2}\x \dot{H}^{1/2}}\lesssim c_0
$$
\end{proposition}

\bigskip

This time we show that the transformation reduces the equations to the form
\begin{equation}\label{pt_eqn}
\Box\Phi_0=-2\plt\dad\plt^T\dau\psi_0+error
\end{equation}
and
\begin{align}
\Box\Phi_n&=-2(L_n\phi)_{\leq-10}\dd_\alpha\phi_{\leq-10}^T\dd^\al\psi_0\nonumber\\
&\quad\quad-2\phi_{\leq-10}\dd_\alpha(L_n\phi)_{\leq-10}^T\dd^\al\psi_0\nonumber\\
&\quad\quad-2\phi_{\leq-10}\dd_\alpha\phi_{\leq-10}^T\dd^\al\psi_n+error\label{pt_eqn_n}
\end{align}

Observe that
\begin{align}
\Box(\D_2^0)&=\Box \left(\int_0^1\int_y\check{\chi}_0(y)y^T\plt(x- \tht y)\na_x\plt^T(x- \tht y)\pso(x-y)d \tht dy\right)\nonumber\\
&=\int_0^1\int_y\check{\chi}_0(y)y^T\Box\plt(x- \tht y)\na_x\plt^T(x- \tht y)\pso(x-y)d \tht dy\label{1111}\\
&\quad+\int_0^1\int_y\check{\chi}_0(y)y^T\plt(x- \tht y)\Box\na_x\plt^T(x- \tht y)\pso(x-y)d \tht dy\label{2222}\\
&\quad+\int_0^1\int_y\check{\chi}_0(y)y^T\plt(x- \tht y)\na_x\plt^T(x- \tht y)\Box\pso(x-y)d \tht dy\label{3333}\\
&\quad+2\int_0^1\int_y\check{\chi}_0(y)y^T\dau\plt(x- \tht y)\dad\na_x\plt^T(x- \tht y)\pso(x-y)d \tht dy\label{4444}\\
&\quad+2\int_0^1\int_y\check{\chi}_0(y)y^T\dau\plt(x- \tht y)\na_x\plt^T(x- \tht y)\dad\pso(x-y)d \tht dy\label{5555}\\
&\quad+2\int_0^1\int_y\check{\chi}_0(y)y^T\plt(x- \tht y)\dau\na_x\plt^T(x- \tht y)\dad\pso(x-y)d \tht dy\label{6666}
\end{align}
and similar expressions for the $\D_2^{n,i}$.
The final term \eqref{6666} cancels with the integral expression in equation \eqref{int_p1_eq}, so we must show the following:
\begin{proposition}
We have
$$\eqref{1111}=\ldots=\eqref{5555}=error$$
and the same holds when any one factor of $\phi$ in the expressions \eqref{1111},$\ldots$,\eqref{5555} is replaced by $L_n\phi$ ($n=1,2,3$).
\end{proposition}
\begin{proof}
We will neglect the angular and vector derivatives in what follows in order to reduce notation. We also drop the transpose symbols.

Let's start with \eqref{1111}. Using the estimate \eqref{box_estimate} on $\Box\phi$ and the $L^1$ boundedness of the kernel $\check{\chi}_0(y)y $ we have
\begin{align*}
\|\eqref{1111}\|_{1,2}&\lesssim\|\Box\plt\|_{2+,\infty-}\|\na_x\plt\|_{2+,\infty-}\|\pso\|_{\infty-,2+}\lesssim C_0^3\eps c_0
\end{align*}
which is acceptable. For \eqref{2222} we again have to iterate the equation. For the wave maps source terms we have
\begin{align}
&\left\|\int_0^1\int_y\check{\chi}_0(y)y\, (\plt\cdot\na_xP_{\leq-10}(\phi\dad\phi\dau\phi) )(x- \tht y)\pso(x-y)d \tht dy\right\|_{1,2}\nonumber\\
&\lesssim\sum_{k\leq-10}\left\|\int_0^1\int_y\check{\chi}_0(y)y\, (\plt\cdot\na_xP_k(\phi\dad\phi_{\leq k-10}\dau\phi_{>k-10}) )(x- \tht y)\pso(x-y)d \tht dy\right\|_{1,2}\label{p1}\\
&\quad+\left\|\int_0^1\int_y\check{\chi}_0(y)y\, (\plt\cdot\na_xP_k(\phi\dad\phi_{>k-10}\dau\phi_{>k-10}) )(x- \tht y)\pso(x-y)d \tht dy\right\|_{1,2}\label{p2}\\
&\quad+\left\|\int_0^1\int_y\check{\chi}_0(y)y\, (\plt\cdot\na_xP_k(\psk\dad\phi_{\leq k-10}\dau\phi_{\leq k-10}) )(x- \tht y)\pso(x-y)d \tht dy\right\|_{1,2}\label{p3}
\end{align}
Placing all the low frequency factors of $\phi$ into $L^\infty_{t,x}$ we bound the above by
\begin{align*}
&\sum_{k\leq-10}2^k\|\dad\plkmt\|_{2+,\infty-}\|\dau\phi_{>k-10}\|_{\infty-,2+}\|\pso\|_{2+,\infty-}\\
&\quad\quad\quad+2^k\|\dad\pgkmt\|_{2,5} \|\dau\pgkmt\|_{9,10/3} \|\pso\|_{18/7,\infty}\\
&\quad\quad\quad+2^k\|\psk\|_{\infty,\infty}\|\dad\plkmt\|_{2+,\infty-}\|\dau\plkmt\|_{2+,\infty-}\|\pso\|_{\infty-,2+}\\
&\lesssim C_0^3\eps c_0
\end{align*}
The half-wave maps terms are treated analogously using methods as in the proof of Proposition \ref{big_error_prop}.

We now turn to \eqref{3333}. Considering the wave maps portion of $\Box\phi$ we have
\begin{align*}
&\left\|\int_0^1\int_y\check{\chi}_0(y)y\, (\plt\na_x\plt )(x- \tht y)P_{\sim0}(\phi\dad\pgt\dau\pgt)(x-y)d \tht dy\right\|_{1,2}\\
&\quad+\left\|\int_0^1\int_y\check{\chi}_0(y)y\, (\plt\na_x\plt )(x- \tht y)P_{\sim0}(\phi\dad\plt\dau\pgt)(x-y)d \tht dy\right\|_{1,2}\\
&\quad+\left\|\int_0^1\int_y\check{\chi}_0(y)y\, (\plt\na_x\plt )(x- \tht y)P_{\sim0}(\phi\dad\plt\dau\plt)(x-y)d \tht dy\right\|_{1,2}\\
&\lesssim\|\na_x\plt\|_{\f{18}{7},\infty}\|\dad\pgt\|_{2,5} \|\dau\pgt\|_{9,\f{10}{3}}\\
&\quad+\|\na_x\plt\|_{2+,\infty-} \|\dad\plt\|_{2+,\infty-} \|\dau\pgt\|_{\infty-,2+}\\
&\quad+\|\na_x\plt\|_{\infty,\infty} \|\pso\|_{\infty-,2+} \|\dad\plt\|_{2+,\infty-} \|\dau\plt\|_{2+,\infty-}\\
&\lesssim C_0^3\eps c_0
\end{align*}
Again half-wave maps terms are analogous.

Lastly,
\begin{align*}
\|\eqref{4444}\|_{1,2}\lesssim \|\dau\plt\|_{2+,\infty-} \|\dad\na_x\plt\|_{2+,\infty-}\|\pso\|_{\infty-,2+}\lesssim C_0^3\eps c_0
\end{align*}
while
\begin{align*}
\|\eqref{5555}\|_{1,2}\lesssim \|\dau\plt\|_{2+,\infty-} \|\na_x\plt\|_{2+,\infty-}\|\dad\pso\|_{\infty-,2+}\lesssim C_0^3\eps c_0
\end{align*}
\end{proof}

We can now recast the equations in an even more distilled form. Denote
\begin{align*}
A_\al:=-\plt\dad\plt^T; && A_\al^n:=-(( L_n\phi )_{\leq-10}\dad\plt^T+\plt\dad ( L_n\phi )_{\leq-10}^T)
\end{align*}
Then writing $A_\al^n$ for the block vector $(A_\al^1,A_\al^2,A_\al^3)$ and $\mathbf{I}_3$ the block $3\x 3$ identity matrix (so a $9\x9$ matrix), we arrive at
\begin{align}\label{matrix_eqn}
\Box\Phi^L=2
\begin{pmatrix}
A_\al & 0\\
A_\al^n & A_\al\mathbf{I}_3
\end{pmatrix}
\dau\psi^L+error
\end{align}
with
\begin{align}\label{Phi_bound}
\|\Phi^L\|_{S_0}\leq\|\psi^L\|_{S_0}+C_0^3\eps c_0\lesssim C_0c_0
\end{align}

Henceforth we will use the notation
\begin{align*}
\mathbf{A}^L_\al:=\begin{pmatrix}
A_\al & 0\\
A_\al^n & A_\al\mathbf{I}_3
\end{pmatrix}
\end{align*}

\section{The gauge transformation}\label{approx_PT_chapter}
In this chapter, we perform a nonlinear transformation to cancel out the remaining nontrivial term in the equation above. Our construction is a simplification of that in \cite{Tao_1} (possible due to our working in Besov rather than Sobolev spaces).

Fix a large integer $N$ \textit{depending on T}. Define the matrix field $\Uu$ by
$$
\Uu:=\mathbf{I}_4+\sum_{-N<k\leq-10}\Uu_k
$$
Here $\mathbf{I}_4$ is the block $4\x4$ identity matrix and the $\Uu_k$ are defined inductively by
\begin{equation*}
\Uu_k:=
\begin{pmatrix}
-\phi_{<k}\phi_k^T && 0\\
-(\phi_{<k}( L_n\phi )_k^T+( L_n\phi )_{<k}\phi_k^T)  && -\phi_{<k}\phi_k^T \cdot \mathbf{I}_3
\end{pmatrix}
\Uu_{<k}
\end{equation*}
and
$$
\Uu_{<k}:=\mathbf{I}_4+\sum_{-N<k'<k}\Uu_{k'}
$$
This $\Uu$ is constructed so as to almost satisfy $\dd_\alpha \Uu=\Aa^L_\al \Uu$ and so cancel out the troublesome frequency interactions remaining in our equation. This will be discussed in the next chapter.

One may verify inductively that $\Uu_{<k}$ has frequency support on $\{|\xi|\lesssim 2^k\}$, so $\Uu$ has frequency support on $\{|\xi|\lesssim1\}$. Moreover $\Uu$ is of the form
$$
\Uu=
\begin{pmatrix}
U&&0\\
U^n &&U\mathbf{I}_3
\end{pmatrix}
$$
for a $3\x3$ matrix $U$ and a block vector $U^n=(U^1,U^2,U^3)$ with each $U^i$ a $3\x3$ matrix. The same lower triangular structure of couse holds for $\Uu_k$, $\Uu_{<k}$ and $\Uu^{-1}$, the existence of which is shown in the following proposition.
\begin{proposition}\label{big_prop}
For $C_0\gg1$ fixed, $\eps(C_0)$ sufficiently small, the matrix $\Uu$ is a perturbation of the identity,
\begin{equation}\label{orthogonality}
\|\LOR (\Uu-\mathbf{I}_4)\|_{\infty,\infty},\|\LOR \dd_t(\Uu-\mathbf{I}_4)\|_{\infty,\infty}\lesssim C_0\eps
\end{equation}
and is invertible with
\begin{equation}\label{Uinfty}
\|\LOR \Uu\|_{\infty,\infty},\|\LOR (\Uu^{-1})\|_{\infty,\infty}\lesssim1
\end{equation}
Moreover, for any admissible pair $(p,q)\in\mathcal{Q}$ with $1-\f{1}{p}-\f{3}{q}>\sg>0$, we have
\begin{equation}
\|\LOR^{1-\dl(p,q)} \dd_\alpha \Uu\|_{p,q}, \|\LOR^{1-\dl(p,q)} \dd_\alpha (\Uu^{-1})\|_{p,q}\lesssim_{p,q} C_0\eps\label{49}
\end{equation}
for each $\alpha=0,\ldots,3$. More precisely,
\begin{equation}\label{duk_bound}
\|\LOR^{1-\dl(p,q)}  \dad  \Uu_K\|_{p,q}\lesssim_{p,q} 2^{(1-\f{1}{p}-\f{3}{q})K}C_0c_K
\end{equation}
for all $-N<K\leq-10$.
\end{proposition}

\begin{proof}
We first show (\ref{Uinfty}).
We will show by induction that $$\|\LOR \Uu_{<K}\|_{\infty,\infty}\leq2$$
for all $-N\leq K\leq-9$. Since $\Uu=\Uu_{<-9}$ this proves the first part of (\ref{Uinfty}). When $K=-N$ this is clearly true, so suppose it holds for all $k$ below some fixed $K>-N$. Then for any $-N<k<K$ we have, for $\phi^L$ as in \eqref{pl},
\begin{align}
\|\LOR \Uu_k\|_{\infty,\infty}
\lesssim\|\LOR P_{<k}\phi^L\|_{\infty,\infty}\|\LOR P_k\phi^L\|_{\infty,\infty}\|\LOR \Uu_{<k}\|_{\infty,\infty}\lesssim C_0c_k\label{uk_bound}
\end{align}
It follows that
$$
\|\LOR \Uu_{<K}\|_{\infty,\infty}\lesssim1+\sum_{-N<k<K}C_0c_k\lesssim 2
$$
for $\eps(C_0)$ sufficiently small.  Note that the first part of \eqref{orthogonality} is also a consequence of \eqref{uk_bound}, and it follows that $\Uu$ is invertible with
$$
\|\Uu^{-1}\|_{\infty,\infty}=\|\mathbf{I}_4+(\mathbf{I}_4-\Uu)+(\mathbf{I}_4-\Uu)^2+\ldots\|_{\infty,\infty}\lesssim 2
$$
for $\eps(C_0)$ sufficiently small as required.

Using the relation $\Uu^{-1}\Uu=\mathbf{I}_4$, we can express the angular derivatives of $\Uu^{-1}$ in terms of those of $\Uu$:
\begin{equation*}
\Om_{ij}(\Uu^{-1})=-\Uu^{-1}(\Om_{ij}\Uu)\Uu^{-1}
\end{equation*}
from which
$$
\|\LOR (\Uu^{-1})\|_{\infty,\infty}\lesssim\|\LOR  \Uu\|_{\infty,\infty}\lesssim1
$$
completing the proof of (\ref{Uinfty}).

We now turn to (\ref{49}). Note that this immediately implies the second part of (\ref{orthogonality}), and also that the second part of (\ref{49}) follows from the first thanks to the identity $\dd_\al(\Uu^{-1})=-\Uu^{-1}\cdot\dad \Uu\cdot \Uu^{-1}$.

We will show by induction that
\begin{equation}\label{54}
\|\LOR^{1-\dl(p,q)} \dd_\alpha \Uu_{<K}\|_{p,q}\lesssim 2^{\beta K}C_0c_K
\end{equation}
for all $-N\leq K\leq-9$. Here $\beta=\beta_{p,q}=1-\f{1}{p}-\f{3}{q}$.

For $K=-N$, the claim is trivial. Now suppose that $K>-N$ and the claim has been proven for all smaller $k$. By differentiating the formula defining $\Uu_{K-1}$ we have (neglecting the factor of $\LOR^{1-\dl(p,q)}$ which plays no role)
\begin{align*}
\| \dd_\alpha \Uu_{K-1}\|_{p,q}\lesssim&\| \dd_\alpha \phi^L_{<K-1}\|_{p,q}\| \phi^L_{K-1}\|_{\infty,\infty}\| \Uu_{<K-1}\|_{\infty,\infty}\\
&+\|  \phi^L_{<K-1}\|_{\infty,\infty}\| \dd_\alpha  \phi^L_{K-1}\|_{p,q}\| \Uu_{<K-1}\|_{\infty,\infty}\\
&+\| \phi^L_{<K-1}\|_{\infty,\infty}\| \phi^L_{K-1}\|_{\infty,\infty}\| \dd_\alpha \Uu_{<K-1}\|_{p,q}\lesssim2^{\beta K}C_0c_K
\end{align*}
and the claim follows.
\end{proof}

We can now use this proposition to transform our equation (\ref{matrix_eqn}) into a form in which the only non-trivial term in the forcing cancels. We make the transformation $\Phi^L=\Uu w^L$, so 
$$
\Box\Phi^L=(\Box \Uu)w^L+2\dd_\al \Uu\dd^\al w^L+\Uu\Box w^L
$$
Setting this equal to the right hand side of equation (\ref{matrix_eqn}) and multiplying on the left by $\Uu^{-1}$ we obtain
\begin{equation*}
\Box w^L=-2\Uu^{-1}(\dd_\al \Uu\dau w^L-\Aa ^L_\al\dau\psi^L)-\Uu^{-1}(\Box \Uu)\Uu^{-1}\Phi^L+\Uu^{-1}(error)
\end{equation*}
Note that $\Uu^{-1}(error)=error$ by (\ref{Uinfty}), so we don't need to worry about the final term above.

In order to make use of the fact that $\dd_\al \Uu\simeq \Aa_\al^L \Uu$, we go back a step and use that $\Phi^L=\psi^L+\half(\D_1)-(\D_2)$ to decompose $w^L$ as
$$
w^L=\underbrace{\Uu^{-1}\psi^L}_{w^L_1}+\underbrace{\half \Uu^{-1}(\D_1)}_{w^L_2}-\underbrace{\Uu^{-1}(\D_2)}_{w^L_3}
$$
In particular, since $\dau\psi^L=(\dau \Uu)w^L_1+\Uu\dau w^L_1$, we can write
\begin{align*}
\Box w^L&=-2\Uu^{-1}(\dd_\al \Uu-\Aa ^L_\al\Uu)\dd^\al w^L_1+2\Uu^{-1}\Aa ^L_\al(\dau \Uu)w^L_1-2\Uu^{-1}\dad \Uu\dau(w^L_2+w^L_3)\\
&\quad-\Uu^{-1}(\Box \Uu)\Uu^{-1}\Phi^L+error
\end{align*}

The remainder of this section will be dedicated to showing that the second, third and fourth terms above are all of the form $error$. The remaining term will be studied in Section \ref{cancellation_chapter}.

\begin{proposition}
\begin{equation}\label{101}
\ui \Aa ^L_\al(\dau \Uu)w^L_1=error
\end{equation}
and
\begin{equation}\label{102}
\ui \dad \Uu\dau(w^L_2+w^L_3)=error
\end{equation}
\end{proposition}

\begin{proof}
Again using \eqref{Uinfty} we may neglect the $\ui$. We will also work entirely in standard Strichartz spaces so neglect the angular derivatives.

To bound (\ref{101}) note that by definition of $\Aa ^L_\al$ we have
$$
\|\Aa ^L_\al\|_{\f{2M}{M-1},2M}\lesssim\|\phi^L_{\leq-10}\|_{\infty,\infty}\|\dad\phi^L_{\leq-10}\|_{\f{2M}{M-1},2M} \lesssim C_0\eps
$$
Hence by \eqref{49} it holds
$$
\|\Aa ^L_\al(\dau \Uu) \ui \psi^L\|_{1,2}\lesssim \|\Aa ^L_\al\|_{\f{2M}{M-1},2M}\|\dau \Uu\|_{\f{2M}{M-1},2M}\|\psi^L\|_{M,\f{2M}{M-2}}\lesssim C_0^3\eps^2 c_0
$$
as required.

\bigskip

We now turn to (\ref{102}). 
Let's start with $w_2^L$. Heuristically, we can write $w_2^L\simeq \Uu^{-1}P_0(\phi^L_{\leq-10}(\pgt^L)^T\pgt^L)$. Therefore
\begin{align}
\|\dad \Uu\dau w_2^L\|_{1,2}
\lesssim&\|\dad \Uu \dau(\Uu^{-1})\text{ }P_0(\plt^L(\pgt^L)^T\pgt^L)\|_{1,2}\label{dui}\\
&+\|\dad \Uu\text{ }\Uu^{-1}P_0\dau(\plt^L(\pgt^L)^T\pgt^L)\|_{1,2}\label{ndui}
\end{align}
where
\begin{align*}
\eqref{dui}\lesssim\|\dad \Uu\|_{\f{2M}{M-1},2M}\|\dau(\ui)\|_{\infty,\infty}\|\plt^L\|_{\infty,\infty}\|\pgt^L\|_{\f{2M}{M-1},2M}\|\pgt^L\|_{M,\f{2M}{M-2}}
\lesssim C_0^3\eps^2 c_0 
\end{align*}
On the other hand, for (\ref{ndui}) we have
\begin{align*}
\eqref{ndui}\lesssim&\|\dad \Uu\text{ } \ui P_0(\dau\plt^L(\pgt^L)^T\pgt^L)\|_{1,2}\\
&+\|\dad \Uu\text{ } \ui P_0(\plt^L(\dau\pgt^L)^T\pgt^L)\|_{1,2}\\
\lesssim&\|\dad \Uu\|_{\infty,\infty}\|\dau\plt^L\|_{\f{2M}{M-1},2M}\|\pgt^L\|_{\f{2M}{M-1},2M}\|\pgt^L\|_{M,\f{2M}{M-2}}\\
&+\|\dad \Uu\|_{\f{2M}{M-1},2M}\|\plt^L\|_{\infty,\infty}\|\dau \pgt^L\|_{M,\f{2M}{M-2}}\|\pgt^L\|_{\f{2M}{M-1},2M}\\
\lesssim&C_0^3\eps^2 c_0
\end{align*}

\bigskip

Next we study
$$
w_3^L\simeq- \Uu^{-1}\int_0^1\int_y\check{\psi_0}(y)y^T\,\plt^L(x- \tht y)\na_x\plt^L(x- \tht y)^T\pso^L(x-y)d \tht dy
$$
We have
\begin{align*}
&\left\|\dau \Uu\dad(\ui) \int_0^1\int_y\check{\psi_0}(y)y^T\,\plt^L(x- \tht y)\na_x\plt^L(x- \tht y)^T\pso^L(x-y)d \tht dy\right\|_{1,2}\\
\lesssim&\|\dau \Uu\|_{\f{2M}{M-1},2M}\|\dad(\ui)\|_{\infty,\infty}\|\plt^L\|_{\infty,\infty}\|\na\plt^L\|_{\f{2M}{M-1},2M}\|\pso^L\|_{M,\f{2M}{M-2}}
\end{align*}
which is acceptable. On the other hand,
\begin{align*}
&\left\|\dau \Uu\text{ }\ui \dad\int_0^1\int_y\check{\psi_0}(y)y^T\,\plt^L(x- \tht y)\na_x\plt^L(x- \tht y)^T\pso^L(x-y)d \tht dy\right\|_{1,2}\\
\lesssim&\|\dau \Uu\|_{\f{2M}{M-1},2M}\|\dad\plt^L\|_{\infty,\infty}\|\na\plt^L\|_{\f{2M}{M-1},2M}\|\pso^L\|_{M,\f{2M}{M-2}}\\
&+\|\dau \Uu\|_{\f{2M}{M-1},2M}\|\plt^L\|_{\infty,\infty}\|\dad\na\plt^L\|_{\f{2M}{M-1},2M}\|\pso^L\|_{M,\f{2M}{M-2}}\\
&+\|\dau \Uu\|_{\f{2M}{M-1},2M}\|\plt^L\|_{\infty,\infty}\|\na\plt^L\|_{\f{2M}{M-1},2M}\|\dad\pso^L\|_{M,\f{2M}{M-2}}
\end{align*}
which is also acceptable.
\end{proof}

\bigskip

\begin{proposition}
$$
\ui(\Box \Uu)\ui\Phi^L=error
$$
\end{proposition}
\begin{proof}
We still ignore the $\ui$, however in this case we cannot simply neglect the angular derivatives.
We will show inductively that
$$
\|\LOR (\Box \Uu_{<K}\ui\Phi^L)\|_{1,2}\leq 2^{K/M}C_0^2\eps c_0
$$
for all $-N\leq K\leq-9$, for $M$ sufficiently large.

The claim is trivial for $K=-N$, so suppose it is true up to $K-1\geq-N$. Observe that
\begin{align}
\Box\Uu_{K-1}\simeq&\phi^L_{<K-1}\phi^L_{K-1}\Box\Uu_{<K-1}+\phi^L_{<K-1}  \Box\phi^L_{K-1}\Uu_{<K-1}\nonumber\\
&+\Box\phi^L_{<K-1}\phi^L_{K-1}\Uu_{<K-1}+\dad\phi^L_{<K-1}\dau\phi^L_{K-1}\Uu_{<K-1}\nonumber\\
&+\dad\phi^L_{<K-1}\phi^L_{K-1}\dau\Uu_{<K-1}+\phi^L_{<K-1}\dad\phi^L_{K-1}\dau\Uu_{<K-1}\label{box_U}
\end{align}
The last three terms are the easiest to handle. For instance, using \eqref{Phi_bound} to bound $\Phi$ we have
\begin{align*}
&\|\LOR (\dad\phi^L_{<K-1}\dau\phi^L_{K-1}\Uu_{<K-1}  \cdot\ui\Phi^L)\|_{1,2}\\&\lesssim\|\LOR \dad\phi^L_{<K-1}\|_{\f{2M}{M-1},2M}\|\LOR \dau\phi^L_{K-1}\|_{\f{2M}{M-1},2M}\|\LOR \Phi^L\|_{M,\f{2M}{M-2}}\\
&\lesssim2^{(1-\f{2}{M})K}C_0^3\eps^2 c_0
\end{align*}
which is more than we need. The last 2 terms of \eqref{box_U} can be bounded in the same way.

We now study the terms of \eqref{box_U} involving the wave operator. The first one will be bounded using the induction hypothesis so let's start with the second term. We have to bound
\begin{align*}
\|\LOR (\plko^L\Box\pkmo^L\Uu_{<K-1} \cdot \ui\Phi^L)\| _{1,2}\lesssim\|\LOR (\Box\pkmo^L)\|_{\f{2M}{M+1},\f{2M}{M-1}}\|\LOR \Phi^L\|_{\f{2M}{M-1},2M}
\end{align*}
We then need the following claim.
\begin{claim}
$$
\|\LOR (\Box\pkmo^L)\|_{\f{2M}{M+1},\f{2M}{M-1}}\lesssim 2^{k/M}C_0^2\eps  c_K
$$
\end{claim}
\begin{proof}[Proof of claim]
By scaling, it suffices to prove the claim for $K=1$, so study $\Box\psi^L$ (see \eqref{Boxpsi}).
First consider the wave maps part of $\Box \psi^L$. The action of $L$ does not play an important role here, so for simplicity we only study $\Box\phi$. Using Bernstein's inequality to lower the exponent for the high frequency interactions, we have
\begin{align*}
\|\LOR P_0(\phi\dau\phi^T\dad\phi)\|_{\f{2M}{M+1},\f{2M}{M-1}}&\lesssim\|\LOR P_0(\phi\dau\phi_{>-10}^T\dad\phi_{>-10})\|_{\f{2M}{M+1},\f{2M}{M-1}}\\
&\quad+\|\LOR P_0(\phi\dau\phi_{\leq-10}^T\dad\phi_{>-10})\|_{\f{2M}{M+1},\f{2M}{M-1}}\\
&\quad+\|\LOR P_0(\phi\dau\phi_{\leq-10}^T\dad\phi_{\leq-10})\|_{\f{2M}{M+1},\f{2M}{M-1}}\\
&\lesssim \|\LOR  \dau\phi_{>-10}\|_{\f{4M}{M-1},\f{4M}{M+1}}\|\dau\pgt\|_{\f{4M}{M+3},\f{4M}{M+1}}\\
&\quad+\|\LOR  \dau\plt\|_{\f{4M}{M+3},\f{4M}{M-3}}\|\LOR  \dau\pgt\|_{\f{4M}{M-1},\f{4M}{M+1}}\\
&\quad+\|\LOR  \dau\plt\|_{\f{4M}{M+1},\f{4M}{M-1}}\|\LOR  \dad\plt\|_{\f{4M}{M+1},\f{4M}{M-1}}\\
&\lesssim C_0^2\eps c_0
\end{align*}
provided $M$ is sufficiently large and $\sg$ sufficiently small. The half-wave maps terms can be treated similarly.
\end{proof}
We therefore have
$$
\|\LOR (\plko^L\Box\pkmo^L\Uu_{<K-1} \cdot \ui\Phi^L)\| _{1,2}
\lesssim 2^{K/M}C_0^3\eps^2 c_0
$$
Similarly,
\begin{align*}
\| \LOR (\Box\plko^L\pkmo^L\Uu_{<K-1} \cdot \ui\Phi^L)\|_{1,2}
\lesssim\sum_{J<K-1}2^{J/M}C_0^2\eps c_0\lesssim 2^{K/M}C_0^3\eps^2 c_0
\end{align*}
and lastly by the induction hypothesis we have
\begin{align*}
\| \LOR (\plko^L  \pkmo^L  \Box\Uu_{<K-1} \cdot\ui\Phi^L)\|_{1,2}&\lesssim\|\LOR  \pkmo^L\|_{\infty,\infty}\|\LOR (\Box\Uu_{<K-1}\cdot\ui\Phi^L)\|_{1,2}\\
&\lesssim C_0\eps \cdot2^{K/M}C_0^2\eps c_0\lesssim2^{K/M}C_0^3\eps^2 c_0
\end{align*}

Therefore, letting $D$ denote the sum of the implicit constants above and using (\ref{box_U}), we have
\begin{align*}
\|\LOR (\Box \Uu_{<K}\cdot\ui\Phi^L)\|_{1,2}&\leq\|\LOR (\Box \Uu_{<K-1}\cdot\ui\Phi^L)\|_{1,2}+D\cdot 2^{K/M}C_0^3\eps^2 c_0\\
&\leq 2^{K/M}C_0^2\eps c_0(2^{-1/M}+DC_0\eps)
\end{align*}
Hence choosing $\eps\ll(DC_0)^{-1}$ completes the induction.
\end{proof}

\bigskip

Combining these two propositions with the equation previously obtained for $w^L$, we arrive at
\begin{equation*}
\Box w^L=-2\ui(\dd_\al \Uu-\Aa ^L_\al \Uu)\dd^\al w^L_1+error
\end{equation*}

\section{The (very low-low-high) cancellation}\label{cancellation_chapter}
In this chapter, we will show that
\begin{equation}\label{difficult=error}
\ui(\dd_\al \Uu-\Aa ^L_\al \Uu)\dd^\al w^L_1=error,
\end{equation}
finally doing away with the difficult $\lowestlh$ frequency interactions in the wave maps source term.

As usual we may neglect the $\ui$.
The first step in proving (\ref{difficult=error}) is to use the telescoping identity
\begin{align*}
\sum_{-N<k\leq-10}(\Aa^L_{\alpha,\leq k}\Uu_{\leq k}-\Aa^L_{\alpha,<k}U_{<k})
=&\Aa^L_\alpha \Uu-\Aa^L_{\alpha,\leq-N}
\end{align*}
where
$$
\Aa^L_{\alpha,\leq k}:=\Aa^L_{\alpha,<k+1}:=
\begin{pmatrix}
-\phi_{\leq k}\dd_\alpha \phi_{\leq k}^T & 0\\
-(\phi_{\leq k}\dad( L_n\phi )_{\leq k}^T+( L_n\phi )_{\leq k}\dd_\alpha \phi_{\leq k}^T)  & -\phi_{\leq k}\dd_\alpha \phi_{\leq k}^T \,\mathbf{I}_3
\end{pmatrix}
$$
(so $\Aa^L_\al=\Aa^L_{\al,\leq-10}$) to write
\begin{equation}\label{difference}
\dd_\alpha \Uu-\Aa^L_\alpha U=\left[\sum_{-N<k\leq-10}\dd_\alpha \Uu_k-(\Aa^L_{\alpha,\leq k}\Uu_{\leq k}-\Aa^L_{\alpha,<k}\Uu_{<k})\right]-\Aa^L_{\alpha,\leq-N}
\end{equation}

We first show that the $\Aa^L_{\alpha,\leq-N}$ part is acceptable:
\begin{lemma}\label{A_alpha_lem}
$$
\Aa^L_{\alpha,\leq-N} \dd^\al w^L_1=error
$$
\end{lemma}

\begin{proof}
We have
\begin{equation*}
\|\LOR  (\Aa^L_{\alpha,\leq-N} \dd^\al w^L_1)\|_{1,2}
\lesssim\|\LOR  \Aa^L_{\alpha,\leq-N}\|_{1,\infty}\|\LOR  \dd^\al w^L_1\|_{\infty,2}
\end{equation*}
Then
\begin{align*}
\|\LOR \Aa^L_{\alpha,\leq-N}\|_{1,\infty}\lesssim T\|\LOR  \phi^L_{\leq -N}\|_{\infty,\infty}\|\LOR  \dd_\al\phi^L_{\leq -N}\|_{\infty,\infty}\lesssim 2^{-N}TC_0\eps
\end{align*}
and by the identity
\begin{equation}\label{dw1}
\dd^\al w^L_1=\Uu^{-1}\dd^\al\psi^L-\Uu^{-1}(\dd^\al\Uu)\Uu^{-1}\psi^L
\end{equation}
also
\begin{equation}\label{w_bound}
\|\LOR  \dd^\al w^L_1\|_{\infty,2}\lesssim C_0c_0
\end{equation}
The result is now immediate upon taking $N(T,C_0)$ sufficiently large.
\end{proof}

Next we study the sum in (\ref{difference}). It is here that we observe the critical cancellation of the $\plk\dau\pk^T$ terms. Indeed, as in \cite{Tao_1} we may write
\begin{align}\label{exp1}
\dad\Uu_k&=
\begin{pmatrix}
-\dad\phi_{<k}\phi_k^T&0\\
-(\dad\phi_{<k}( L_n\phi )^T_k+\dad( L_n\phi )_{<k}\phi_k^T) & -\dad\phi_{<k}\pk^T\,\mathbf{I}_3
\end{pmatrix}
\Uu_{<k}\nonumber\\
&\quad+
\begin{pmatrix}
-\phi_{<k}\dad\phi_k^T&0\\
-(\phi_{<k}\dad( L_n\phi )^T_k+( L_n\phi )_{<k}\dad\phi_k^T) & -\phi_{<k}\dad\pk^T\,\mathbf{I}_3
\end{pmatrix}
\Uu_{<k}\nonumber\\
&\quad+
\begin{pmatrix}
-\phi_{<k}\phi_k^T&0\\
-(\phi_{<k}( L_n\phi )^T_k+( L_n\phi )_{<k}\phi_k^T) & -\phi_{<k}\pk^T\,\mathbf{I}_3
\end{pmatrix}
\dad\Uu_{<k}
\end{align}
and
\begin{align}\label{exp2}
\Aa^L_{\alpha,\leq k}\Uu_{\leq k}-\Aa^L_{\alpha,<k}\Uu_{<k}&=
\begin{pmatrix}
-\phi_{< k}\dd_\alpha \phi_{ k}^T & 0\\
-(\phi_{< k}\dad( L_n\phi )_{ k}^T+( L_n\phi )_{< k}\dd_\alpha \phi_{ k}^T)  & -\phi_{< k}\dd_\alpha \phi_{ k}^T \,\mathbf{I}_3
\end{pmatrix}
\Uu_{<k}\nonumber\\
&\quad+
\begin{pmatrix}
-\phi_{ k}\dd_\alpha \phi_{\leq k}^T & 0\\
-(\phi_{ k}\dad( L_n\phi )_{\leq k}^T+( L_n\phi )_{ k}\dd_\alpha \phi_{\leq k}^T)  & -\phi_{ k}\dd_\alpha \phi_{\leq k}^T \,\mathbf{I}_3
\end{pmatrix}
\Uu_{<k}\nonumber\\
&\quad+
\begin{pmatrix}
-\phi_{\leq k}\dd_\alpha \phi_{\leq k}^T & 0\\
-(\phi_{\leq k}\dad( L_n\phi )_{\leq k}^T+( L_n\phi )_{\leq k}\dd_\alpha \phi_{\leq k}^T)  & -\phi_{\leq k}\dd_\alpha \phi_{\leq k}^T \,\mathbf{I}_3
\end{pmatrix}
\Uu_{k}
\end{align}
Crucially, the second line in \eqref{exp1} and the first line in \eqref{exp2} cancel and we are left with
\begin{align*}
\dd_\alpha \Uu_k-(\Aa^L_{\alpha,\leq k}\Uu_{\leq k}-\Aa^L_{\alpha,<k}\Uu_{<k})=&
\begin{pmatrix}
-\dad\phi_{<k}\phi_k^T&0\\
-(\dad\phi_{<k}( L_n\phi )^T_k+\dad( L_n\phi )_{<k}\phi_k^T) & -\dad\phi_{<k}\pk^T\,\mathbf{I}_3
\end{pmatrix}
\Uu_{<k}\\
&\quad+
\begin{pmatrix}
-\phi_{<k}\phi_k^T&0\\
-(\phi_{<k}( L_n\phi )^T_k+( L_n\phi )_{<k}\phi_k^T) & -\phi_{<k}\pk^T\,\mathbf{I}_3
\end{pmatrix}
\dad\Uu_{<k}\\
&\quad-
\begin{pmatrix}
-\phi_{ k}\dd_\alpha \phi_{\leq k}^T & 0\\
-(\phi_{ k}\dad( L_n\phi )_{\leq k}^T+( L_n\phi )_{ k}\dd_\alpha \phi_{\leq k}^T)  & -\phi_{ k}\dd_\alpha \phi_{\leq k}^T\,\mathbf{I}_3
\end{pmatrix}
\Uu_{<k}\\
&\quad-
\begin{pmatrix}
-\phi_{\leq k}\dd_\alpha \phi_{\leq k}^T & 0\\
-(\phi_{\leq k}\dad( L_n\phi )_{\leq k}^T+( L_n\phi )_{\leq k}\dd_\alpha \phi_{\leq k}^T) & -\phi_{\leq k}\dd_\alpha \phi_{\leq k}^T \,\mathbf{I}_3
\end{pmatrix}
\Uu_{k}
\end{align*}
\eqref{difficult=error} is therefore implied by the following result.
\begin{lemma}\label{key_lemma}
\begin{align}
&\sum_{-N<k\leq-10} 
\begin{pmatrix}
\dad\phi_{<k}\phi_k^T&0\\
(\dad\phi_{<k}( L_n\phi )^T_k+\dad( L_n\phi )_{<k}\phi_k^T) & \dad\phi_{<k}\pk^T\,\mathbf{I}_3
\end{pmatrix}
\Uu_{<k}\,\dau w_1^L
=error
\label{1k=error'}\\
&\sum_{-N<k\leq-10}
\begin{pmatrix}
\phi_{<k}\phi_k^T&0\\
(\phi_{<k}( L_n\phi )^T_k+( L_n\phi )_{<k}\phi_k^T) & \phi_{<k}\pk^T\,\mathbf{I}_3
\end{pmatrix}
\dad\Uu_{<k}\,\dau w_1^L
=error
\label{2k=error'}\\
&\sum_{-N<k\leq-10}
\begin{pmatrix}
\phi_{ k}\dd_\alpha \phi_{\leq k}^T & 0\\
(\phi_{ k}\dad( L_n\phi )_{\leq k}^T+( L_n\phi )_{ k}\dd_\alpha \phi_{\leq k}^T)  & \phi_{ k}\dd_\alpha \phi_{\leq k}^T\,\mathbf{I}_3
\end{pmatrix}
\Uu_{<k}\,\dau w_1^L
=error
\label{3k=error'}\\
&\sum_{-N<k\leq-10}
\begin{pmatrix}
\phi_{\leq k}\dd_\alpha \phi_{\leq k}^T & 0\\
(\phi_{\leq k}\dad( L_n\phi )_{\leq k}^T+( L_n\phi )_{\leq k}\dd_\alpha \phi_{\leq k}^T)  & \phi_{\leq k}\dd_\alpha \phi_{\leq k}^T \,\mathbf{I}_3
\end{pmatrix}
\Uu_{k}\,\dau w_1^L
=error
\label{4k=error'}
\end{align}
\end{lemma}

\begin{proof}\hfill
\begin{enumerate}[leftmargin=*]
\item Proof of \eqref{1k=error'}, \eqref{3k=error'}. These two inequalities are to all intents and purposes the same, so we consider only \eqref{1k=error'}. Recall that $\dau w_1^L=\Uu^{-1}\dau\psi^L-\Uu^{-1}(\dau\Uu) \Uu^{-1}\psi^L$ where $\Uu^{-1}$ has the form
$$
\Uu^{-1}=
\begin{pmatrix}
\bigast & 0\\
\bigast & \bigast
\end{pmatrix}
$$
and $\Uu_{<k}$ has the same structure. We first study the part of $\dau w_1^L$ involving $\dau\psi^L$. Write
$$
\Uu_{<k}\Uu^{-1}=
\begin{pmatrix}
U_1 & 0\\
U_2 & U_1\mathbf{I}_3
\end{pmatrix}
$$
Then expanding the matrix product we have
\begin{align*}
&\begin{pmatrix}
\dad\phi_{<k}\phi_k^T&0\\
(\dad\phi_{<k}( L_n\phi )^T_k+\dad( L_n\phi )_{<k}\phi_k^T) & \dad\phi_{<k}\pk^T\,\mathbf{I}_3
\end{pmatrix}
\Uu_{<k}\Uu^{-1}\dau \psi^L\\
&=
\begin{pmatrix}
\dad\phi_{<k}\phi_k^T U_1& 0\\
(\dad\phi_{<k}( L_n\phi )^T_k+\dad( L_n\phi )_{<k}\phi_k^T) U_1+\dad\phi_{<k}\pk^T U_2 & \dad\phi_{<k}\pk^T U_1\,\mathbf{I}_3
\end{pmatrix}
\begin{pmatrix}
\dau\psi\\
\dau\psi^L_n
\end{pmatrix}
\end{align*}

Further expanding this product, it remains to study the following:
\medskip
\begin{enumerate}
\item $\dad\phi_{<k}\pk^T U_i \dau\psi$, $i=1,2$.
\item $\dad\phi_{<k}( L_n\phi )^T_k U_1\dau\psi$
\item $\dad( L_n\phi )_{<k}\phi_k^T U_1\dau\psi$
\item $\dad\phi_{<k}\pk^T U_1\dau\psi_n^L$
\end{enumerate}
\medskip
For the rest of this proof we will treat all functions as scalars, even though they are really vector or matrix fields, by working componentwise. This reduction is possible since none of the arguments that follow rely on any geometric structure and in particular Lemma \ref{SIL} held for scalar functions (see remark at end of said lemma). In this spirit, since all the $\LOR  U_i$ are bounded in $L^\infty_{t,x}$ (by Proposition \ref{big_prop}) we may ignore these terms in the above expressions. What's left is treated by direct application of Lemma \ref{SIL}.

Starting with (a), we easily reduce to the following three terms:
\begin{align}
\|\LOR (\dad\phi_{<k}\pk \dau\psi)\|_{1,2}&\lesssim\|\LOR \dad\phi_{<k}\cdot\pk \cdot\dau\psi\|_{1,2}\label{AaA}\\
&\quad+\|\dad\phi_{<k}\cdot\LOR \pk \cdot\dau\psi\|_{1,2}\label{BbB}\\
&\quad+\|\dad\phi_{<k}\cdot\pk \cdot\LOR \dau\psi\|_{1,2}\label{CcC}
\end{align}
For \eqref{AaA} we apply point 1 of Lemma \ref{SIL} as in \eqref{example} (recalling that $\psi$ is at unit frequency) to see
\begin{align*}
\|\LOR \dad\plk\cdot\pk\cdot\dau\psi\|_{1,2}
\lesssim \sum_{j<k}2^{(\half-\f{1}{2M})(j-k)}C_0^2c_jc_k(C_0c_0+\|\dd_t\dau\psi\|_{\infty,2})
\end{align*}
and we have to do just a little work to bound $\|\dd_t\dau\psi\|_{\infty,2}$ in the case $\al=0$. We use the equation to find
$$
\|\dd_t^2\psi\|_{\infty,2}\leq\|\D \psi\|_{\infty,2}+\|\LOR \Box P_0\phi\|_{\infty,2}
$$
where for example
\begin{align}
\| P_0(\phi\dau\phi^T\dad\phi)\|_{\infty,2}&\lesssim\| P_0(\phi\dau\phi_{>-10}^T\dad\phi_{>-10})\|_{\infty,1}+\| P_0(\phi\dau\phi_{\leq-10}^T\dad\phi_{>-10})\|_{\infty,2}\nonumber\\
&\quad+\| P_0(\phi\dau\phi_{\leq-10}^T\dad\phi_{\leq-10})\|_{\infty,2}\nonumber\\
&\lesssim\|\dau\phi_{>-10}\|_{\infty,2}\|\dau\phi_{>-10}\|_{\infty,2}+\|\dau\phi_{\leq-10}\|_{\infty,\infty}\|\dau\phi_{>-10}\|_{\infty,2}\nonumber\\
&\quad+\|\dau\phi_{\leq-10}\|_{\infty,4}\|\dau\phi_{\leq-10}\|_{\infty,4}\nonumber\\
&\ll C_0 c_0\label{dd_tpsi}
\end{align}
Bounding the half-wave maps terms similarly we find 
$$
\|\LOR \dad\plk\cdot\pk\cdot\dau\psi\|_{1,2}\lesssim\sum_{j<k}2^{(\half-\f{1}{2M})(j-k)}C_0^3c_jc_kc_0
$$
which is acceptable when summed over $k\leq-10$.

\eqref{CcC} can be handled in the same way and for \eqref{BbB} we use Point 2 of Lemma \ref{SIL}. (b) and (c) can be treated identically to (a), and (d) is similar upon using points 3 and 4 of Lemma \ref{SIL} rather than 1 and 2 respectively.

The remaining part $(\Uu^{-1}\dau\Uu \Uu^{-1})\psi^L$ of $\dau w^L_1$ can be treated in the same way, since $\Uu^{-1}\dau\Uu \Uu^{-1}$ has the same block structure as $\Uu_{<k}\Uu^{-1}$ and $\psi^L$ is at unit frequency so behaves like $\dau\psi^L$.

\item Proof of \eqref{2k=error'}. Expanding $\dau w_1^L$ as before and restricting to the term $\Uu^{-1}\dau\psi^L$ for simplicity, we have to consider
$$
\left\|\LOR \left[\begin{pmatrix}
\phi_{<k}\phi_k^T&0\\
(\phi_{<k}( L_n\phi )^T_k+( L_n\phi )_{<k}\phi_k^T) & \phi_{<k}\pk^T\,\mathbf{I}_3
\end{pmatrix}
\dad\Uu_{<k}\Uu^{-1}\dau \psi^L
\right]\right\|_{1,2}
$$
Generally speaking, the argument for this term is similar to the previous one, using Lemma \ref{SIL} and the lower triangular structure of the matrices involved to limit the interactions. For the sake of presentation, we will only consider the top left component of the expression above,
$$
\|\LOR (\phi_{<k} \pk^T\dad U_{<k}U^{-1}\dau\psi)\|_{1,2}
$$
We will also restrict to the case where the angular derivative falls on $\pk$, the other cases bring similar. Note that by placing $\phi_{<k}$ and $U^{-1}$ into $L^\infty_{t,x}$ it suffices to consider
$$
\|\LOR \pk\cdot\dad U_{<k}\cdot\dau\psi\|_{1,2}
$$
(working componentwise).
We proceed by induction. Set
$$
R(j):=\|\LOR \phi_k\cdot\dad U_{<j}\cdot\dau \psi\|_{1,2}
$$
\begin{claim}
For all $-N\leq j\leq k$ it holds
$$
R(j)\lesssim 2^{(\half-\f{1}{2M})(j-k)}C_0^3c_jc_kc_0
$$
\end{claim}
\begin{proof}[Proof of claim]
The claim is trivial for $j=-N$ so suppose it is true up to some fixed $-N<j\leq k$.
By definition of $U_j$ we have
\begin{align}
R(j)\leq R(j-1)&+\|\LOR \phi_k\cdot \phi_{<j-1}\phi_{j-1}^T\dad U_{<j-1}\cdot\dau \psi\|_{1,2}\label{Rj11}\\
&+\|\LOR \phi_k\cdot \phi_{<j-1}\dad\phi_{j-1}^TU_{<j-1}\cdot\dau \psi\|_{1,2}\label{Rj22}\\
&+\|\LOR \phi_k\cdot \dad\phi_{<j-1}\phi_{j-1}^TU_{<j-1}\cdot\dau \psi\|_{1,2}\label{Rj33}
\end{align}
For \eqref{Rj11} we pull out $\|\phi_{<j-1}\phi_{j-1}^T\|_{L^\infty_{t,x}}$ to find 
\begin{align*}
\eqref{Rj11}&\lesssim C_0c_{j-1} R(j-1)
\end{align*}
using the induction hypothesis.

For \eqref{Rj22} we place $U_{<j-1}$ and $\phi_{<j-1}$ into $L^\infty_{t,x}$ and apply part 2 of Lemma \ref{SIL} in conjunction with \eqref{dd_tpsi} to bound
\begin{align*}
\|\LOR \phi_k\cdot \phi_{<j-1}\dad\phi_{j-1}^T U_{<j-1}\cdot\dau \psi\|_{1,2}
\lesssim \|\LOR \phi_k\cdot \dad\phi_{j-1}\cdot\dau \psi\|_{1,2}
\lesssim2^{(\half-\f{1}{2M})(j-k)}C_0^3c_jc_kc_0
\end{align*}

Similarly, for \eqref{Rj33} we have 
\begin{align*}
\|\LOR \phi_k\cdot \dad\phi_{<j-1}\phi_{j-1}^TU_{<j-1}\cdot\dau \psi\|_{1,2}
&\lesssim C_0c_j\sum_{l<j-1}\|\LOR \phi_k\cdot \dad\phi_l\cdot\dau \psi\|_{1,2}\\
&\lesssim C_0c_j\sum_{l<j-1}2^{(\half-\f{1}{2M})(l-k)}C_0^3c_lc_kc_0\\
&\lesssim 2^{(\half-\f{1}{2M})(j-k)}C_0^3c_jc_kc_0
\end{align*}

We deduce that, for some constant $D>0$,
\begin{align*}
R(j)&\leq (1+D\cdot C_0\eps)R(j-1)+D\cdot2^{(\half-\f{1}{2M})(j-k)}C_0^3c_j c_kc_0
\end{align*}
and the claim follows upon taking $\eps(C_0)$ sufficiently small.
\end{proof}

With this claim in hand, we have
\begin{align*}
\sum_{-N<k\leq-10}\|\LOR  \pk\cdot\dad U_{<k}\cdot\dau\psi\|_{1,2}&\lesssim\sum_{-N<k\leq-10}C_0^3 c_k^2\eps
\end{align*}
which is as required.

\item Proof of \eqref{4k=error'}. This is another straightforward application of Lemma \ref{SIL}. We again focus only on the top left component of the term, that is
$$
\phi_{\leq k}\dad\phi_{\leq k}^T \cdot U_k \cdot \dau w_1
$$
Expand
$$
U_k:=-
\phi_{<k}\phi_{k}^TU_{<k}
$$
and place $\plk$, $U_{<k}$ and the other $\phi_{\leq k}$ appearing in the term into $L^\infty_{t,x}$ to reduce to bounding
$$
\sum_{-N<k\leq-10}\|\LOR  (\dad\phi_{\leq k}\cdot\pk\cdot\dau w_1)\|_{1,2}
$$
Upon expanding $\dau w_1=U^{-1}\dau\psi+U^{-1}\dau U U^{-1} \psi$ as before, one sees that this can be treated via a direct application of Lemma \ref{SIL} as in part (1) of this proof.
\end{enumerate}
\end{proof}

\section{Putting it all together}\label{summing up}
We have succeeded in reducing our equation to 
\begin{equation*}
\Box w^L=error
\end{equation*}
for $w^L$ defined through $\Phi^L=\Uu w^L$. In order to exploit the linear estimate, we need to check that we still have the correct smallness on the initial data.
\begin{proposition}\label{data_prop}
Let $\phi[0]$ satisfy assumption \eqref{small1}. Then 
$$
\|\LOR\tilde{P}_0w^L[0]\|_{\dot{H}^{3/2}\times\dot{H}^{1/2}}\lesssim c_0
$$
\end{proposition}
\begin{proof}
By \eqref{Uinfty} and \eqref{49} it suffices to show the corresponding bound on $\Phi^L$, which by Propositions \ref{D1_bounded_prop} and \ref{LLH_NT} further reduces to
$$
\|\LOR\psi^L[0]\|_{\dot{H}^{3/2}\times\dot{H}^{1/2}}\lesssim c_0
$$
In the absence of $L$ the bound is immediate. Then for $n\in\{1,2,3\}$ we have
\begin{align*}
\|\LOR P_0(L_n\phi)[0]\|_{\dot{H}^{3/2}\times\dot{H}^{1/2}}&\lesssim\|\LOR P_0(x_n\dd_t\phi(0))\|_{L^2}+\|\LOR P_0(x_n\dd_t^2\phi(0))\|_{L^2}+\|\LOR P_0(\dd_{x_n}\phi(0))\|_{L^2}\\
&\lesssim c_0+\|\LOR P_0(x_n\Box\phi(0))\|_{L^2}
\end{align*}
which is acceptable thanks to \eqref{x_nbox}.
\end{proof}

Remember that our actual goal is to bound $\psi^L$. By Propositions \ref{D1_bounded_prop} and \ref{LLH_NT} we see that
\begin{align}
\|\psi^L\|_{S_0}\lesssim \|\Phi^L\|_{S_0}+\|(\D_1)\|_{S_0}+\|(\D_2)\|_{S_0}\lesssim\|\Uu w^L\|_{S_0}+C_0^2\eps c_0\label{final_eqn}
\end{align}
where by \eqref{Uinfty} and \eqref{49},
$$
\|\Uu w^L\|_{S_0}\lesssim\|w^L\|_{S_0}+\max_{\mathcal{Q}}\|\LOR^{1-\dl(p,q)}w^L\|_{p,q}=:\|w^L\|_{\tilde{S}_0}
$$

Using the linear estimate, Theorem \ref{linear_estimate}, we are now almost done, modulo the fact that $w^L$ is not quite at unit frequency. To get around this, use that $w^L \simeq \Phi^L$ by writing
$$
w^L=\Phi^L-(\Uu-\mathbf{I}_4)w^L
$$
where $\Phi^L=\tilde{P}_0(\Phi^L)$. Then since $\tilde{S}_0$ is equivalent to $S_0$ at unit frequency, we can use Theorem \ref{linear_estimate} to bound
\begin{align*}
\|\tilde{P}_0w^L\|_{\tilde{S}_0}
\lesssim\|\LOR w^L[0]\|_{\dot{H}^{3/2}\x\dot{H}^{1/2}}+C_0^3c_0\eps
\lesssim c_0
\end{align*}
upon taking $\eps(C_0)$ sufficiently small.
On the other hand by Proposition \ref{big_prop} we have
\begin{align*}
\|(1-\tilde{P}_0)w^L\|_{\tilde{S}_0}\lesssim \|(\Uu-\mathbf{I}_4)w^L\|_{\tilde{S}_0}
\lesssim C_0\eps \|w^L\|_{\tilde{S}_0}
\end{align*}

We have thus found $\|w^L\|_{\tilde{S}_0}\lesssim c_0+C_0\eps \|w^L\|_{\tilde{S}_0}$ and taking $\eps(C_0)$ sufficiently small deduce that 
$$
\|w^L\|_{\tilde{S}_0}\lesssim c_0
$$
Plugging this into \eqref{final_eqn} completes the proof of Proposition \ref{reduced_main_prop}, and hence of the global existence of $\phi$.

\section{Proof of local wellposedness}\label{LWP_proof}
This section is devoted to the proof of Theorem \ref{LWP_thm}. The argument is a combination of the scheme from \cite{KS} with standard methods for studying subcritical wave maps (see for instance \cite{KlSe,klainerman1996estimates,selberg1999multilinear,geba_grillakis}), however we run into various technical issues which lengthen the presentation. In the \hyperref[ss1]{first subsection}, we prove the local wellposedness of the differentiated half-wave maps equation \eqref{eqn1.2}, and in the \hyperref[ss2]{second subsection} we prove that this solution indeed solves the original half-wave maps equation for compatible initial data.

Throughout this section, $p\in\Sp^2$ is fixed.

\subsection{Local Wellposedness of the Differentiated Equation \eqref{eqn1.2}.}\label{ss1}
We start by outlining the argument. We will work in the subcritical function space $\xst_1$ defined by the norm\footnote{
As before we say $\phi$
in $\xst_1$ when $\phi\in p+\xst_1$ and write $\|\phi\|_{\xst}$ to meant $\|\phi-p\|_{\xst_1}$. We have a similar statement for $B^s_{2,1}$.}
$$
\|\phi\|_{\xst_1}:=\sum_{k\geq0}\|\pk\|_{\xst}:=\sum_{k\geq0}\|\langle |\tau|+|\xi|\rangle^s\langle||\tau|-|\xi||\rangle^\tht\tilde{\F}(\pk)(\tau,\xi)\|_{L^2_{\tau,\xi}}
$$
for $3/2+\nu>s>3/2$, $\tht>1/2$ and $s-3/2>\tht-1/2$. Here $\tilde{\F}$ denotes the spacetime Fourier transform, and henceforth we denote $\phi_0:=\F^{-1}(\chi(\xi)\hat{\phi}(\xi))$ the low frequency portion of $\phi$, for $\chi$ as in Section \ref{notation_section}. Note that $\xst_1$ controls the Besov norm (see, for example, Proposition 2.7, \cite{geba_grillakis}):
$$
\|\phi\|_{L^\infty_tB^s_{2,1}}+\|\dd_t\phi\|_{L^\infty_t B^{s-1}_{2,1}}\lesssim\|\phi\|_{\xst_1}
$$

The iteration argument, inspired by the scheme of \cite{KS}, is then as follows.
\begin{enumerate}
\item Set $\phi^{(0)}=p$, the limit of the initial data at infinity.
\item Construct $\phi^{(1)}$ as the local solution to the wave maps equation by iteration in the space $\xst_1$. This solution lies on the sphere, satisfies $\|\phi^{(1)}\|_{\xst_1}\leq 2C\|\phi[0]\|_{B^s_{2,1}\x B^{s-1}_{2,1}}$ (which may be large), and has minimal time of existence depending only on $\|\phi[0]\|_{B^s_{2,1}\x B^{s-1}_{2,1}}$. By standard persistence of regularity arguments, it can be seen that $\phi^{(1)}$ is smooth.
\item Suppose that for all $1\leq j<J$ we have found smooth $\phi^{(j)}\in \xst_1$ solving
\begin{equation*}
\begin{cases}
\Box\puj=-\puj\dau\puj\dad\puj+\Pi_{\puj_{\perp}}(HWM(\pujmo))\\
\puj[0]=\phi[0]
\end{cases}
\end{equation*}
on some interval $[0,T(\|\phi[0]\|_{B^s_{2,1}\x B^{s-1}_{2,1}})]$ which lies on the sphere and satisfies $\|\phi^{(j)}\|_{\xst_1}\leq 2C\|\phi[0]\|_{B^s_{2,1}\x B^{s-1}_{2,1}}$ ($j=1$ is done).
\item Construct $\phi^{(J)}$ as the local solution to 
\begin{equation}\label{it_eqn}
\begin{cases}
\Box\puJ=-\puJ\dau\puJ\dad\puJ+\tilde{\Pi}_{\puJ_{\perp}}(HWM(\puJmo))\\
\puJ[0]=\phi[0]
\end{cases}
\end{equation}
on the same time interval with $\|\puJ\|_{\xst_1}\leq 2C\|\phi[0]\|_{B^s_{2,1}\x B^{s-1}_{2,1}}$. Here
\begin{equation}\label{modified_projection}
\tilde{\Pi}_{{\tp}_{\perp}}(\phi):=\phi-(\phi\cdot g(\tp))g(\tp)
\end{equation}
for $g$ a smooth, compactly supported version of $\tp/\|\tp\|$ equal to that function for $\|\tp\|\simeq1$ but vanishing in a neighbourhood of the origin.
We make this modification since the subcritical argument assumes no smallness on $\puJ$ and there is nothing to stop it crossing the origin, at which point the projection operator is non-smooth.

Note that we must evaluate the half-wave maps terms in $\puJmo$ rather than $\puJ$ in order to control this part of the forcing, since we do not yet know that $\puJ$ lies on the sphere.
\item To close the iteration we must show that $\puJ$ lies on the sphere. We can only show this for the true projection $\Pi_{\puJ_{\perp}}$ (by the same argument as in \cite{KS}), as opposed to the modified version above. We therefore need $\|\puJ-p\|_\infty\ll1$ so that the two projections coincide, and this follows from the assumed smallness of the data in the critical space. In particular, after constructing each iterate $\puJ$ we use the main argument of this paper (Chapters \ref{reduction_chap}-\ref{cancellation_chapter}) to show that the iterate remains small in the critical space, and so in $p+L^\infty_{t,x}$ as required.\footnote{Two very minor adaptations are required in Chapters \ref{reduction_chap}-\ref{cancellation_chapter} to handle the current situation. Firstly, \eqref{it_eqn} involves both $\puJ$ and $\puJmo$, however this causes no issues since the wave maps and half-wave maps source terms are treated wholly independently in the main argument, and we may assume iteratively that \eqref{result} already holds for $\puJmo$. The second adaptation is that the terms $HWM_2$ are now accompanied by a projection which must be taken into account in the steps where one iterates the equation (e.g. in Section \ref{normal_forms_chapter}), however this is easily seen to be unproblematic using Moser estimates.}
\item Having constructed the sequence $\puJ$, we take the limit $J\to\infty$ in $\xst_1$ to obtain a solution $\phi\in\xst_1$ solving the half-wave maps equation on the interval $[0,T(\|\phi[0]\|_{B^s_{2,1}\x B^{s-1}_{2,1}})]$. The higher regularity of the solution follows by standard arguments.
\end{enumerate}

\bigskip

To carry out the above argument, we must establish the necessary estimates for the iteration steps (2), (4) and (6). The following linear estimate shows that we must control the forcing in the space $\xstmo_1$.
\begin{lemma}[Linear Estimate, see e.g. Theorems 2.9 and 2.10, \cite{geba_grillakis}]\label{lin_lem}
Fix $\phi[0]=(\phi_0,\phi_1)\in B^s_{2,1}\x B^{s-1}_{2,1}$ and define the solution operator 
\begin{equation*}
\Phi(F):=p+\eta(t)\left(S_{-p}(\phi[0])-\int_0^t\f{\sin((t-s)\rd)}{\rd}F(s)ds\right)
\end{equation*}
for 
$$
S_{-p}(\phi[0])(t):=\cos(t\sqrt{-\D})(\phi_0-p)+\f{\sin(t\rd)}{\rd}\phi_1
$$
and $\eta:\R\to[0,1]$ a smooth, compactly supported cut-off function equal to 1 on $[-1,1]$. Hence $\Phi(F)$ solves the wave equation with data $\phi[0]$ and forcing $F$ on the interval $[-T,T]$.
It holds
$$
\|\Phi(F)\|_{\xst_1}\lesssim\|\phi[0]\|_{B^s_{2,1}\x B^{s-1}_{2,1}}+\|F\|_{\xstmo_1}
$$
\end{lemma}

\subsubsection{Subcritical Multilinear Estimates}
The following proposition contains the multilinear estimates needed for the iteration argument. We note that the result for the wave maps source terms is considered standard, however we were unable to find a proof in the literature including the necessary gain in $T^\eps$, so we provide a proof for completeness.

Using this proposition in conjunction with the linear estimate Lemma \ref{lin_lem} in the scheme outlined at the beginning of this section, the proof of local wellposedness is complete. Henceforth denote $s=3/2+s'$, $\tht=1/2+\tht'$ for $0<\tht'<s'<\nu$.
\begin{proposition}\label{all_ML_estimates}
Fix $0<T<1$ and set $\eta_T(t):=\eta(T^{-1}t)$ for $\eta$ as in Lemma \ref{lin_lem}. For $\phi,\,\tp\in\xst_1$, define (suppressing the dependence on $\tp$),
\begin{gather}
    WM(\phi)=-\phi \dd^\al\phi\dad\phi\label{tt1}\\
    HWM_1(\phi)=\widetilde{\Pi}_{\tilde{\phi}^{\perp}}(\du)\,(\phi\cdot\du)\label{tt2}\\
    HWM_2(\phi)=\wip[\phi\x(\dhalf(\phi\x{\dhalf\phi})-\phi\x(-\D)\phi)]\label{tt3}
\end{gather}
with $\widetilde{\Pi}_{\tilde{\phi}_{\perp}}(\phi)$ as in \eqref{modified_projection}. Then there exists $\eps(s',\tht')>0$ and a function $C(\|\tp\|_{\xst_1})$ growing polynomially in $\|\tp\|_{\xst_1}$ such that for $\Tt$ any of the trilinear terms \eqref{tt1}-\eqref{tt3} it holds
\begin{equation}\label{ML_estimates}
    \|\eta_T\cdot\Tt(\phi)\|_{\xstmo_1}\lesssim C(\|\tp\|_{\xst_1})\, T^{\eps}(1+\|\phi\|_{\xst_1})\|\phi\|_{\xst_1}^2
\end{equation}
We also have the difference estimates
\begin{align*}
    &\|\eta_T(\Tt(\phi^{(1)})-\Tt(\phi^{(2)}))\|_{\xstmo_1}\\
    &\hspace{0.5em}\lesssim C(\|\tp^{(1)}\|_{\xst_1},\|\tp^{(2)}\|_{\xst_1}) T^{\eps}(\|\phi^{(1)}-\phi^{(2)}\|_{\xst_1}(1+\max_j\|\phi^{(j)}\|_{\xst_1})\max_j\|\phi^{(j)}\|_{\xst_1}\\
    &\hspace{14em}+\|\tp^{(1)}-\tp^{(2)}\|_{\xst_1}(1+\max_j\|\phi^{(j)}\|_{\xst_1})\max_j\|\phi^{(j)}\|_{\xst_1}^2)
\end{align*}
for all $i$ and a similar function $C$.
\end{proposition}

We restrict to proving the multilinear estimates \eqref{ML_estimates}, the difference estimates being similar. We will constantly use the following well-known transferred Strichartz estimate, see for example Proposition 26 \cite{burzio} for a proof.
\begin{lemma}[Strichartz embedding]\label{str_lem}
Let $p,q\geq2$, $\f{1}{p}+\f{1}{q}\leq\f{1}{2}$, $(p,q)\neq(2,\infty)$. Then for any $\tht>1/2$, $s=3/2+s'>3/2$ it holds
\begin{equation}
\|P_k\phi\|_{p,q}\lesssim_{p,q} 2^{-(\f{1}{p}+\f{3}{q}+s')k}\|P_k\phi\|_{\xst}\label{str_lem}
\end{equation}
for all $k>0$.
\end{lemma}
It follows that the norm $S^s$ defined as the $\ell^1_{k\geq0}$ sum over
\begin{align*}
\|\phi_k\|_{S^s_k}:=\max_{(p,q)}2^{(\f{1}{p}+\f{3}{q}-1+s')k}\|\na_{t,x}\phi_k\|_{p,q}, \quad\quad \|\phi_0\|_{S^s_0}:=\max_{\substack{(p,q),\\\f{1}{p}+\f{3}{q}<1-s'}}\|\na_{t,x}\phi_0\|_{p,q}+\|\phi_0\|_{\infty,\infty}
\end{align*}
is controlled by the $\xst_1$ norm whenever the maxima are taken over a finite number of standard Strichartz pairs (taking slight care including the $\na_t$ at high modulations):
\begin{align*}
\|\phi_k\|_{S^s_k}\lesssim\|\phi_k\|_{\xst}, \quad\quad \|\phi_0\|_{S^s_0}\leq 1+\|\phi_0\|_{\xst}
\end{align*}
The restriction on $(p,q)$ in the low frequency case results from the fact that \eqref{str_lem} holds only at frequencies localised away from the origin. 

We start by recording the following key bilinear estimates. Such estimates first appeared in \cite{KlSe}, however the reader may consult Theorems 2.11 and 2.12 of \cite{geba_grillakis} for a textbook proof of (1)-(2). For the third estimate see Lemma 2.11, \cite{tao_book}.
\begin{lemma}[Bilinear estimates]\label{stan_lem}
Let $s'>\tht'>0$. Then the following hold \cite{KlSe}:
\begin{enumerate}
\item $\|\vp\cdot\phi\|_{X^{s,\tht}}\lesssim\|\vp\|_{\xst}\|\phi\|_{\xst}$
\item $\|\vp\cdot F\|_{X^{s-1,\tht-1}}\lesssim\|\vp\|_{\xst}\|F\|_{\xstmo}$
\end{enumerate}
Moreover, for $\tilde{s}\in\R$, $-\half<\tilde{\tht}<\half$, it holds
$$
\|\eta_T\vp\|_{X^{\tilde{s},\tilde{\tht}}}\lesssim_\eta\|\vp\|_{X^{\tilde{s},\tilde{\tht}}}
$$
uniformly in $T\in(0,1)$.
\end{lemma}

We also need estimates to control the projection from the half-wave maps terms in the $\xst$ spaces. We introduce the notation
$$
Q_kF:=\tilde{\F}^{-1}(\chi_k(||\tau|-|\xi||)\tilde{\F}(F)(\tau,\xi))
$$
which decomposes the modulation of a function on dyadic scales. We again use $Q_0F:=\tilde{\F}^{-1}(\chi(||\tau|-|\xi||)\tilde{\F}(F)(\tau,\xi))$ to cover the low modulations. Observe the following modulation Bernstein-type estimate:
\begin{equation}\label{MB}
\|Q_jP_k\vp\|_{p,q}\lesssim 2^{3({\half}-\f{1}{q})k}2^{(\half-\f{1}{p})j}\|P_kQ_j\vp\|_{2,2}\quad\quad (p,q\geq2)\tag{MB}
\end{equation}

\begin{lemma}[Subcritical projection estimate]\label{sc_projection_lemma}
Fix $0<s'<1/4$ and let $\phi,\,\tp\in S^s$. Let $g:\R^3\rightarrow \R^3$ be smooth with bounded derivatives and $(p,q)$ an admissible pair. There exists a constant $C(\|\phi\|_{S^s})$ growing polynomially in $\|\phi\|_{S^s}$ such that for any $k>0$ we have
\begin{gather}
\|P_kg(\tp)\|_{p,q}+\tk\|P_k\dd_t(g(\tp))\|_{p,q}\lesssim C(\|\tp\|_{S^s}) 2^{-(\f{1}{p}+\f{3}{q}+s')k}\label{sc_moser}
\end{gather}
and
\begin{equation}\label{something}
\|P_k(\widetilde{\Pi}_{\tp^{\perp}}(\dhalf \phi))\|_{p,q}\lesssim C(\|\tp\|_{S^s})2^{(1-\f{1}{p}-\f{3}{q}-s')k} \sum_{\ko\geq0}2^{-\sg|k-\ko|}\|P_{\ko}\phi\|_{\Sko^s}
\end{equation}
for some $\sg(s',p,q)>0$. The second estimate \eqref{something} also holds for $k=0$.
\end{lemma}
We omit the proof of this Lemma which is analogous to that in the critical case, see Lemmas \ref{Moser} and \ref{projection_lemma}.
From these estimates we can deduce similar bounds in the $\xst$ spaces, discussed further in Appendix \ref{appendixA}.
\begin{lemma}\label{xsb_projection_estimate}
Let $0<s'<1/4$, $\phi,\,\tp\in \xst_1$. Then for $(j,k)\neq(0,0)$ we have the Moser inequality
\begin{equation}\label{xst_moser}
\|P_kQ_j g(\tp)\|_{\xst}\lesssim C(\|\tp\|_{\xst_1})
\end{equation}
and the projection estimate
\begin{equation}\label{xst_proj_estimate}
\|P_k Q_j(\widetilde{\Pi}_{\tilde{\phi}^{\perp}}\dphi)\|_{\xst}\lesssim C(\|\tp\|_{\xst_1})\left(2^k\sum_{k'\geq 0}2^{-\sg|k-k'|}\|\phi_{k'}\|_{\xst}+\dl_{j\gg k}\cdot2^j\sum_{k'\gtrsim j}2^{-\sg|j-k'|}\|\phi_{k'}\|_{\xst}\right)
\end{equation}
which also holds for $(j,k)=(0,0)$. Here $\dl_{j\gg k}=1$ if $j\geq k+20$ and $0$ otherwise, and $C(\|\tp\|_{\xst_1})$ is a constant which grows polynomially in $\|\tp\|_{\xst_1}$.

Furthermore, the projections are continuous on $\xstmo_1$: for $j,\, k\geq0$ it holds
\begin{equation}\label{contp}
    \|P_kQ_j \wip F\|_{\xstmo}\lesssim C(\|\tp\|_{\xst_1})\sum_{r\geq0}\sum_{l\geq0}2^{-\sg(|r-k|+|l-j|)}\|P_rQ_lF\|_{\xstmo}
\end{equation}
\end{lemma}
\begin{remark}\label{red-remark}\emph{
    The continuity property \eqref{contp} allows us to neglect the outer projection $\wip$ appearing in $HWM_2$ when proving  \eqref{ML_estimates}. Henceforth, we therefore redefine
    \begin{equation*}
    HWM_2(\phi):=\phi\x(\dhalf(\phi\x{\dhalf\phi})-\phi\x(-\D)\phi)
    \end{equation*}}
\end{remark}

To prove Proposition \ref{all_ML_estimates} we will use the frequency decompositions
$$
\eta_T\cdot WM(\phi)=-\sum_{\ko,\kt,\kth\geq0}\eta_T \underbrace{(\phi_{k_1} \dd^\al\phi_{k_2}\dad\phi_{k_3})}_{WM_{\ko,\kt,\kth}(\phi)},
$$
$$
\eta_T\cdot HWM_1(\phi)=\sum_{\ko,\kt,\kth\geq0}\eta_T \underbrace{(P_{\kt}(\widetilde{\Pi}_{\tilde{\phi}^{\perp}}\du)\,(\pko\cdot\dphi_{\kth}))}_{HWM_{1;\ko,\kt,\kth}(\phi)},
$$
and
$$
\eta_T\cdot HWM_2(\phi)=\sum_{\ko,\kt,\kth\geq0}\eta_T \underbrace{(\pko\x[\dhalf(\pkt\x{\dhalf\pkth})-\pkt\x(-\D)\pkth])}_{HWM_{2;\ko,\kt,\kth}(\phi)}
$$
We will first deal with all but the $\lhh$ interactions, for which we can get by using Strichartz estimates, the gain in $T$ coming from H\"older's inequality. For this we need the following straightforward bound which tells us that multiplication by a time cut-off does not affect the geometry of the interactions in a serious way.
\begin{lemma}\label{T_lem}
Let $2\leq p\leq\infty$, $l\geq0$. Then it holds
$$
\|P_l^{(t)}\eta_T\|_{L^p_t}\lesssim_{\eta,N} T^{1/p}(2^lT)^{-N}
$$
for any $N>0$. Here $P_l^{(t)}$ is the projection to temporal frequency $\sim 2^l$.
\end{lemma}
Most of the frequency interactions are then handled in the following proposition.
\begin{proposition}\label{loc_prop1}
    Fix $T>0$. Set
    $$
    \Ss_{\ast}:=\{(\ko,\kt,\kth)\in\mathbb{N}_{\geq0}^3:2^{\kt}T>1,\ 2^{\kth}T>1,\ \ko<\max\{\kt,\kth\}-10\}
    $$
    Then there exist $\eps>0$ such that for any $\phi,\,\tp\in \xst_1$ it holds
    $$
    \|\sum_{(\ko,\kt,\kth)\notin\Ss_{\ast}}\eta_T\cdot\Tt_{\ko,\kt,\kth}(\phi)\|_{\xstmo_1}\lesssim C(\|\tp\|_{\xst_1})\,T^{\eps}(1+\|\phi\|_{\xst_1})\|\phi\|_{\xst_1}^2
    $$
    for $\Tt\in\{WM,HWM_1,HWM_2\}$ and $C(\|\tp\|_{\xst_1})$ as in Lemma \ref{xst_proj_estimate}.
\end{proposition}

\begin{proof}
Throughout this proof any implicit constants may depend polynomially on $\|\tp\|_{\xst_1}$. We will only prove the estimates for the wave maps terms, $WM$, the other terms being entirely analogous using Lemma \ref{xsb_projection_estimate} and the fact that 
$$
\|\dhalf(\pkt\x{\dhalf\pkth})-\pkt\x(-\D)\pkth\|_{p,q}\lesssim 2^{\kt+\kth}\|\pkt\|_{p_1,q_1}\|\pkth\|_{p_2,q_2}
$$
for all $\kt,\,\kth\geq0$ (see Lemma \ref{useful_lemma}) and conjugate triples $p^{-1}=p_1^{-1}+p_2^{-1}$, $q^{-1}=q_1^{-1}+q_2^{-1}$.

Fix $k\geq0$ and consider
$$
\sum_{(\ko,\kt,\kth)\notin\Ss_{\ast}}\|P_k(\eta_T\cdot WM_{\ko,\kt,\kth}(\phi))\|_{\xstmo}
$$
We start with the case $\ko\geq\max\{\kt,\kth\}-10$. Note that in this case the whole term is restricted to frequency $P_{\lesssim\ko}$ so we must have $\ko\gtrsim k$. We then consider different cases for the modulation.
\begin{itemize}[leftmargin=*]
\item $P_{\lesssim\ko}Q_{\lesssim\ko}$: Using that $k\lesssim \ko$ followed by the modulation-Bernstein estimate and H\"older's inequality we have
\begin{align*}
&\|P_kQ_{\lesssim\ko}(\eta_T\phi_{k_1}\dau\phi_{k_2}\dad\phi_{k_3})\|_{\xstmo}\\
&\lesssim 2^{(s-1)\ko}\sum_{j\lesssim\ko}2^{(\tht-1)j}2^{j/2}\|\eta_T\|_M\|\phi_{k_1}\|_{\infty,2}\|\dau\phi_{k_2}\|_{\f{2M}{M-1},\infty}\|\dad\phi_{k_3}\|_{\f{2M}{M-1},\infty}\\
&\lesssim T^\f{1}{M} 2^{(s-1)\ko}2^{\tht'\ko}\cdot 2^{-s\ko}\|\phi_{k_1}\|_{\xst}\cdot2^{(\half+\f{1}{2M}-s')(\kt+\kth)}\|\phi_{k_2}\|_{\xst}\|\phi_{k_3}\|_{\xst}\\
&\lesssim T^\f{1}{M} 2^{-(2s'-\tht'-\f{1}{M})\ko}\|\phi_{k_1}\|_{\xst}\|\phi_{k_2}\|_{\xst}\|\phi_{k_3}\|_{\xst}
\end{align*}
We may then take, for example, $\f{1}{M}=\tht'$, $s'>\tht'$ so bound this by 
\begin{align*}
T^{\tht'}2^{-2(s'-\tht')(\ko-k)}&2^{-2(s'-\tht')k}\|\phi_{k_1}\|_{\xst}\|\phi_{k_2}\|_{\xst}\|\phi_{k_3}\|_{\xst}
\end{align*}
which is summable over $\kt\,\kth\leq\ko-10$, $\ko\gtrsim k$, $k\geq0$ as required.
\bigskip

\item  $\sum_{l\gg\ko}P_{\lesssim\ko}Q_l$: In this case one of the four factors must be at modulation (or frequency in the case of $\eta_T$) at least comparable to $2^l$. We study each option separately.
\begin{enumerate}[leftmargin=*]
\item $(P_{\gtrsim l}\eta_T)\,\phi_{k_1}\dau\phi_{k_2}\dad\phi_{k_3}$: Here we use Lemma \ref{T_lem} to see that
\begin{align*}
&\sum_{l\gg\ko}\|P_kQ_l((P_{\gtrsim l}^{(t)}\eta_T)\phi_{k_1}\dau\phi_{k_2}\dad\phi_{k_3})\|_{\xstmo}\\
&\lesssim \sum_{l\gg\ko}2^{(s'+\tht')l}\|P_{\gtrsim l}^{(t)}\eta_T\|_M\|\phi_{k_1}\|_{\infty,2}\|\dau\phi_{k_2}\|_{\f{2M}{M-2},\infty}\|\dau\phi_{k_3}\|_{\infty,\infty}\\
&\lesssim \sum_{l\gg\ko}2^{(s'+\tht')l}\cdot T^{1/M}(2^lT)^{-N}\cdot 2^{-s\ko}\|\phi_{k_1}\|_{\xst}\cdot 2^{(\half+\f{1}{M}-s')\kt}\|\phi_{k_2}\|_{\xst}\cdot 2^{(1-s')\kth}\|\phi_{k_3}\|_{\xst}
\end{align*}
We may then take, e.g., $N=s'+\tht'+\f{1}{2M}$ and $1/2M=s'+2\tht'$ to bound this by
\begin{align*}
&T^{\f{1}{2M}-s'-\tht'}2^{-\ko/2M}2^{-s\ko}2^{(\half+\f{1}{M}-s')\ko}2^{(1-s')\ko}\|\phi_{k_1}\|_{\xst}\|\phi_{k_2}\|_{\xst} \|\phi_{k_3}\|_{\xst}\\
&\lesssim T^{\tht'}2^{-2(s'-\tht')(\ko-k)}2^{-2(s'-\tht')k}\|\phi_{k_1}\|_{\xst}\|\phi_{k_2}\|_{\xst} \|\phi_{k_3}\|_{\xst}
\end{align*}
which is acceptable.
\bigskip

\item $\eta_T\,(Q_{\gtrsim l}\phi_{k_1})\dau\phi_{k_2}\dad\phi_{k_3}$: This time we place $Q_{\gtrsim l}\pko$ directly into $\xst$ to find
\begin{align*}
&\sum_{l\gg\ko}\|P_kQ_l(\eta_T(Q_{\gtrsim l}\phi_{k_1})\dau\phi_{k_2}\dad\phi_{k_3})\|_{\xstmo}\\
&\lesssim \sum_{l\gg\ko}2^{(s'+\tht')l}\|\eta_T\|_M\|Q_{\gtrsim l}\phi_{k_1}\|_{\f{2M}{M-2},2}\|\dau\phi_{k_2}\|_{\infty,\infty}\|\dau\phi_{k_3}\|_{\infty,\infty}\\
&\lesssim T^{1/M}\sum_{l\gg\ko}2^{(s'+\tht')l}2^{l/M}2^{-\tht l}2^{-sl}\|\phi_{k_1}\|_{\xst}\cdot 2^{(1-s')\kt}\|\phi_{k_2}\|_{\xst}\cdot 2^{(1-s')\kth}\|\phi_{k_3}\|_{\xst}\\
&\lesssim T^{1/M}2^{(\f{1}{M}-2s')\ko}\|\phi_{k_1}\|_{\xst}\|\phi_{k_2}\|_{\xst}\|\phi_{k_3}\|_{\xst}
\end{align*}
which is acceptable choosing $1/M=2\tht'$. 
\bigskip

\item $\eta_T\, \phi_{k_1}( Q_{\gtrsim l}\dau\phi_{k_2})\dad\phi_{k_3}$:
\begin{align}
&\sum_{l\gg\ko}\|P_kQ_l(\eta_T\, \phi_{k_1}( Q_{\gtrsim l}\dau\phi_{k_2})\dad\phi_{k_3})\|_{\xstmo}\nonumber\\
&\lesssim \sum_{l\gg\ko}2^{(s'+\tht')l}\|\eta_T\|_M\|\phi_{k_1}\|_{\infty,\infty}\|Q_{\gtrsim l}\dau\phi_{k_2}\|_{\f{2M}{M-2},2}\|\dau\phi_{k_3}\|_{\infty,\infty}\nonumber\\
&\lesssim T^{1/M}\sum_{l\gg\ko}2^{(s'+\tht')l} (\dl_{\ko,0}+\|\phi_{k_1}\|_{\xst})\cdot 2^{(\f{1}{M}-\tht+1-s)l}\|\phi_{k_2}\|_{\xst}\cdot 2^{(1-s')\kth}\|\phi_{k_3}\|_{\xst}\nonumber\\
&\lesssim T^{1/M}2^{(\f{1}{M}-s')\ko}(\dl_{\ko,0}+\|\phi\|_{\xst})\|\phi_{k_2}\|_{\xst}\|\phi_{k_3}\|_{\xst}\label{pt3}
\end{align}
which is acceptable choosing e.g. $1/M=\tht'$.
\bigskip

\item $\eta_T \phi_{k_1}\dau\phi_{k_2}(Q_{\gtrsim l}\dad\phi_{k_3})$: as above.
\end{enumerate}
\end{itemize}
This completes the study of the case $\ko\geq\max\{k_2,k_3\}-10$.
\bigskip

We next turn to the case $\ko<\max\{k_2,k_3\}-10$. WLOG $\kt\geq\kth$. This time the whole term is at frequency $\lesssim 2^{\kt}$ and we must have $k\lesssim\kt$. We first study the case where the whole term has large modulation,
$$
\sum_{l\gg\kt}\|P_{\lesssim\kt}Q_l(\eta_T\phi_{k_1}\dau\phi_{k_2}\dad\phi_{k_3})\|_{\xstmo}
$$
Again, one of the factors must have modulation of order at least $2^l$, so we have four cases to consider.
\begin{itemize}[leftmargin=*]
\item $\sum_{l\gg\kt}P_{\lesssim\kt}Q_l$: 
\begin{enumerate}
\item $P_{\gtrsim l}^{(t)}\eta_T$: In this case we again use Lemma \ref{T_lem} to see that
\begin{align*}
&\sum_{l\gg\kt}\|P_kQ_l((P_{\gtrsim l}^{(t)}\eta_T)\phi_{k_1}\dau\phi_{k_2}\dad\phi_{k_3})\|_{\xstmo}\\
&\lesssim\sum_{l\gg\kt}2^{(s'+\tht')l}\|P_{\gtrsim l}^{(t)}\eta_T\|_M\|\phi_{k_1}\|_{\infty,\infty}\|\dau\phi_{k_2}\|_{\infty,2}\|\dau\phi_{k_3}\|_{\f{2M}{M-2},\infty}\\
&\lesssim\sum_{l\gg\kt}2^{(s'+\tht')l}T^{1/M}(2^lT)^{-N}2^{(1-s)\kt}2^{(\half+\f{1}{M}-s')\kth}(\dl_{\ko,0}+\|\phi_{k_1}\|_{\xst})\|\phi_{k_2}\|_{\xst}\|\phi_{k_3}\|_{\xst}\\
&\lesssim T^{\f{1}{2M}-s'-\tht'}2^{(\f{1}{2M}-2s')\kt}\|\phi_{k_1}\|_{\xst}\|\phi_{k_2}\|_{\xst}\|\phi_{k_3}\|_{\xst} 
\end{align*}
where we again used $N=\f{1}{2M}+s'+\tht'$ and $\kth\leq\kt$. Choosing $M$ such that $s'+\tht'<\f{1}{2M}<2s'$ we obtain the result.
\bigskip

\item $Q_{\gtrsim l}\phi_{k_1}$: This is a direct application of H\"older's inequality. Placing all three factors of $\phi$ into Strichartz spaces we have
\begin{align*}
&\sum_{l\gg\kt}\|P_kQ_l(\eta_T(Q_{\gtrsim l}\phi_{k_1})\dau\phi_{k_2}\dad\phi_{k_3})\|_{\xstmo}\\
&\lesssim\sum_{l\gg\kt}2^{(s'+\tht')l}\|\eta_T\|_M\| Q_{\gtrsim l}\phi_{k_1}\|_{\f{2M}{M-2},\infty}\|\dau\phi_{k_2}\|_{\infty,2}\|\dau\phi_{k_3}\|_{\infty,\infty}\\
&\lesssim T^{1/M}\sum_{l\gg\kt}2^{(s'+\tht')l}\cdot2^{3\ko/2}2^{(\f{1}{M}-\tht)l}2^{-sl}\|\phi_{k_1}\|_{\xst}\cdot2^{(1-s)\kt}\|\phi_{k_2}\|_{\xst}\cdot2^{(1-s')\kth}\|\phi_{k_3}\|_{\xst}\\
&\lesssim T^{1/M}2^{(\f{1}{M}-2s')\kt}\|\phi_{k_1}\|_{\xst}\|\phi_{k_2}\|_{\xst}\|\phi_{k_3}\|_{\xst}
\end{align*}
which is acceptable for e.g. $1/M=2\tht'$. 
\bigskip

\item $Q_{\gtrsim l}\dau\phi_{k_2}$:
\begin{align*}
&\sum_{l\gg\kt}\|P_kQ_l(\eta_T\phi_{k_1}(Q_{\gtrsim l}\dau\phi_{k_2})\dad\phi_{k_3})\|_{\xstmo}\\
&\lesssim\sum_{l\gg\kt}2^{(s'+\tht')l}\|\eta_T\|_M\|\phi_{k_1}\|_{\infty,\infty}\|Q_{\gtrsim l}\dau\phi_{k_2}\|_{\f{2M}{M-2},2}\|\dau\phi_{k_3}\|_{\infty,\infty}\\
&\lesssim T^{1/M}\sum_{l\gg\kt}2^{(s'+\tht')l}\cdot(\dl_{\ko,0}+\|\phi_{k_1}\|_{\xst})\cdot2^{(\f{1}{M}-\tht)l}2^{(1-s)l}\|\phi_{k_2}\|_{\xst}\cdot2^{(1-s')\kth}\|\phi_{k_3}\|_{\xst}\\
&\lesssim T^{1/M}2^{(\f{1}{M}-s')\kt}(\dl_{\ko,0}+\|\phi_{k_1}\|_{\xst})\|\phi_{k_2}\|_{\xst}\|\phi_{k_3}\|_{\xst}
\end{align*}
which is acceptable for $\f{1}{M}<s'$.
\bigskip

\item $Q_{\gtrsim l}\dad\phi_{k_3}$:
\begin{align*}
&\sum_{l\gg\kt}\|P_kQ_l(\eta_T\phi_{k_1}\dau\phi_{k_2}(Q_{\gtrsim l}\dad\phi_{k_3}))\|_{\xstmo}\\
&\lesssim\sum_{l\gg\kt}2^{(s'+\tht')l}\|\eta_T\|_M\|\phi_{k_1}\|_{\infty,\infty}\|\dau\phi_{k_2}\|_{\infty,2}\|Q_{\gtrsim l}\dau\phi_{k_3}\|_{\f{2M}{M-2},\infty}\\
&\lesssim T^{1/M}\sum_{l\gg\kt}2^{(s'+\tht')l}\cdot(\dl_{\ko,0}+\|\phi_{k_1}\|_{\xst})\cdot2^{(1-s)\kt}\|\phi_{k_2}\|_{\xst}\cdot2^{3\kth/2}2^{(\f{1}{M}-\tht)l}2^{(1-s)l}\|\phi_{k_3}\|_{\xst}\\
&\lesssim T^{1/M}2^{(\f{1}{M}-s')\kt}(\dl_{\ko,0}+\|\phi_{k_1}\|_{\xst})\|\phi_{k_2}\|_{\xst}\|\phi_{k_3}\|_{\xst}
\end{align*}
which is acceptable for $\f{1}{M}<s'$.
\end{enumerate}
\bigskip

\item $P_{\lesssim\kt}Q_{\lesssim\kt}:$ It remains to study the term with overall modulation restricted to $\lesssim 2^{\kt}$. We consider the cases $2^{\kt}T\leq1$ and $2^{\kt}T>1$ separately. The latter case we further split into $2^{\kth}T\leq1$ and $2^{\kth}T>1$.
\begin{enumerate}
\item $2^{\kt}T\leq1$: Here we use the modulation Bernstein estimate followed by Bernstein's inequality to bound (for $k\lesssim\kt$)
\begin{align*}
&\|P_kQ_{\lesssim\kt}(\eta_T\phi_{k_1}\dau\phi_{k_2}\dad\phi_{k_3})\|_{\xstmo}\\
&\lesssim\sum_{l\lesssim\kt}2^{(s-1)\kt}2^{(\tht-1)l}2^{l/2}\|\eta_T\|_1\|\phi_{k_1}\|_{\infty,\infty}\|\dau\phi_{k_2}\|_{\infty,2}\|\dad\phi_{k_3}\|_{\infty,\infty}\\
&\lesssim T\sum_{l\lesssim\kt}2^{(s-1)\kt}2^{\tht'l}2^{(1-s)\kt}2^{(1-s')\kth}(\dl_{\ko,0}+\|\phi_{k_1}\|_{\xst})\|\phi_{k_2}\|_{\xst}\|\phi_{k_3}\|_{\xst}\\
&\lesssim T2^{(1+\tht'-s')\kt}(\dl_{\ko,0}+\|\phi_{k_1}\|_{\xst})\|\phi_{k_2}\|_{\xst}\|\phi_{k_3}\|_{\xst}
\end{align*}
We then use that $s'$, $\tht'$ are very small and separate $2^{(1+\tht'-s')\kt}$ into $2^{(1+\half(\tht'-s'))\kt}2^{\half(\tht'-s')\kt}$. Since $2^{\kt}\leq T^{-1}$ this allows us to bound the previous line by
\begin{align*}
&\lesssim T^{\half(s'-\tht')}2^{\half(\tht'-s')\kt}(\dl_{\ko,0}+\|\phi_{k_1}\|_{\xst})\|\phi_{k_2}\|_{\xst}\|\phi_{k_3}\|_{\xst}
\end{align*}
which is acceptable since $s'>\tht'$.
\bigskip

\item $2^{\kt}T>1$, $2^{\kth}T\leq1$: In this case we start by using \eqref{MB} to lower the time exponent from $2$ to $1+$, then place all three factors of $\phi$ into Strichartz spaces:
\begin{align*}
&\|P_kQ_{\lesssim\kt}(\eta_T\phi_{k_1}\dau\phi_{k_2}\dad\phi_{k_3})\|_{\xstmo}\\
&\lesssim 2^{(s-1)\kt} \sum_{l\lesssim\kt}2^{(\tht-1)l}2^{(\half-\f{1}{M})l}\|\eta_T\|_{\f{2M}{M-1}}\|\phi_{k_1}\|_{\infty,\infty}\|\dau\phi_{k_2}\|_{\infty,2}\|\dad\phi_{k_3}\|_{\f{2M}{M-1},\infty}\\
&\lesssim T^{\half-\f{1}{2M}}2^{(\half+\f{1}{2M}-s')\kth}(\dl_{\ko,0}+\|\phi_{k_1}\|_{\xst})\|\phi_{k_2}\|_{\xst}\|\phi_{k_3}\|_{\xst}
\end{align*}
where we chose $1/M>\tht'$ and summed over $l\geq0$. 
In order to gain some decay in $k$ we need to split into two further sub-cases. Henceforth assume $s'+\tht'>2/M>2\tht'$.
\begin{enumerate}
\item $\kt\simeq k$: In this case, we simply bound $2^{\kth}\leq T^{-1}$ to find
\begin{align*}
\|P_kQ_{\lesssim\kt}(\eta_T\phi_{k_1}\dau\phi_{k_2}\dad\phi_{k_3})\|_{\xstmo}&\lesssim T^{s'-\f{1}{M}}(\dl_{\ko,0}+\|\phi_{k_1}\|_{\xst})\|\phi_{k_2}\|_{\xst}\|\phi_{k_3}\|_{\xst}
\end{align*}
which is acceptable.
\item $\kt\gg k$. Since $\ko\leq\kt-10$, we must in this case have $\kth\simeq\kt$. We find
\begin{align*}
&\|P_kQ_{\lesssim\kt}(\eta_T\phi_{k_1}\dau\phi_{k_2}\dad\phi_{k_3})\|_{\xstmo}\\
&\lesssim T^{\half-\f{1}{2M}}2^{(\half+\f{1}{2M}-\half(s'+\tht'))\kth}2^{-\half(s'-\tht')\kth}(\dl_{\ko,0}+\|\phi_{k_1}\|_{\xst})\|\phi_{k_2}\|_{\xst}\|\phi_{k_3}\|_{\xst}\\
&\lesssim T^{\f{1}{2}(s'+\tht')-\f{1}{M}}2^{-\half(s'-\tht')\kt}(\dl_{\ko,0}+\|\phi_{k_1}\|_{\xst})\|\phi_{k_2}\|_{\xst}\|\phi_{k_3}\|_{\xst}
\end{align*}
where we used that $2^{-\half(s'-\tht')\kth}\simeq^{-\half(s'-\tht')\kt}$ for the final inequality.
\end{enumerate}
\end{enumerate}
\end{itemize}
The remaining case $\ko<\max\{\kt,\kth\}-10$, $2^{\kt}T,\,2^{\kth}T>$, corresponds to a triple in $\Ss_{\ast}$ so the proof for $WM$ is complete.
\end{proof}

To handle the remaining $\lhh$ interactions we must incorporate the structures in the different terms. For the wave maps source terms we will use the following lemma, proved in Appendix \ref{appendixB}.
\begin{lemma}\label{big_lem}
Set $s=3/2+s'$, $\tht=1/2+\tht'$ for $\nu>s'>\tht'>0$. Let $\kt,\,\kth\geq0$. It holds
\begin{equation}\label{big_lem1}
\|\vp_{\kt}\cdot F_{\kth}\|_{X^{s-1,\tht-1}}\lesssim 2^{-s'\min\{\kt,\kth\}}\|\vp_{\kt}\|_{\xst}\|\fkth\|_{\xstmo}
\end{equation}
and
\begin{equation}\label{big_lem2}
\|\vp^{(2)}_{\kt}\cdot \vpthk\|_{X^{s,\tht}}\lesssim 2^{-s'\min\{\kt,\kth\}}\|\vp^{(2)}_{\kt}\|_{\xst}\|\vp^{(3)}_{\kth}\|_{\xst}
\end{equation}
\end{lemma}

The remaining interactions are then handled in the following proposition.
\begin{proposition}
Let $(\ko,\kt,\kth)\in\Ss_{\ast}$. Then for any $\phi,\,\tp\in \xst_1$ it holds
$$
\|\sum_{(\ko,\kt,\kth)\in\Ss_{\ast}}\eta_T\cdot \Tt_{\ko,\kt,\kth}(\phi)\|_{\xstmo_1}\lesssim C(\|\tp\|_{\xst_1}) T^{\eps}\|\phi\|_{\xst_1}^3
$$
for $\Tt\in\{WM,HWM_1,HWM_2\}$ and $C(\|\tp\|_{\xst_1})$ a constant as in the previous proposition.
\end{proposition}
\begin{proof}
We start with $WM$, again taking $\kt\geq\kth$ without loss of generality, so that $\ko<\kt-10$ and $2^{\kt}T,\,2^{\kth}T>1$. Use the null structure to write
$$
\dau\phi_{k_2}\cdot\dad\phi_{k_3}=\half[\Box(\phi_{k_2}\cdot\phi_{k_3})-\phi_{k_2}\cdot\Box\phi_{k_3}-\Box\phi_{k_2}\cdot\phi_{k_3}]
$$
First consider $\|\eta_T(\phi_{k_1}\Box(\phi_{k_2}\cdot\phi_{k_3}))\|_{\xstmo_1}$. Note that we may neglect the cut-off $\eta_T$ thanks to Lemma \ref{stan_lem}. By point 2 of Lemma \ref{stan_lem} followed by the definition of the $\xst$ space and Lemma \ref{big_lem} we have
\begin{align*}
\|P_k(\phi_{k_1}\Box(\phi_{k_2}\cdot\phi_{k_3}))\|_{\xstmo}&\lesssim\|\phi_{k_1}\|_{\xst}\|\Box(\phi_{k_2}\cdot\phi_{k_3})\|_{\xstmo}\\
&\lesssim\|\phi_{k_1}\|_{\xst}\|\phi_{k_2}\cdot\phi_{k_3}\|_{\xst}\\
&\lesssim 2^{-s'\kth}\|\phi_{k_1}\|_{\xst}\|\phi_{k_2}\|_{\xst}\|\phi_{k_3}\|_{\xst}
\end{align*}
If $\kt\simeq k$ we bound this by $T^{s'}\|\phi_{k_1}\|_{\xst}\|\phi_{k_2}\|_{\xst}\|\phi_{k_3}\|_{\xst}$ which is acceptable, and if $\kt\gg k$, we know (since $\ko\ll \kt$) that $\kth\gtrsim \kt$ so we can bound this by
$$
T^{s'/2}2^{-s'(\kth-k)/2}\|\phi_{k_1}\|_{\xst}\|\phi_{k_2}\|_{\xst}\|\phi_{k_3}\|_{\xst}
$$
which is also fine.

Similarly for $\|\phi_{k_1}\phi_{k_2}\cdot\Box\phi_{k_3}\|_{\xstmo}$ we use Lemma \ref{big_lem} to bound
\begin{align*}
\|P_k(\phi_{k_1}\phi_{k_2}\cdot\Box\phi_{k_3})\|_{\xstmo}
&\lesssim 2^{-s'\kth} \|\phi_{k_1}\|_{\xst}\|\phi_{k_2}\|_{\xst}\|\phi_{k_3}\|_{\xst}
\end{align*}
which is acceptable for the same reasons. The remaining term is similar.


\bigskip

For $HWM_1$ and $HWM_2$ we don't actually need to use that $2^{\kt}T$, $2^{\kth}T>1$ and we can get the gain we need from H\"older's inequality.
Let's start with $HWM_1$ in the case $\kth\geq\kt$, so $\ko<\kth-10$. Note that the high modulation case $||\tau|-|\xi||\gg\kth$ was handled in the previous proof, so we only have to consider low modulations. Since $\kth$ is the highest frequency the output is restricted to $2^k\lesssim2^{\kth}$, so for fixed $\kt$, $\kth$ we have
\begin{align*}
    &\|\sum_{\ko<\kth-10}Q_{\lesssim\kth}P_k(\eta_T\cdot HWM_{1;\ko,\kt,\kth}(\phi))\|_{\xstmo}\\
    &\lesssim\sum_{l\lesssim\kth}2^{(\tht-1)l}2^{(s-1)\kth}\|Q_lP_k(\eta_T P_{\kt}(\widetilde{\Pi}_{\tilde{\phi}^{\perp}}\du)(\phi_{<\kth-10}\cdot\dphi_{\kth}))\|_{2,2}\\
    &\lesssim \sum_{l\lesssim\kth}2^{(\tht-1)l}2^{(s-1)\kth}\|Q_lP_k(\eta_T P_{\kt}(\widetilde{\Pi}_{\tilde{\phi}^{\perp}}\du)(\phi_{<\kth-10}\cdot\dphi_{\kth}\\
    &\hspace{23em}-\dhalf P_{\kth}(\phi_{<\kth-10}\cdot\phi_{\geq\kth-10}))\|_{2,2}\\
    &\quad+\sum_{l\lesssim\kth}2^{(\tht-1)l}2^{(s-1)\kth}\|Q_lP_k(\eta_T P_{\kt}(\widetilde{\Pi}_{\tilde{\phi}^{\perp}}\du)\dhalf P_{\kth}(\phi_{<\kth-10}\cdot\phi_{\geq\kth-10}))\|_{2,2}
    \end{align*}
The first term above sees a derivative moved onto the low frequency factor $\phi_{<\kth-10}$ (see Lemmas \ref{useful_lemma} and \ref{moving_loc_lemma}), so is easier to handle. For the third line we use the geometric identity \eqref{geometric_identity_4} to swap the low frequency factor for a high one and find
\begin{align*}
    &\sum_{l\lesssim\kth}2^{(\tht-1)l}2^{(s-1)\kth}\|Q_lP_k(\eta_T P_{\kt}(\widetilde{\Pi}_{\tilde{\phi}^{\perp}}\du)\dhalf P_{\kth}(\phi_{<\kth-10}\cdot\phi_{\geq\kth-10}))\|_{2,2}\\
    &\lesssim\sum_{l\lesssim\kth}2^{(\tht-1)l}2^{(s-1)\kth}2^{l(\half-\f{1}{M})}\|\eta_T\|_M\|P_{\kt}(\widetilde{\Pi}_{\tilde{\phi}^{\perp}}\du)\|_{\f{2M}{M-4},\infty}\cdot 2^{\kth}\|\phi_{\geq\kth-10}\|_{2M,\f{2M}{M-1}}\|\phi_{\geq\kth-10}\|_{\f{2M}{M-1},2M}\\
    &\lesssim T^{1/M}2^{(\half+\f{2}{M}-s')(\kt-\kth)}2^{2(\f{1}{M}-s')(\kth-k)}2^{2(\f{1}{M}-s')k}\|\phi\|_{\xst_1}^2\sum_{k'\geq0}2^{-\sg|\kt-k'|}\|\phi_{k'}\|_{\xst}
\end{align*}
where we chose $M$ such that $\tht'<M^{-1}<s'$. This can be summed over $\kt\leq\kth$, $\kth\gtrsim k$ and $k\geq0$ as required. 

The case $\kt>\kth$ is similar, with the exception that we must separately study $\ko<\kth-10$ and $\ko\in[\kth-10,\kt-10]$ in order to apply \eqref{geometric_identity_4}.

\bigskip

Finally, we turn to $HWM_2$, again restricting to modulation $\lesssim 2^{\max\{\kt,\kth\}}$. We first consider $\kth\geq\kt+10$, in which case we must have output frequency $k\sim\kth$ and can write $HWM_{2;\ko,\kt,\kth}(\phi)=\pko\x\Lc_{\kt+\kth}(\pkt,\pkth)$ for $\Lc$ as in \eqref{Lc}. We therefore have
\begin{align*}
    &\|\sum_{\ko<\kth-10} Q_{\lesssim\kth} P_k(\eta_T\cdot HWM_{2;\ko,\kt,\kth}(\phi))\|_{\xstmo}\\
    &\lesssim \sum_{l\lesssim\kth}2^{(\tht-1)l}2^{(s-1)\kth}2^{(\half-\f{1}{M})l}\|\eta_T\,\phi_{<\kth-10}\x\Lc_{\kt+\kth}(\pkt,\pkth)\|_{\f{M}{M-1},2}\\
    &\lesssim T^{1/M}\sum_{l\lesssim\kth}\sum_{a,b}2^{(\tht'-\f{1}{M})l}2^{(s-1)\kth}c^{(\kt+\kth)}_{a,b}(\|\pkt(x+2^{-\kt}a)\ \phi_{<\kth-10}(x)\cdot\pkth(x+2^{-\kth}b)\|_{\f{M}{M-2},2}\\
    &\hspace{18em}+\|\pkth(x+2^{-\kth}b)\ \phi_{<\kth-10}(x)\cdot\pkt(x+2^{-\kt}a)\|_{\f{M}{M-2},2})
\end{align*}
where we used \eqref{FS} and \eqref{vector_id} in the second inequality. We then write
$$
\phi_{<\kth-10}(x)=\phi_{<\kth-10}(x+2^{-\kth}b)-2^{-\kth}b\int_0^1\na\phi_{<\kth-10}(x+2^{-\kth}b\tht)d\tht
$$
and use \eqref{geometric_identity_4} to bound
\begin{multline*}
    \|\pkt(x+2^{-\kt}a)\ \phi_{<\kth-10}(x)\cdot\pkth(x+2^{-\kth}b)\|_{\f{M}{M-2},2}\\
    \lesssim\langle b\rangle2^{-(\half-\f{2}{M}+s')\kt}2^{-(2+2s')\kth}\|\phi\|_{\xst_1}^2\left(\sum_{k'\gtrsim\kth}2^{-(\half+\f{1}{M}+s')(k'-\kth)}\|\phi_{k'}\|_{\xst}\right)
\end{multline*}
and similarly
\begin{align*}
    \|\pkth(x+2^{-\kth}b)\ \phi_{<\kth-10}(x)\cdot\pkt(x+2^{-\kt}a)\|_{\f{M}{M-2},2}
    \lesssim\langle a\rangle2^{-s\kth}2^{-(1-\f{2}{M}+2s')\kt}\|\pkth\|_{\xst_1}\|\phi\|_{\xst_1}^2
\end{align*}
It follows that, choosing $\tht'<M^{-1}<s'$,
\begin{align*}
    &\|\sum_{\ko<\kth-10} Q_{\lesssim\kth} P_k(\eta_T\cdot HWM_{2;\ko,\kt,\kth}(\phi))\|_{\xstmo}\\
    &\lesssim T^{1/M}\sum_{l\lesssim\kth}\sum_{a,b}2^{(\tht'-\f{1}{M})l}2^{(s-1)\kth}c^{\kt+\kth}_{a,b}\langle a \rangle\langle b \rangle 2^{-s\kth}2^{-(1-\f{2}{M}+2s')\kt}\|\phi\|_{\xst_1}^2\left(\sum_{k'\gtrsim\kth}2^{-\sg|k'-\kth|}\|\phi_{k'}\|_{\xst}\right)\\
    &\lesssim T^{1/M} 2^{2(\f{1}{M}-s')\kt}\|\phi\|_{\xst_1}^2\left(\sum_{k'\gtrsim\kth}2^{-\sg|k'-\kth|}\|\phi_{k'}\|_{\xst}\right)
\end{align*}
which is acceptable when summed over $\kt\geq0$, $\kth\sim k$, $k\geq0$. The case $\kt\geq\kth+10$ can be treated identically.

In the remaining case $\kt\simeq\kth$, we again call upon the identity \eqref{geometric_identity_4}, however this time there is nothing to be gained by cancellation and so one must instead split $HWM_2$ into its two components and treat each separately. The term involving
$$
\pkt\x(-\D)\pkth
$$
is easier to handle as there are no nonlocal operators acting so one can directly apply the vector product identity. For the term involving
$$
\dhalf(\pkt\x\dhalf\pkth),
$$
we only need to use \eqref{geometric_identity_4} when the frequency of this output output is comparable to $2^k$. The details are left to the reader.

\end{proof}

In combination with Remark \ref{red-remark}, the previous two propositions complete the proof of Proposition \ref{all_ML_estimates}.

\subsection{Local Wellposedness of the Half-Wave Maps Equation \eqref{eqn1.1}.}\label{ss2}
It remains to show that the local solution to the differentiated equation in fact solves the original problem (\ref{eqn1.1}) under the compatibility assumption $\phi_1=\phi_0\x\dhalf\phi_0$. We use an energy argument as in \cite{KS}.

Let $\phi$ be a smooth local solution to equation \eqref{eqn1.2} with data $(\phi_0,\phi_1)$ as above. Set
$$
X:=\phi_t-\phi\times\dhalf \phi
$$
Our goal is to show $X\equiv 0$. To this end, consider the energy type functional
$$
\tilde{E}(t)=\half\int_{\R^3}|X(t,x)|^2 dx
$$
A calculation as in \cite{KS} shows
\begin{equation*}
\dd_t X=-\phi(X\cdot(\phi\times\dhalf \phi+\phi_t))-X\times\dhalf \phi-\phi\times\dhalf X
\end{equation*}
from which
\begin{align}
\f{d}{dt}\tilde{E}(t)=-\int_{\R^3}(\phi(X\cdot(\phi\times\dhalf \phi+\phi_t)))\cdot X dx-\int_{\R^3}(\phi\times\dhalf X)\cdot Xdx\label{dtet2}
\end{align}
We immediately see that
\begin{align*}
\left|\int_{\R^3}(\phi(X\cdot(\phi\times\dhalf \phi+\phi_t)))\cdot X \,dx\right|\lesssim&\|\phi\|_{\infty}\|X\|_2^2\|\phi\times\dhalf \phi+\phi_t\|_{\infty}\lesssim_\phi\|X\|_2^2
\end{align*}
since $\phi,\,\na_{t,x} \phi\in \xst\hookrightarrow L^\infty_{t,x}$. For the second term, we subtract a term which is zero (by Plancherel):
\begin{align*}
\int_{\R^3}(\phi\ti(-\D)^{\half}X)\cdot X dx=&\int_{\R^3}[(\phi\times\dhalf X)-(-\Delta)^{\f{1}{4}}(\phi\times(-\Delta)^{\f{1}{4}}X)]\cdot X dx
\end{align*}
then bound
\begin{align*}
\|(\phi\times\dhalf X)-(-\D)^\f{1}{4}(\phi\times(-\D)^\f{1}{4}X)\|_2\lesssim\sum_{\ko\geq0}\|\Lc(\phi_{\ko},X_{<\ko+10})\|_2+\left(\sum_{\kt\geq0}\|\Lc(\phi_{<\kt-10},X_{\kt})\|_2^2\right)^\half
\end{align*}
with
$$
\Lc(\pko,X_{\kt})=\int_{\xi,\eta}e^{ix\cdot(\xi+\eta)}|\eta|^{\half}(|\eta|^{\half}-|\xi+\eta|^\half)\chi_{\ko}(\xi)\hat{\phi}(\xi)\chi_{\kt}(\eta)\hat{X}(\eta)d\xi d\eta
$$
It is then straightforward that 
\begin{align*}
   \sum_{\ko\geq0}\|\Lc(\phi_{\ko},X_{<\ko+10})\|_2\lesssim\sum_{\ko\geq0}2^{\ko}\|\pko\|_{\infty}\|X\|_2
   \lesssim_\phi\|X\|_2
\end{align*}
and applying Lemma \ref{useful_lemma} (using $\left||\eta|^{\half}(|\eta|^{\half}-|\xi+\eta|^\half|)\right|\lesssim|\xi|$), we also have
\begin{align*}
\left(\sum_{\kt\geq0}\|\Lc(\phi_{<\kt-10},X_{\kt})\|_2^2\right)^\half\lesssim\left(\sum_{\kt\geq0}\|\na \phi\|_{\infty}^2\|X_{\kt}\|_2^2\right)^\half
\lesssim_\phi \|X\|_2
\end{align*}
We have therefore shown that
\begin{equation*}
\f{d}{dt}\tilde{E}(t)\lesssim_\phi\tilde{E}(t)    
\end{equation*}
and since the initial conditions imply that $\tilde{E}(0)=0$, we conclude that $\tilde{E}\equiv0$ for all time. This completes the proof.

\appendix
\section{Control of the Low Frequencies}\label{appendixN}
In this short appendix we show that the low frequency portion of $\phi$ cannot blow up. It is recommended that the reader ignores this appendix until the end of the proof, since some of the methods will by then be familiar.

By the energy estimate for the wave equation we find
\begin{align*}
    \|P_{\leq0}\dd_t\phi\|_{L^\infty_tL^2_x([0,T])}&\lesssim\|P_{\leq0}\phi[0]\|_{\dot{H}^1\x L^2}+\|P_{\leq0}\Box\phi\|_{L^1_tL^2_x([0,T])}
\end{align*}
By our assumptions on the initial data certainly $\|P_{\leq0}\phi[0]\|_{\dot{H}^1\x L^2}<\infty$. For the nonhomogeneous term we use H\"older's inequality in time and Bernstein in space to find, for instance,
\begin{align*}
    \|P_{\leq0}(\phi\dad\phi\dau\phi)\|_{L^1_tL^2_x([0,T])}&\lesssim T(\|P_{\leq0}(\phi\dad\phi_{>10}\dau\phi_{>10})\|_{L^\infty_tL^1_x([0,T])}\\
    &\hspace{2em}+\|P_{\leq0}(\phi\dad\phi_{\leq10}\dau\phi_{>10})\|_{L^\infty_tL^2_x([0,T])}\\
    &\hspace{2em}+\|P_{\leq0}(\phi\dad\phi_{\leq10}\dau\phi_{\leq10})\|_{L^\infty_tL^2_x([0,T])})\\
    &\lesssim T(\|\phi\|_{\infty,\infty}\|\dad\phi_{>10}\|_{\infty,2}\|\dau\phi_{>10}\|_{\infty,2}\\
    &\hspace{2em}+\|\phi\|_{\infty,\infty}\|\dad\phi_{\leq10}\|_{\infty,\infty}\|\dau\phi_{>10}\|_{\infty,2}\\
    &\hspace{2em}+\|\phi\|_{\infty,\infty}\|\dad\phi_{\leq10}\|_{\infty,4}\|\dau\phi_{\leq10}\|_{\infty,4})
\end{align*}
All of these terms are bounded by $T\eps^2$ using the definition of $S$ and the local constancy of the frequency envelope. The half-wave maps source terms can be treated similarly using arguments as in Section \ref{normal_forms_chapter} (see for example Claim \ref{claim1}, Proposition \ref{D1_bounded_prop}).

This shows that the low frequency portion of $\dd_t\phi$ remains bounded for all time (even if this bound is growing in $T$). For the $L^2$ norm of the solution itself we can then use that the data is certainly in $L^2$ (upon subtracting the constant $p$) and calculate the derivative
\begin{align*}
    \f{d}{dt}\|P_{\leq0}\phi(t)\|_{L^2_x}^2&=2\int_{\R^3}P_{\leq0}\phi\cdot P_{\leq0}\dd_t\phi dx\\
    &\leq \eps^2\|P_{\leq0}\phi\|_{L^\infty_tL^2_x([0,T])}^2+\eps^{-2}\|P_{\leq0}\dd_t\phi\|_{L^\infty_tL^2_x([0,T])}^2
\end{align*}
Choosing $\eps=(2T)^{-1/2}$ and using the fundamental theorem of calculus this yields
\begin{align*}
    \|P_{\leq0}\phi\|_{L^\infty_tL^2_x([0,T])}^2&\leq2\|\phi_0\|_{L^2_x}^2+4T^2\|P_{\leq0}\dd_t\phi\|_{L^\infty_tL^2_x([0,T])}^2
\end{align*}
In combination with the bound already shown for $\|P_{\leq0}\dd_t\phi\|_{L^\infty_tL^2_x([0,T])}$ this shows that $\|P_{\leq0}\phi\|_{L^\infty_tL^2_x([0,T])}$ also remains bounded on the interval $[0,T]$.

We remark that the control on $\|P_{\leq0}\phi\|_{L^2_x}$ could also be obtained from the conserved mass
\begin{equation}\label{mass_con}
    M(t):=\int_{\R^n}|\phi-p|^2dx
\end{equation}
of the half-wave maps equation, although the approach above is of course more general.

\section{Proof of Lemma \ref{xsb_projection_estimate}}\label{appendixA}
In this appendix we give a brief proof of Lemma \ref{xsb_projection_estimate}, focusing only on the first point \eqref{xst_moser}. The remaining estimates \eqref{xst_proj_estimate}-\eqref{contp} follow upon writing
$$
\wip \Phi=\Phi-\Phi^TG(\tp)
$$
for $G(\tp)=g(\tp)\, g(\tp)$.
\begin{proof}[Proof of \eqref{xst_moser}]
    We study the different regimens of $(j,k)$ separately. In this proof all implicit constants may depend polynomially on $\tpxst$.
    \begin{itemize}[leftmargin=*]
    \item\underline{$k\ll j$:} We have to show that
    $$
    \|P_k Q_j g(\tp)\|_{2,2}\lesssim_{\tpxst} 2^{-(s+\tht)j}
    $$
    Note that since $j\gg k$ we have $P_kQ_j=P_kQ_jP^{(t)}_{\sim j}$ and $j>0$ so
    \begin{align*}
        \|P_k Q_j g(\tp)\|_{2,2}\lesssim \underbrace{2^{-j}\|P_kQ_j[Q_{\gtrsim j}\dd_t\tp\cdot g'(\tp)]\|_{2,2}}_{(A)}+\underbrace{2^{-j}\|P_kQ_j[Q_{\ll j}\dd_t\tp\cdot g'(\tp)]\|_{2,2}}_{(B)}
    \end{align*}
    Here
    \begin{align*}
        (A)\lesssim 2^{-j}\|Q_{\gtrsim j}\dd_t\tp\|_{2,2}\lesssim 2^{-j}2^{-(s+\tht-1)j}\|\dd_t\tp\|_{X^{0,s-1+\tht}}\lesssim 2^{-(s+\tht)j}\|\tp\|_{\xst}
    \end{align*}
For (B) we differentiate in $t$ a second time and find
    \begin{align*}
        (B)\lesssim \underbrace{2^{-2j}\|P_kQ_j[Q_{\ll j}\dd_t^2\tp\cdot g'(\tp)]\|_{2,2}}_{(B1)}+\underbrace{2^{-2j}\|P_kQ_j[Q_{\ll j}\dd_t\tp\cdot\dd_t\tp\cdot g''(\tp)]\|_{2,2}}_{(B2)}
    \end{align*}
We start with 
    $$
    (B1)\lesssim \underbrace{2^{-2j}\|P_kQ_j[Q_{< k-10}\dd_t^2\tp\cdot g'(\tp)]\|_{2,2}}_{(B1)_{\ll k}}+\underbrace{2^{-2j}\|P_kQ_j[Q_{[k-10,j-10]}\dd_t^2\tp\cdot g'(\tp)]\|_{2,2}}_{(B1)_{\gtrsim k}}
    $$
    For the lowest modulation case we have
    \begin{align}
        (B1)_{\ll k}&\lesssim \sum_{l<k-10}(2^{-2j}\|P_kQ_j[P_{<l-10}Q_l\dd_t^2\tp\cdot g'(\tp)]\|_{2,2}\tag*{$(B1a)_{\ll k}$}\\
        &\hspace{5em}+2^{-2j}\|P_kQ_j[P_{\geq l-10}Q_l\dd_t^2\tp\cdot g'(\tp)]\|_{2,2})\tag*{$(B1b)_{\ll k}$}
    \end{align}
    Now the real calculations begin. For the first of these terms we use that $l$ is far smaller than the scales $k$ or $j$ so the factor of $g'$ must also be localised to $P_{\sim k}Q_{\sim j}$. It follows that
    \begin{align*}
        (B1a)_{\ll k}&\lesssim2^{-2j}\sum_{l<\kmt}\|P_kQ_j[P_{<l-10}Q_l\dd_t^2\tp\cdot\dd_t^{-1}P_{\sim k}Q_{\sim j}[\dd_t\tp\cdot g''(\tp)]]\|_{2,2}\\
        &\lesssim2^{-2j}\sum_{l<\kmt}2^{-j}\|P_{<l-10}Q_l\dd_t^2\tp\|_{\infty,\infty}\|Q_{\gtrsim l}\dd_t\tp\|_{2,2}\|g''(\tp)\|_{\infty,\infty}\\
        &\quad+2^{-2j}\sum_{l<\kmt}\|P_kQ_j[P_{<l-10}Q_l\dd_t^2\tp\cdot\dd_t^{-1}P_{\sim k}Q_{\sim j}[Q_{\ll l}\dd_t\tp\cdot g''(\tp)]]\|_{2,2}
    \end{align*}
    The first line above can be bounded by 
    \begin{equation*}
        2^{-3j}\sum_{l<k-10}2^{[2-(s'+\tht')]l}\|\tp\|_{\xst}2^{-(s-1+\tht)l}\|Q_{\gtrsim l}\dd_t\tp\|_{X^{0,s-1+\tht}}
        \lesssim 2^{[3-(s+\tht)-(s'+\tht')](k-j)}2^{-(s+\tht)j}
    \end{equation*}
    as required. For the second line we further split $Q_{\ll l}\dd_t\tp$ into low and high frequencies to find
    \begin{align*}
        &2^{-2j}\sum_{l<\kmt}\|P_kQ_j[P_{<l-10}Q_l\dd_t^2\tp\cdot\dd_t^{-1}P_{\sim k}Q_{\sim j}[Q_{\ll l}\dd_t\tp\cdot g''(\tp)]]\|_{2,2}\\
        &\lesssim2^{-3j}\sum_{l<\kmt}\|P_{<l-10}Q_l\dd_t^2\tp\|_{2,2}\|P_{\lesssim l}Q_{\ll l}\dd_t\tp\|_{\infty,\infty}
        +2^{-3j}\sum_{l<\kmt}\|P_{<l-10}Q_l\dd_t^2\tp\|_{2,\infty}\|P_{\gg l}Q_{\ll l}\dd_t\tp\|_{\infty,2}\\
        &\lesssim2^{-3j}\sum_{l<\kmt}\left(2^{[2-(s+\tht)l]}2^{(1-s')l}+2^{3l/2}2^{[2-(s+\tht)]l}2^{(1-s)l}\right)\\
        &\lesssim2^{[3-(s+\tht)](k-j)}2^{-(s+\tht)j}
    \end{align*}
    This completes the study of $(B1a)_{\ll k}$. We now turn to $(B1b)_{\ll k}$. Write
    \begin{align}
        (B1b)_{\ll k}&\leq2^{-2j}\sum_{l<\kmt}\|P_kQ_j[P_{> k+10}Q_l\dd_t^2\tp\cdot g'(\tp)]\|_{2,2}\label{po0}\\
        &+2^{-2j}\sum_{l<\kmt}\|P_kQ_j[P_{[l-10, k+10]}Q_l\dd_t^2\tp\cdot g'(\tp)]\|_{2,2}\label{po1}
    \end{align}
    The easier of these terms is \eqref{po0}, which we write as follows:
    \begin{align}
        \eqref{po0}&\lesssim 2^{-2j}\sum_{l\ll k}\sum_{r\gg k}\|P_kQ_j[P_rQ_l\dd_t^2\tp\cdot P_{\sim r}g'(\tp)]\|_{2,2}\nonumber\\
        &\lesssim 2^{-2j}\sum_{\substack{l\ll k\\r\gg k}}\|P_kQ_j[P_rQ_l\dd_t^2\tp\cdot \na^{-1}P_{\sim r}[\na\tp_{\ll r}\cdot g''(\tp)]]\|_{2,2}\label{THIS-ONE}\\
        &\quad+2^{-2j}\sum_{\substack{l\ll k\\r\gg k}}\|P_kQ_j[P_rQ_l\dd_t^2\tp\cdot \na^{-1}P_{\sim r}[\na\tp_{\gtrsim r}\cdot g''(\tp)]]\|_{2,2}\label{OTHER-ONE}
    \end{align}
To study \eqref{THIS-ONE}, we differentiate in $t$ a further time and obtain
    \begin{align}
        \eqref{THIS-ONE}&\lesssim 2^{-3j}\sum_{\substack{l\ll k\\r\gg k}}\|P_kQ_j[P_rQ_l\dd_t^3\tp\cdot \na^{-1}P_{\sim r}[\na\tp_{\ll r}\cdot P_{\sim r}g''(\tp)]]\|_{2,2}\label{1i}\\
        &\quad+2^{-3j}\sum_{\substack{l\ll k\\r\gg k}}\|P_kQ_j[P_rQ_l\dd_t^2\tp\cdot \na^{-1}P_{\sim r}[\na\dd_t\tp_{\ll r}\cdot P_{\sim r}g''(\tp)]]\|_{2,2}\label{2ii}\\
        &\quad+2^{-3j}\sum_{\substack{l\ll k\\r\gg k}}\|P_kQ_j[P_rQ_l\dd_t^2\tp\cdot \na^{-1}P_{\sim r}[\na\tp_{\ll r}\cdot \dd_tP_{\sim r}g''(\tp)]]\|_{2,2}\label{3iii}
    \end{align}
    For \eqref{1i} we use Bernstein at frequency $2^k$ to see
    \begin{align*}
        \eqref{1i}&\lesssim 2^{-3j}\sum_{\substack{l\ll k\\r\gg k}}2^{3k/2}\|P_rQ_l\dd_t^3\tp\|_{\infty,2}\cdot 2^{-r}\|\na\tp_{\ll r}\|_{\f{2M}{M-1},2M}\|P_{\sim r}g''(\tp)\|_{2M,\f{2M}{M-1}}\\
        &\lesssim2^{-3j}\sum_{\substack{l\ll k\\r\gg k}}2^{3k/2}2^{-\tht'l}2^{(3-s)r}2^{-r}2^{(\half-\f{1}{M}-s')r}2^{-(\f{3}{2}-\f{1}{M}+s')r}\\
        &\lesssim2^{(1-3s')(k-j)}2^{-(2+3s')j}
    \end{align*}
    which is acceptable. The second term \eqref{2ii} can be treated in the same way. For \eqref{3iii} we use Lemma \ref{sc_projection_lemma} to bound
    \begin{align*}
        \eqref{3iii}&\lesssim 2^{-3j}\sum_{\substack{l\ll k\\r\gg k}}\|P_kQ_j[P_rQ_l\dd_t^2\tp\cdot \na^{-1}P_{\sim r}[\na\tp_{\ll r}\cdot \dd_t P_{\sim r}g''(\tp)]]\|_{2,2}\\
        &\lesssim 
        2^{-3j}\sum_{\substack{l\ll k\\r\gg k}}2^{3k/2}\|P_rQ_l\dd_t^2\tp\|_{\infty,2}\cdot 2^{-r}\|\na\tp_{\ll r}\|_{\f{2M}{M-1},2M}\|\dd_t P_{\sim r}g''(\tp)\|_{2M,\f{2M}{M-1}}\\
        &\lesssim 
        2^{-3j}\sum_{\substack{l\ll k\\r\gg k}}2^{3k/2}2^{-\tht'l}2^{(2-s)r}2^{-r}2^{(\half-\f{1}{M}-s')r}2^{-(\half-\f{1}{M}+s')r}\\
        &\lesssim 2^{-3j}2^{(1-3s')k}
    \end{align*}
    which is as required. We now turn to \eqref{OTHER-ONE}.    
    If we restrict the sum to $r\gtrsim j$ the term is easily handled:
    \begin{align*}
        &2^{-2j}\sum_{\substack{l\ll k\\r\gtrsim j}}\|P_kQ_j[P_rQ_l\dd_t^2\tp\cdot \na^{-1}P_{\sim r}[\na\tp_{\gtrsim r}\cdot g''(\tp)]]\|_{2,2}\\
        &\lesssim2^{-2j}\sum_{\substack{l\ll k\\r\gtrsim j}}\|P_rQ_l\dd_t^2\tp\|_{2M,\f{2M}{M-1}} 2^{-r}\|P_{\sim r}(\na\tp_{\gtrsim r}\cdot g''(\tp))\|_{\f{2M}{M-1},2M}\\
        &\lesssim2^{-2j}\sum_{\substack{l\ll k\\r\gtrsim j}}2^{(\half-\f{1}{2M}-\tht)l}2^{3r/2M}2^{(2-s)r}\cdot 2^{-r}2^{(\half-\f{1}{M}-s')r}\\
        &\lesssim2^{-(2+2s'-\f{1}{2M})j}
    \end{align*}
    Here we used the bound
    \begin{align*}
    \|P_{\sim r}(\na\tp_{\gtrsim r}\cdot g''(\tp))\|_{\f{2M}{M-1},2M}
    &\lesssim\|\na\tp_{\sim r}\|_{\f{2M}{M-1},2M}+\sum_{m\gg r}2^{3r/2}\|\na\tp_m\|_{\f{2M}{M-1},2M}\|P_{\sim m}g''(\tp)\|_{\infty,2}
    \end{align*}
    to go from the second to the third line. Such decompositions will be used frequently without comment in the sequel. Choosing $M$ such that $s'-\f{1}{2M}\geq\tht'$, we see that the sum over $r\gtrsim j$ is acceptable.
    
    For the sum over $r\in[k+10,j-10]$ we differentiate again and have
    \begin{align}
        &2^{-2j}\sum_{\substack{l\ll k\\r\in[k+10,j-10]}}\|P_kQ_j[P_rQ_l\dd_t^2\tp\cdot\na^{-1}P_{\sim r}[\na\tp_{\gtrsim r}\cdot g''(\tp)]]\|_{2,2}\nonumber\\
        &\lesssim 2^{-3j}\sum_{\substack{l\ll k\\r\in[k+10,j-10]}}\|P_kQ_j[P_rQ_l\dd_t^3\tp\cdot\na^{-1}P_{\sim r}[\na\tp_{\gtrsim r}\cdot g''(\tp)]]\|_{2,2}\tag{I}\label{I}\\
        &\quad+2^{-3j}\sum_{\substack{l\ll k\\r\in[k+10,j-10]}}\|P_kQ_j[P_rQ_l\dd_t^2\tp\cdot\na^{-1}P_{\sim r}[\na\dd_t\tp_{\gtrsim r}\cdot g''(\tp)]]\|_{2,2}\tag{II}\label{II}\\
        &\quad+2^{-3j}\sum_{\substack{l\ll k\\r\in[k+10,j-10]}}\|P_kQ_j[P_rQ_l\dd_t^2\tp\cdot\na^{-1}P_{\sim r}[\na\tp_{\gtrsim r}\cdot \dd_t g''(\tp)]]\|_{2,2}\tag{III}\label{III}
    \end{align}
    where
    \begin{align*}
        \eqref{I}\lesssim2^{-3j}\sum_{\substack{l\ll k\\r\in[k+10,j-10]}}\|P_rQ_l\dd_t^3\tp\|_{2M,\f{2M}{M-1}}\cdot 2^{-r}\|P_{\sim r}(\na\tp_{\gtrsim r}\cdot g''(\tp))\|_{\f{2M}{M-1},2M}
        \lesssim 2^{-(2+2s'-\f{1}{2M})j}
    \end{align*}
    (summing over $l\geq0$, $r\ll j$) which is again acceptable for $s'-\f{1}{2M}\geq\tht'$. The bound for \eqref{II} is similar. 
    For \eqref{III} the bound is straightforward upon placing $P_{\sim r}[\na\tp_{\sim r}\cdot\dd_t g''(\tp)]$ into $L^{2+}_tL^2_x$ and separately considering the cases where the frequency of $\dd_t g''(\tp)$ is comparable to or much smaller than that of $\na\tp$.
    
    This completes the work on \eqref{po0} so we now turn to \eqref{po1}:
    \begin{align}
    \eqref{po1}    &\lesssim2^{-2j}\sum_{l<\kmt}\|P_kQ_j[P_{[l-10,k-10]}Q_l\dd_t^2\tp\cdot g'(\tp)]\|_{2,2}\label{po1i}\\
    &\quad+2^{-2j}\sum_{l<\kmt}\|P_kQ_j[P_{[k-10,k+10]}Q_l\dd_t^2\tp\cdot g'(\tp)]\|_{2,2}\label{po1ii}
    \end{align}
    For the first line we use that $g'(\tp)$ must be restricted to frequency $\sim 2^k$ and modulation $\sim 2^j$, which allows us to swap a $2^j$ for a $2^k$ by Moser's inequality \eqref{sc_moser}:
    \begin{align*}
        \eqref{po1i}&\lesssim 2^{-2j}\sum_{l\ll k}\|P_kQ_j[P_{[l-10,k-10]}Q_l\dd_t^2\tp\cdot P_{\sim k}Q_{\sim j}g'(\tp)]\|_{2,2}\\
        &\lesssim 2^{-2j}\sum_{l\ll k}\|P_{[l-10,k-10]}Q_l\dd_t^2\tp\|_{2M,\f{2M}{M-1}}\cdot 2^{-j}\|\dd_t P_{\sim k}g'(\tp)\|_{\f{2M}{M-1},2M}\\
        &\lesssim2^{-3j}\sum_{l\ll k}\sum_{\la=l-10}^{k-10}2^{(\half-\f{1}{2M}-\tht)l}2^{3\la/2M}2^{(2-s)\la}2^{(\half-\f{1}{M}-s')k}\\
        &\lesssim2^{(1+\f{1}{2M}-2s')(k-j)}2^{-(2+2s'-\f{1}{2M})j}
    \end{align*}
    which is acceptable for $s'-\f{1}{2M}\geq\tht'$.

    To complete the work on $(B1b)_{\ll k}$ it remains to study \eqref{po1ii}. We use that $k\ll j$ to see that $g'(\tp)$ must be at modulation $\sim 2^j$ and so
        \begin{align*}
        \eqref{po1ii}&\lesssim 2^{-2j}\sum_{l\ll k}\|P_kQ_j[P_{\sim k}Q_l\dd_t^2\tp\cdot \dd_t^{-1}P_{\lesssim k}Q_{\sim j}\dd_tg'(\tp)]\|_{2,2}\\
        &\lesssim2^{-2j}\sum_{\substack{l\ll k\\\la\lesssim k}}\|P_{\sim k}Q_l\dd_t^2\tp\|_{2M,\f{2M}{M-1}}2^{-j}\|P_{ \la}\dd_tg'(\tp)\|_{\f{2M}{M-1},2M}\\
        &\lesssim2^{-3j}\sum_{\substack{l\ll k\\\la\lesssim k}}2^{(\half-\f{1}{2M}-\tht)l}2^{3k/2M}2^{(2-s)k}2^{(\half-\f{1}{M}-s')\la}\\
        &\lesssim 2^{(1-2s'+\f{1}{2M})(k-j)}2^{-(2+2s'-\f{1}{2M})j}
    \end{align*}
    which is acceptable for $s'-\f{1}{2M}>\tht'$. This completes the study of $(B1b)_{\ll k}$, and so of $(B1)_{\ll k}$.

    To finish the work on (B1), we therefore now have to study
    \begin{align}
        (B1)_{\gtrsim k}&\lesssim 2^{-2j}\sum_{l=\kmt}^{j-10}\|P_kQ_j[P_{<k-10}Q_l\dd_t^2\tp\cdot g'(\tp)]\|_{2,2}\tag*{$(B1a)_{\gtrsim k}$}\label{B1agk}\\
        &\quad+2^{-2j}\sum_{l=\kmt}^{j-10}\|P_kQ_j[P_{\geq k-10}Q_l\dd_t^2\tp\cdot g'(\tp)]\|_{2,2}\tag*{$(B1b)_{\gtrsim k}$}\label{B1bgk}
    \end{align}
    For \ref{B1agk} we note that $g'(\tp)$ must be at frequency $\sim 2^k$ and modulation $\sim 2^j$ and decompose
    \begin{align*}
        \text{\ref{B1agk}}&\lesssim 2^{-2j}\sum_{l=k-10}^{j-10}\|P_kQ_j[P_{<k-10}Q_l\dd_t^2\tp\cdot\dd_t^{-1}P_{\sim k}Q_{\sim j}\dd_tP_{\lesssim k}g'(\tp)]\|_{2,2}\\
        &\lesssim2^{-2j}\sum_{l\gtrsim k}\|P_{\ll k}Q_l\dd_t^2\tp\|_{2,2}\cdot 2^{-j}\|P_{\sim k}\dd_tg'(\tp)\|_{\infty,\infty}\\
        &\lesssim 2^{(1-s'-\tht')(k-j)}2^{-(2+s'+\tht')j}
    \end{align*}
    Then for \ref{B1bgk} we note that if $l\gg k$ we can write
    \begin{align}
        2^{-2j}\sum_{l=k+10}^{j-10}\|P_kQ_j[P_{\geq k-10}Q_l\dd_t^2\tp\cdot g'(\tp)]\|_{2,2}
        &\lesssim2^{-2j}\sum_{l=k+10}^{j-10}\|P_kQ_j[P_{> k+10}Q_l\dd_t^2\tp\cdot g'(\tp)]\|_{2,2}\label{pol0}\\
        &\quad+2^{-2j}\sum_{l=k+10}^{j-10}\|P_kQ_j[P_{\sim k}Q_l\dd_t^2\tp\cdot g'(\tp)]\|_{2,2}\label{pol1}
    \end{align}
    where \eqref{pol0} can be treated in exactly the same way as \eqref{po0}, and for \eqref{pol1} we observe that $g'(\tp)$ is restricted to $P_{\lesssim k}Q_{\sim j}g'(\tp)=P^{(t)}_{\sim j}P_{\lesssim k}Q_{\sim j}g'(\tp)$ and so
    \begin{align*}
        \eqref{pol1}&\lesssim2^{-2j}\sum_{l=k+10}^{j-10}\|P_kQ_j[P_{\sim k}Q_l\dd_t^2\tp\cdot \dd_t^{-1} P_{\lesssim k}Q_{\sim j}\dd_tg(\tp)]\|_{2,2}\\
        &\lesssim 2^{-3j}\sum_{l\gg k}\|P_{\sim k}Q_l\dd_t^2\tp\|_{2,2}\|P_{\lesssim k}\dd_tg(\tp)\|_{\infty,\infty}\\
        &\lesssim_{\tpxst} 2^{(1-s'-\tht')(k-j)}2^{-(2+s'+\tht')j}
    \end{align*}
    For the remaining part of \ref{B1bgk} with $l\sim k$ we have
    \begin{align}
        2^{-2j}\|P_kQ_j[P_{\gtrsim k}Q_{\sim k}\dd_t^2\tp\cdot g'(\tp)]\|_{2,2}
        &\lesssim2^{-2j}\|P_kQ_j[P_{\sim k}Q_{\sim k}\dd_t^2\tp\cdot P_{\lesssim k}Q_{\sim j}g'(\tp)]\|_{2,2}\label{ah1}\\
        &\quad+2^{-2j}\sum_{r\gg k}\|P_kQ_j[P_rQ_{\sim k}\dd_t^2\tp\cdot P_{\sim r}g'(\tp)]\|_{2,2}\label{ah2}
    \end{align}
    where
    \begin{align*}
        \eqref{ah1}&\lesssim 2^{-2j}\|P_kQ_j[P_{\sim k}Q_{\sim k}\dd_t^2\tp\cdot\dd_t^{-1}P_{\lesssim k}Q_{\sim j}\dd_t g'(\tp)]\|_{2,2}\\
        &\lesssim 2^{-2j}\|P_{\sim k}Q_{\sim k}\dd_t^2\tp\|_{2,2}\cdot 2^{-j}\|P_{\lesssim k}\dd_tg'(\tp)\|_{\infty,\infty}\\
        &\lesssim 2^{(3-s-\tht-s')(k-j)}2^{-(s+\tht)j}2^{-s'j}
    \end{align*}
    as required. For \eqref{ah2} we have
    \begin{align*}
        \eqref{ah2}
        &\lesssim 2^{-2j}\sum_{r\geq j-10}\|P_kQ_j[P_rQ_{\sim k}\dd_t^2\tp\cdot P_{\sim r} g'(\tp)]\|_{2,2}\\
        &\quad+2^{-2j}\sum_{r=k+10}^{j-10}\|P_kQ_j[P_rQ_{\sim k}\dd_t^2\tp\cdot\dd_t^{-1}P_{\sim r}Q_{\sim j}\dd_t g'(\tp)]\|_{2,2}\\
        &\lesssim 2^{-2j}\sum_{r\gtrsim j}2^{3k/2}\|P_rQ_{\sim k}\dd_t^2\tp\|_{2,2}\|P_{\sim r}g'(\tp)\|_{\infty,2}\\
        &\quad+2^{-2j}\sum_{r\gg k}2^{3k/2}\|P_rQ_{\sim k}\dd_t^2\tp\|_{2,2}\cdot 2^{-j}\|\dd_tP_{\sim r}g'(\tp)\|_{\infty,2}\\
        &\lesssim_{\tpxst} 2^{-2j}2^{3k/2}\sum_{r\gtrsim j}2^{-\tht k}2^{(2-s)r}2^{-sr}+2^{-3j}2^{3k/2}\sum_{r\gg k}2^{-\tht k}2^{(2-s)r}2^{(1-s)r}\\
        &\lesssim_{\tpxst} 2^{(s-\tht)(k-j)}2^{-(s+\tht)j}+2^{(3-s-\tht)(k-j)}2^{-(s+\tht)j}
    \end{align*}
    which is acceptable. This completes the work on (B1).

    (B2) can be treated similarly and this complete the study of $j\gg k$.
    \bigskip
    
    \item\underline{$j\simeq k$:} This time we have to show
    $$
    \|P_kQ_j g(\tp)\|_{2,2}\lesssim2^{-(s+\tht)k}
    $$
    We have
    \begin{align*}
        \|P_kQ_{\sim k} g(\tp)\|_{2,2}&\lesssim2^{-k}\|P_kQ_{\sim k}(\na\tp_{\gtrsim k}\cdot g'(\tp))\|_{2,2}+\tk\|P_kQ_{\sim k}(\na\tp_{\ll k}\cdot P_{\sim k}g'(\tp))\|_{2,2}\\
        &\lesssim\tk\|Q_{\gtrsim k}\na\tp_{\gtrsim k}\|_{2,2}+\tk\|P_kQ_{\sim k}(Q_{\ll k}\na\tp_{\gtrsim k}\cdot g'(\tp))\|_{2,2}\\
        &\quad+\tk\|\na\tp_{\ll k}\|_{\f{2M}{M-1},2M}\|P_{\sim k}g'(\tp)\|_{2M,\f{2M}{M-1}}\\
        &\lesssim\tk2^{(1-s-\tht)k}+\tk\|P_kQ_{\sim k}(Q_{\ll k}\na\tp_{\gtrsim k}\cdot g'(\tp))\|_{2,2}\\
        &\quad+\tk2^{(\half-\f{1}{M}-s')k}2^{-(\f{3}{2}-\f{1}{M}+s')k}
    \end{align*}
    The first and third terms here are as required, so it remains to study
    \begin{align}
        \tk\|P_kQ_{\sim k}(Q_{\ll k}\na\tp_{\gtrsim k}\cdot g'(\tp))\|_{2,2}&\lesssim\tk\|P_kQ_{\sim k}(Q_{\ll k}\na\tp_{\sim k}\cdot P_{\lesssim k}g'(\tp))\|_{2,2}\label{yup}\\
        &\quad+\sum_{r\gg k}\tk\|P_kQ_{\sim k}(Q_{\ll k}\na\tp_{ r}\cdot P_{\sim r}g'(\tp))\|_{2,2}\label{yup2}
    \end{align}
    where
    \begin{align*}
        \eqref{yup2}\lesssim\tk\sum_{r\gg k}\|\na\tp_r\|_{\f{2M}{M-1},2M}\|P_{\sim r}g'(\tp)\|_{2M,\f{2M}{M-1}}\lesssim\tk2^{(\half-\f{1}{M}-s')k}2^{-(\f{3}{2}-\f{1}{M}+s')k}
    \end{align*}
    is fine, and
    \begin{align*}
        \eqref{yup}&\lesssim\tk\|P_kQ_{\sim k}[Q_{\ll k}\na\tp_{\sim k}\cdot P_{\ll k}Q_{\sim k}g'(\tp)]\|_{2,2}\\
        &\quad+\tk\|P_kQ_{\sim k}[Q_{\ll k}\na\tp_{\sim k}\cdot P_{\sim k}g'(\tp)]\|_{2,2}\\
        &\lesssim\tk\|Q_{\ll k}\na\tp_{\sim k}\cdot \dd_t^{-1} P_{\ll k}Q_{\sim k}(\dd_t\tp\  g''(\tp))\|_{2,2}\\
        &\quad+\tk\|Q_{\ll k}\na\tp_{\sim k}\cdot \na^{-1}P_{\sim k}Q_{\sim k}(\na\tp\  g''(\tp))\|_{2,2}\\
        &\lesssim2^{-2k}(\|\na\tp_{\sim k}\|_{2M,\f{2M}{M-1}}\|\dtpt_{\lesssim k}\|_{\f{2M}{M-1},2M}+\|\na\tp_{\sim k}\|_{\f{2M}{M-1},2M}\|\dtpt_{\gg k}\|_{2M,\f{2M}{M-1}})\\
        &\quad+2^{-2k}(\|\na\tp_{\sim k}\|_{2M,\f{2M}{M-1}}\|\na\tp_{\lesssim k}\|_{\f{2M}{M-1},2M}+\|\na\tp_{\sim k}\|_{\f{2M}{M-1},2M}\|\na\tp_{\gg k}\|_{2M,\f{2M}{M-1}})\\
        &\lesssim 2^{-(2+2s')k}
    \end{align*}

    \bigskip
    
    \item\underline{$j\ll k$:} This time our goal is
    $$
    \|P_kQ_j g(\tp)\|_{2,2}\lesssim2^{-sk-\tht j}
    $$
    We have
    \begin{align}
        \|P_kQ_jg(\tp)\|_{2,2}&\lesssim\tk\|P_kQ_j(\na\tp_{\ll k}\cdot P_{\sim k}g'(\tp))\|_{2,2}\label{AAAAAa}\\
        &\quad+\tk\|P_kQ_j(\na\tp_{\sim k}\cdot g'(\tp))\|_{2,2}\label{BBBBBb}\\
        &\quad+\tk\sum_{r\gg k}\|P_kQ_j(\na\tp_r\cdot P_{\sim r}g'(\tp))\|_{2,2}\label{CCCCCc}
    \end{align}
    Here \eqref{AAAAAa} and \eqref{CCCCCc} can be handled as in the case $j\simeq k$. For \eqref{BBBBBb} we separate
    \begin{align*}
        \eqref{BBBBBb}\lesssim\tk\|P_kQ_j(Q_{\ll j}\na\tp_{\sim k}\cdot g'(\tp))\|_{2,2}+\tk\|P_kQ_j(Q_{\gtrsim j}\na\tp_{\sim k}\cdot g'(\tp))\|_{2,2}
    \end{align*}
    The second line here is straightforward to handle by placing $\na\tp$ into $L^2_tL^2_x$, so we consider only the first term. Referring to the result for $j\gg k$ to handle $g'$ we find
    \begin{align*}
        \tk\|P_kQ_j(Q_{\ll j}\na\tp_{\sim k}\cdot g'(\tp))\|_{2,2}
        &\lesssim\tk\|P_kQ_j(Q_{\ll j}\na\tp_{\sim k}\cdot Q_{\gtrsim j}P_{\ll j}g'(\tp))\|_{2,2}\\
        &\quad+\tk\|P_kQ_j(Q_{\ll j}\na\tp_{\sim k}\cdot P_{\gtrsim j}g'(\tp))\|_{2,2}\\
        &\lesssim\tk\|Q_{\ll j}\na\tp_{\sim k}\|_{\infty,2}\sum_{\substack{l\gtrsim j\\r\ll j}}\| Q_lP_rg'(\tp))\|_{2,\infty}\\
        &\quad+\tk\|Q_{\ll j}\na\tp_{\sim k}\|_{2M,2}\sum_{r\gtrsim j}2^{-r}\|P_r(\na\tp\ g''(\tp))\|_{\f{2M}{M-1},\infty}\\
        &\lesssim2^{-sk-\tht j}
    \end{align*}
    for $s'-\f{1}{2M}\geq\tht'$. This completes the proof of the Moser estimate.
    \end{itemize}
proof of projection estimate
\end{proof}

\section{Proof of Lemma \ref{big_lem}.}\label{appendixB}
The two parts of this Lemma are proved similarly, so we restrict to proving the second statement (which is somewhat simpler due to symmetry reductions). 
comment of first statement
\begin{proof}[Proof of \eqref{big_lem2}.]
In this proof we will constantly use the Strichartz estimate Lemma \ref{str_lem} and the modulation-Bernstein estimate \eqref{MB}.

Assume without loss of generality $\kt\geq\kth$. First suppose $\kt\geq\kth+10$ so the whole term is at frequency $\sim 2^{\kt}$. We split into the following cases:
\begin{itemize}
\item $\sum_{l\lesssim\kth}P_{\sim\kt}Q_l(Q_{\gtrsim l}\vptk\cdot\vpthk)$: Square summing over $l$ we have
\begin{align*}
\|\sum_{l\lesssim\kth}P_{\sim\kt}Q_l(Q_{\gtrsim l}\vptk\cdot\vpthk)\|_{\xst}
&\lesssim\left(\sum_{l\lesssim\kth}(2^{s\kt}2^{\tht l}\|Q_{\gtrsim l}\vptk\|_{2,2}\|\vpthk\|_{\infty,\infty})^2\right)^\half\\
&\lesssim\left(\sum_{l\lesssim\kth}(\sum_{j\gtrsim l}2^{\tht (l-j)}\|Q_j\vptk\|_{\xst}\cdot 2^{-s'\kth}\|\vpthk\|_{\xst})^2\right)^\half\\
&\lesssim 2^{-s'\kth}\|\vptk\|_{\xst}\|\vpthk\|_{\xst}
\end{align*}
where we used Cauchy-Schwarz for the final inequality.
\item $\sum_{l\lesssim\kth}P_{\sim\kt}Q_l(Q_{\ll l}\vptk\cdot\vpthk)$: Using \eqref{MB} on $\vp^{(2)}$ this term can be bounded in $\xst$ by
\begin{align*}
&\sum_{l\lesssim\kth}\sum_{j\ll l}2^{s\kt}2^{\tht l}\|Q_j\vptk\|_{M,2}\|\vpthk\|_{\f{2M}{M-2},\infty}\\
&\lesssim\sum_{l\lesssim\kth}\sum_{j\ll l}2^{\tht l}2^{(\f{1}{2}-\f{1}{M})j}2^{-\tht j}\|\vptk\|_{\xst}2^{-(\half-\f{1}{M}+s')\kth}\|\vpthk\|_{\xst}\\
&\lesssim2^{(\tht'+\f{1}{M}-s')\kth}\|\vptk\|_{\xst}\|\vpthk\|_{\xst}
\end{align*}
which is acceptable upon choosing $1/M<s'-\tht'$.
\end{itemize}
When the outer modulation is $\gg 2^{\kth}$, one of the inner terms must be of at least comparable modulation or the interaction is null. 
\begin{itemize}
\item $\sum_{\kth\ll l\lesssim\kt}P_{\sim\kt}Q_l(Q_{>l-10}\vptk\cdot\vpthk)$: In this case we use the bound
\begin{align*}
\|P_{\sim\kt}Q_l(Q_{>l-10}\vptk\cdot\vpthk)\|_{\xst}&\lesssim\sum_{j\gtrsim l}2^{s\kt}2^{\tht l}\|Q_j\vptk\|_{2,2}\|\vpthk\|_{\infty,\infty}\\
&\lesssim2^{-s'\kth}\sum_{j\gtrsim l}2^{\tht (l-j)}\|Q_j\vptk\|_{\xst}\|\vpthk\|_{\xst}
\end{align*}
which is again acceptable when square-summed in $l$.

\item $\sum_{\kth\ll l\lesssim\kt}P_{\sim\kt}Q_l(Q_{\ll l}\vptk\cdot\vpthk)$: This time we bound
\begin{align*}
\|P_{\sim\kt}Q_l(Q_{\ll l}\vptk\cdot\vpthk)\|_{\xst}&\lesssim2^{s\kt}2^{\tht l}\|Q_{\ll l}\vptk\|_{\infty,2}\|Q_{\gtrsim l}\vpthk\|_{2,\infty}\\
&\lesssim\|\vptk\|_{\xst}\sum_{j\gtrsim l}2^{-s'\kth}2^{\tht (l-j)}\|Q_j\vpthk\|_{\xst}
\end{align*}
which is again acceptable.
\end{itemize}
When the outer modulation is very large $\gg2^{\kt}$, we have a similar situation:
\begin{itemize}
\item $\sum_{l\gg\kt}P_{\sim\kt}Q_l(Q_{>l-10}\vptk\cdot\vpthk)$:
\begin{align*}
\|P_{\sim\kt}Q_l(Q_{>l-10}\vptk\cdot\vpthk)\|_{\xst}&\lesssim\sum_{j\gtrsim l}2^{sl}2^{\tht l}\|Q_j\vptk\|_{2,2}\|\vpthk\|_{\infty,\infty}\\
&\lesssim2^{-s'\kth}\sum_{j\gtrsim l}2^{(s+\tht) (l-j)}\|Q_j\vptk\|_{\xst}\|\vpthk\|_{\xst}
\end{align*}
which is acceptable.
\bigskip

\item $\sum_{l\gg\kt}P_{\sim\kt}Q_l(Q_{\ll l}\vptk\cdot\vpthk)$: Here
\begin{align*}
\|P_{\sim\kt}Q_l(Q_{\ll l}\vptk\cdot\vpthk)\|_{\xst}&\lesssim2^{sl}2^{\tht l}\|Q_{\ll l}\vptk\|_{\infty,\infty}\|Q_{\gtrsim l}\vpthk\|_{2,2}\\
&\lesssim2^{-s'\kt}\sum_{j\gtrsim l}2^{(s+\tht) (l-j)}\|\vptk\|_{\xst}\|Q_j\vpthk\|_{\xst}
\end{align*}
\end{itemize}

It remains to consider the case $|\kt-\kth|\leq10$. We first suppose that the outer modulation is restricted to $\lesssim 2^{\kt}$:
\begin{align*}
\|P_{\lesssim\kt}Q_{\lesssim\kt}(\vptk\cdot\vpthk)\|_{\xst}
&\lesssim2^{(s+\tht)\kt}\|\vptk\|_{\f{2M}{M-2},M}\|\vpthk\|_{M,\f{2M}{M-2}}\\
&\lesssim2^{(s+\tht)\kt}2^{-(\half+\f{2}{M}+s')\kt}\|\vptk\|_{\xst}\cdot 2^{-(\f{3}{2}-\f{2}{M}+s')\kth}\|\vpthk\|_{\xst}
\end{align*}
which is acceptable for $s'>\tht'$.

When the outer modulation is at $2^l\gg 2^{\kt}$, at least one of the terms must also be at modulation at least $\sim 2^l$ or else the term is null. Since we are considering $\kt\sim\kth$ we may assume WLOG it is the factor $\vptk$:
\begin{align*}
\sum_{l\gg\kt}\|P_{\lesssim\kt}Q_l(Q_{\gtrsim l}\vptk\cdot\vpthk)\|_{\xst}
&\lesssim\sum_{l\gg\kt}2^{(s+\tht)l}\|Q_{\gtrsim l}\vptk\|_{2,2}\|\vpthk\|_{\infty,\infty}\\
&\lesssim 2^{-s'\kth}\sum_{l\gg\kt}\sum_{j\gtrsim l}2^{(s+\tht)(l-j)}\|Q_j\vptk\|_{\xst}\|\vpthk\|_{\xst}
\end{align*} 
This was the final case and the proof is complete.
\end{proof}


\bibliographystyle{plain}
\bibliography{refs} 

\begin{thebibliography}{10}

\bibitem{berntson2020multi}
Bjorn Berntson, Rob Klabbers, and Edwin Langmann.
\newblock Multi-solitons of the half-wave maps equation and {C}alogero--{M}oser spin--pole dynamics.
\newblock {\em Journal of Physics A: Mathematical and Theoretical}, 53(50):505702, 2020.

\bibitem{burzio}
Stefano Burzio.
\newblock {\em On long time behaviour of solutions to nonlinear dispersive equations}.
\newblock PhD thesis, \'Ecole Polytechnique F\'ed\'erale de Lausanne, 2020.

\bibitem{spheres_book}
Feng Dai and Yuan Xu.
\newblock {\em Approximation theory and harmonic analysis on spheres and balls}.
\newblock Springer, 2013.

\bibitem{eyeson2022uniqueness}
Eugene Eyeson, Silvino~Reyes Farina, and Armin Schikorra.
\newblock On uniqueness for half-wave maps in dimension $d\geq 3$.
\newblock {\em arXiv preprint arXiv:2211.05652}, 2022.

\bibitem{geba_grillakis}
Dan-Andrei Geba and Manoussos~G Grillakis.
\newblock {\em An introduction to the theory of wave maps and related geometric problems}.
\newblock World Scientific Publishing Company, 2016.

\bibitem{LG}
Patrick G{\'e}rard and Enno Lenzmann.
\newblock A {L}ax pair structure for the half-wave maps equation.
\newblock {\em Letters in Mathematical Physics}, 108:1635--1648, 2018.

\bibitem{hirayama2024sharp}
Hiroyuki Hirayama, Shinya Kinoshita, and Mamoru Okamoto.
\newblock Sharp well-posedness for the cauchy problem of the two dimensional quadratic nonlinear schr{\"o}dinger equation with angular regularity.
\newblock {\em Journal of Differential Equations}, 395:181--222, 2024.

\bibitem{hong2022improved}
Seokchang Hong.
\newblock Improved multilinear estimates and global regularity for general nonlinear wave equations in $(1+ 3) $ dimensions.
\newblock {\em arXiv preprint arXiv:2207.02412}, 2022.

\bibitem{KK}
Anna Kiesenhofer and Joachim Krieger.
\newblock Small data global regularity for half-wave maps in $n=4$ dimensions.
\newblock {\em Communications in Partial Differential Equations}, 46(12):2305--2324, 2021.

\bibitem{klainerman1996estimates}
Sergeiu Klainerman and Matei Machedon.
\newblock Estimates for null forms and the spaces ${H}^{s,\delta}$.
\newblock {\em International Mathematics Research Notices}, 1996(17):853--865, 1996.

\bibitem{klainerman}
Sergiu Klainerman.
\newblock Uniform decay estimates and the {L}orentz invariance of the classical wave equation.
\newblock {\em Communications in Pure Applied Mathematics}, 38:321--332, 1985.

\bibitem{KlSe}
Sergiu Klainerman and Sigmund Selberg.
\newblock Remark on the optimal regularity for equations of wave maps type.
\newblock {\em Communications in Partial Differential Equations}, 22(5-6):99--133, 1997.

\bibitem{KS}
Joachim Krieger and Yannick Sire.
\newblock Small data global regularity for half-wave maps.
\newblock {\em Analysis \& {PDE}}, 11(3):661--682, 2017.

\bibitem{primer}
Enno Lenzmann.
\newblock A short primer on the half-wave maps equation.
\newblock {\em Journ{\'e}es {\'e}quations aux d{\'e}riv{\'e}es partielles}, pages 1--12, 2018.

\bibitem{lenzmann}
Enno Lenzmann and Armin Schikorra.
\newblock On energy-critical half-wave maps into $\mathbb{S}^2$.
\newblock {\em Inventiones mathematicae}, 213:1--82, 2018.

\bibitem{lenzmannCM}
Enno Lenzmann and J{\'e}r{\'e}my Sok.
\newblock Derivation of the half-wave maps equation from {C}alogero--{M}oser spin systems.
\newblock {\em arXiv preprint arXiv:2007.15323}, 2020.

\bibitem{yang}
Yang Liu.
\newblock Global well-posedness for half-wave maps with $\mathbb{S}^2$ and $\mathbb{H}^2$ targets for small smooth initial data.
\newblock {\em arXiv preprint arXiv:2109.13657}, 2021.

\bibitem{yang2}
Yang Liu.
\newblock Global weak solutions for the half-wave maps equation in $\mathbb{R}$.
\newblock {\em arXiv preprint arXiv:2308.06836}, 2023.

\bibitem{matsuno2022integrability}
Yoshimasa Matsuno.
\newblock Integrability, conservation laws and solitons of a many-body dynamical system associated with the half-wave maps equation.
\newblock {\em Physica D: Nonlinear Phenomena}, 430:133080, 2022.

\bibitem{ohlmann2023analytic}
Gaspard Ohlmann.
\newblock Analytic proofs of well-posedness and non-turbulence results for the half-wave maps equation.
\newblock {\em arXiv preprint arXiv:2310.13442}, 2023.

\bibitem{selberg1999multilinear}
Sigmund Selberg.
\newblock {\em Multilinear space-time estimates and applications to local existence theory for nonlinear wave equations}.
\newblock PhD thesis, 1999.

\bibitem{stein1970singular}
Elias~M Stein.
\newblock {\em Singular integrals and differentiability properties of functions}.
\newblock Princeton university press, 1970.

\bibitem{stein_book2}
Elias~M Stein.
\newblock {\em Topics in harmonic analysis, related to the {L}ittlewood-{P}aley theory}.
\newblock Princeton University Press, 1970.

\bibitem{Sterbenz}
Jacob Sterbenz.
\newblock Angular regularity and {S}trichartz estimates for the wave equation.
\newblock {\em International Mathematics Research Notices}, 2005(4):187--231, 2005.

\bibitem{SterbenzYM}
Jacob Sterbenz.
\newblock Global regularity and scattering for general non-linear wave equations {II}. $(4+1)$ dimensional {Y}ang-{M}ills equations in the {L}orentz gauge.
\newblock {\em American journal of mathematics}, 129(3):611--664, 2007.

\bibitem{S_multipliers}
Robert Strichartz.
\newblock Multipliers for spherical harmonic expansions.
\newblock {\em Transactions of the American Mathematical Society}, 167:115--124, 1972.

\bibitem{tao1998counterexamples}
Terence Tao.
\newblock Counterexamples to the $n=3$ endpoint strichartz estimate for the wave equation.
\newblock {\em preprint}, 1998.

\bibitem{Tao_1}
Terence Tao.
\newblock Global regularity of wave maps, {I}: small critical sobolev norm in high dimension.
\newblock {\em International Mathematics Research Notices}, 2001(6):299--328, 2001.

\bibitem{tao_book}
Terence Tao.
\newblock {\em Nonlinear dispersive equations: local and global analysis}.
\newblock Number 106. American Mathematical Soc., 2006.

\bibitem{tataru1}
Daniel Tataru.
\newblock Local and global results for wave maps {I}.
\newblock {\em Communications in Partial Differential Equations}, 23(9-10):1781--1793, 1998.

\bibitem{tataru}
Daniel Tataru.
\newblock On global existence and scattering for the wave maps equation.
\newblock {\em American Journal of Mathematics}, 123(1):37--77, 2001.

\bibitem{CM_derivation2}
Tianci Zhou and Michael Stone.
\newblock Solitons in a continuous classical {H}aldane--{S}hastry spin chain.
\newblock {\em Physics Letters A}, 379(43-44):2817--2825, 2015.

\end{thebibliography}

  \bigskip
  \footnotesize
\textsc{B\^{a}timent des Math\'{e}matiques, EPFL, Station 8, 1015 Lausanne, Switzerland}\par\nopagebreak
\textit{E-mail address}: \texttt{katie.marsden@epfl.ch}\par\nopagebreak

\end{document}